\documentclass[11pt]{amsart}
\usepackage{soul}
\usepackage{color,array,amssymb,amsmath, amsthm, amsfonts, tikz-cd}

\usepackage{graphicx}
\usepackage{listings}
\usepackage[margin=1in]{geometry}
\usepackage{lstautogobble}
\usepackage{enumerate}
\usepackage{thmtools}
\usepackage{thm-restate}
\usepackage{amsthm}
\usepackage{verbatim}

\usepackage{mathtools}
\usepackage{physics}
\usepackage{color}
\usepackage{hyperref}
\usepackage[capitalise]{cleveref}
\usepackage[bottom]{footmisc}
\crefformat{equation}{(#2#1#3)}

\usepackage{float}
\restylefloat{table}

\usepackage{mathrsfs}

\theoremstyle{plain}
\newtheorem{thm}{Theorem}[section]
\newtheorem{prop}[thm]{Proposition}
\newtheorem{lemma}[thm]{Lemma}

\theoremstyle{plain}
\newtheorem{mthm}{Theorem}

\theoremstyle{definition}

\newtheorem{remark}[thm]{Remark}

\Crefname{prop}{Proposition}{Propositions}
\Crefname{mthm}{Theorem}{Theorems}

\numberwithin{equation}{section} %Equation numbering

\DeclarePairedDelimiter\ipp{\langle}{\rangle}

\DeclarePairedDelimiter{\paren}{\lparen}{\rparen}

\DeclareMathOperator{\supp}{supp}

\newcommand{\p}{{\partial}}

\newcommand{\cb}{\color{blue}}
\newcommand{\cre}{\color{red}}

\renewcommand{\d}{\delta}

\newcommand{\R}{{\mathbb{R}}}

\newcommand{\N}{{\mathbb{N}}}

\newcommand{\Z}{{\mathbb{Z}}}

\newcommand{\T}{{\mathbb{T}}}
\newcommand{\g}{{\mathsf{g}}}
\newcommand{\hs}{\mathsf{h}}

\newcommand{\Eb}{\mathbb{E}}

\newcommand{\tl}{\tilde}

\newcommand{\D}{\Delta}

\newcommand{\ph}{\phantom{=}}
\newcommand{\nn}{\nonumber}

\newcommand{\ol}{\overline}

\newcommand{\ux}{X}

\newcommand{\XN}{X_N}

\newcommand{\ep}{\varepsilon}
\newcommand{\al}{\alpha}
\newcommand{\be}{\beta}
\newcommand{\ka}{\kappa}
\newcommand{\la}{\lambda}

\newcommand{\K}{\mathsf{K}}

\newcommand{\W}{{\mathbf{W}}}
\newcommand{\E}{{\mathbb{E}}}
\newcommand{\Tc}{\mathcal{T}}

\newcommand{\Dm}{|\nabla|}

\newcommand{\Fs}{\mathsf{F}}

\newcommand{\Fr}{{F}}
\newcommand{\Ec}{\mathcal{E}}
\renewcommand{\P}{\mathcal{P}}
\newcommand{\Te}{\mathrm{Term}}

\newcommand{\bmu}{\bar\mu}

\newcommand{\cdd}{\mathsf{c}_\ds}
\newcommand{\ga}{\gamma}
\newcommand{\Cc}{\mathcal{C}}
\newcommand{\Dc}{\mathcal{D}}
\newcommand{\muu}{\mu_{\mathrm{unif}}}
\newcommand{\mua}{\mu_{\mathrm{app}}}
\newcommand{\mum}{\mu_{\min}}
\newcommand{\Ra}{R_{\text{app}}}

\newcommand{\bec}{\beta_{\mathrm{c}}}
\newcommand{\bels}{\beta_{\mathrm{s}}}
\newcommand{\beu}{\beta_{\mathrm{u}}}

\newcommand{\bei}{\beta_{\mathrm{i}}}
\newcommand{\ds}{\mathsf{d}}
\newcommand{\Es}{\mathsf{E}}

\let\div\relax
\DeclareMathOperator{\div}{div}

\def\Xint#1{\mathchoice
{\XXint\displaystyle\textstyle{#1}}%
{\XXint\textstyle\scriptstyle{#1}}%
{\XXint\scriptstyle\scriptscriptstyle{#1}}%
{\XXint\scriptscriptstyle\scriptscriptstyle{#1}}%
\!\int}
\def\XXint#1#2#3{{\setbox0=\hbox{$#1{#2#3}{\int}$ }
\vcenter{\hbox{$#2#3$ }}\kern-.6\wd0}}

\def\dashint{\Xint-}

\def \hal{\frac{1}{2}}
\def\({\left(}
\def\){\right)}
\def \ep{\varepsilon}
\def\nab{\nabla}
\def\indic{\mathbf{1}}

\title[The attractive log gas]{The attractive log gas: stability, uniqueness, and propagation of chaos}

\author[A. Chodron de Courcel]{Antonin Chodron de Courcel}
\address{Antonin Chodron de Courcel, Institut des Hautes Etudes Scientifiques, Bures-sur-Yvette}
\thanks{~A. C. is supported by the Fondation CFM pour la Recherche.}
\email{decourcel@ihes.fr}
\author[~M. Rosenzweig]{Matthew Rosenzweig}
\address{Matthew Rosenzweig, Carnegie Mellon University, Department of Mathematical Sciences, Pittsburgh, PA} 
\email{mrosenz2@andrew.cmu.edu}
\thanks{~M. R. is supported by the Simons Foundation through the Simons Collaboration on Wave Turbulence and by NSF grants DMS-2052651, DMS-2206085.}
\author[~S. Serfaty]{Sylvia Serfaty}
\address{Sylvia Serfaty, Courant Institute of Mathematical Sciences, New York University, New York City, NY}
\email{serfaty@cims.nyu.edu}
\thanks{~S. S. is supported by the Simons Foundation through a Simons Investigator award and by NSF grant DMS-2247846.}

\begin{document}
\begin{abstract}
We consider overdamped Langevin dynamics for the attractive log gas on the torus $\T^\ds$, for $\ds\geq 1$. In dimension $\ds=2$, this model coincides with a periodic version of the parabolic-elliptic Patlak-Keller-Segel model of chemotaxis. The attractive log gas (for our choice of units) is well-known to have a critical inverse temperature $\bec={2\ds}$ corresponding to when the free energy is bounded from below. Moreover, it is well-known that the uniform distribution is always a  stationary state regardless of the temperature. We identify another temperature threshold $\bels$ sharply corresponding to the nonlinear stability of the uniform distribution. We show that for $\beta>\bels$, the uniform distribution does not minimize the free energy and moreover is nonlinearly unstable, while for $\beta<\bels$, it is stable. We also show that there exists $\beu$ for which uniqueness of equilibria holds for $\be<\beu$.

Related to the above findings, we establish a uniform-in-time rate for entropic propagation of chaos for a range of $\beta<\bels$. To our knowledge, this is the first such result for singular attractive interactions and affirmatively answers a question of Bresch et al. \cite{BJW2020}.
%\st{on the possibility of such a result}. 
The proof of the convergence is through the modulated free energy method, in particular relying on a \emph{modulated logarithmic Hardy-Littlewood-Sobolev (mLHLS) inequality}. Unlike \cite{BJW2020}, we show that such an inequality holds without truncation of the potential---the avoidance of the truncation being essential to a uniform-in-time result---at sufficiently high temperature and provide a counterexample to the mLHLS inequality when $\beta>\bels$. As a byproduct, we show that it is impossible to have a uniform-in-time rate of propagation of chaos if $\beta>\bels$. 
\end{abstract}
\maketitle

\section{Introduction}\label{sec:intro}
%\subsection{Background}\label{ssec:introback}
The microscopic \emph{attractive log gas} for $N$ particles at temperature $\frac1\beta\geq 0$ on the flat $\ds$-dimensional torus $\T^\ds$ for $\ds\geq 1$,\footnote{The motivation for considering $\T^\ds$, as opposed to Euclidean space $\R^\ds$, is explained below in \cref{rem:TvsR}.} which we identify with $[-\frac12,\frac12)^\ds$ under periodic boundary conditions, is described by the free energy
\begin{equation}\label{eq:mFE}
\frac1\beta\int_{(\T^\ds)^N}\log(f_N)df_N - \frac{1}{2N}\sum_{1\leq i\neq j\leq N}\int_{(\T^\ds)^N}\g(x_i-x_j)df_N,
\end{equation}
where $f_N$ is a probability density\footnote{We abuse notation here and throughout this article by using the same symbol for both the measure $df_N=f_Ndx_1\cdots dx_N$ and the density $f_N$.} on $(\T^\ds)^N$ such that the integrals make sense and $\g$ is the solution of
\begin{equation}\label{eq:gdefa}
\Dm^\ds \g = \cdd(\d_0-1), \qquad \cdd \coloneqq \frac{\Gamma(\ds/2) (4\pi)^{\ds/2}}{2}.
\end{equation}
Here, $\Dm \coloneqq (-\D)^{\frac{1}{2}}$ is the Fourier multiplier with symbol $2\pi |k|$ and $\Gamma$ is the usual gamma function. This definition of $\g$ is such that it is the zero-mean periodization of $-\log|x|$ on $\R^\ds$ and behaves like $-\log|x|$ as $|x|\rightarrow0$.   We note that if $\ds=2$, then $\frac{1}{\cdd}\g$ coincides with the periodic \emph{Coulomb} potential. 
%\st{The choice of sign $\ep$ corresponds to the whether the interaction is \emph{repulsive} ($\ep=1$) or \emph{attractive} ($\ep=-1$). We are interested in the attractive case, and {therefore} set $\ep=-1$ for the remainder of the paper.} 
The attractiveness  is reflected in the $-$ sign in front of the interaction energy term in \eqref{eq:mFE}. The attractive log gas is interesting because the logarithmic interaction is the most singular one for which the attraction can be balanced  by the  effectively repulsive logarithmic effect created by the entropic first  term in \eqref{eq:mFE}.
%, and the least singular one for which the two effects compete. 
In contrast, 
gases with attractions that are stronger than logarithmic do not even have finite free energy.

The \emph{overdamped Langevin dynamics} of the free energy \eqref{eq:mFE} is the following system of SDEs:
\begin{equation}\label{eq:SDEa}
\begin{cases}
dx_{i}^t = \displaystyle\frac{1}{N}\sum_{1\leq j\leq N : j\neq i} \nabla\g(x_i^t-x_j^t)dt + \sqrt{2/\beta}dW_i^t\\
x_i^t|_{t=0} = x_i^0
\end{cases}\qquad i\in\{1,\ldots,N\}.
\end{equation}
Above, $x_i^0\in\T^\ds$ are the pairwise distinct initial positions and $\{W_i\}_{i=1}^N$ are independent standard Brownian motions in $\T^\ds$. The potential $-\g$ is attractive, so that for $\beta=\infty$, the energy $\sum_{1\leq i\neq j\leq N}\g(x_i^t-x_j^t)$ is increasing for $t\in [0,\infty)$. Making sense of the system \eqref{eq:SDEa} is difficult due to collisions between particles with nonzero probability (e.g., see \cite{FJ2017, CP2016ks, FT2021ksc}), reflecting the attractive nature of the interaction. Therefore, we shall instead take as our starting point the associated \emph{Liouville/forward Kolmogorov equation}
\begin{equation}\label{eq:Lioua}
\p_t f_N = -\sum_{i=1}^N \div_{x_i}\paren*{f_N\frac{1}{N}\sum_{1\leq i\neq j\leq N}\nabla\g(x_i-x_j)} + \frac1\beta\sum_{i=1}^N\D_{x_i}f_N.
\end{equation}
Equation \eqref{eq:Lioua}, which follows from \eqref{eq:SDEa} through It\^o's formula, describes the evolution of the \emph{law} $f_N^t$ of the solution $\ux_N^t \coloneqq (x_1^t,\ldots,x_N^t) \in (\T^\ds)^N$  to \eqref{eq:SDEa} starting from an initial distribution $f_N^0$.

The choice of $1/N$ scaling of the potential corresponds to the \emph{mean-field regime}. In principle, the large $N$ behavior of the system \eqref{eq:SDEa} is effectively described by solutions of the PDE
\begin{equation}\label{eq:lima}
\begin{cases}
\p_t\mu^t = -\div(\mu^t\nabla\g\ast\mu^t) +\frac1\beta\D\mu^t  \\
\mu^t|_{t=0} = \mu^0.
\end{cases}
\end{equation}
This equation is formally the gradient flow of the macroscopic free energy
\begin{equation}\label{eq:introFE}
\Ec_\beta(\mu) \coloneqq \frac1\beta\int_{\T^\ds}\log(\mu)d\mu - \frac{1}{2}\int_{(\T^\ds)^2}\g(x-y)d\mu^{\otimes 2}(x,y)
\end{equation}
with respect to the 2-Wasserstein metric $W_2$ on the space $\P(\T^\ds)$ of probability measures on $\T^\ds$. The $\ds=2$ case of equation \eqref{eq:lima} is particularly interesting because when $\beta<\infty$, it corresponds to the well-known \emph{parabolic-elliptic Patlak-Keller-Segel (PKS) equation} with periodic boundary conditions, which is a model for the aggregation of cells by chemotaxis \cite{Patlak1953, KS1970, Nanjundiah1973}. {More generally, the version of equation \eqref{eq:lima} on $\R^\ds$ has been the object of many studies in its own right, as a PDE where the aggregation of mass (leading to possible blowup of solutions) competes with diffusion. Indeed, there is a temperature threshold $\beta_c = 2\ds$ under which diffusion is dominant, and above which blowup occurs. This critical $\beta_c =2\ds$ is precisely the value above which the free energy of \eqref{eq:introFE}, or rather its $\R^\ds$ version, ceases to be bounded from below. It is  worth noting that this value $\beta_c= 2\ds$ is also the power for which $\exp\( -\frac{ \beta}{2} \log |x-y|\)$ ceases to be integrable near the diagonal of $\R^\ds\times \R^{\ds}$, corresponding to the threshold for convergence of the partition function of the associated attractive log gas.}  

 If $\beta=\infty$, then the equation has been studied as a model for the evolution of vortex densities in superconductors \cite{CRS1996} and superfluids \cite{E1994} and as a model for adhesion dynamics \cite{NPS2001, Poupaud2002}. We also mention the $\ds=1$ case of the famous repulsive $\log$ gas toy model, which has attracted much attention over the years due to its connections with random matrix theory \cite{Forrester2010}. 

The connection between \eqref{eq:SDEa} and \eqref{eq:lima} can be seen in multiple ways. One way is through the {\it empirical measure}, the random probability measure on $\T^\ds$ defined by
\begin{equation}\label{munt}
\mu_N^t\coloneqq \frac1N \sum_{i=1}^N \delta_{x_i^t},
\end{equation}
associated to a solution $\ux_N^t$ of the system \eqref{eq:SDEa} at time $t$. If the points $x_i^0$, which themselves depend on $N$, are such that $\mu_N^0$ converges  to some sufficiently regular density $\mu^0$, then one expects from It\^o's lemma and formal computations that for $t>0$, $\mu_N^t$ converges as $N\rightarrow\infty$ to the solution of \eqref{eq:lima}. Convergence of the empirical measure in law is qualitatively equivalent to proving {\it propagation of molecular chaos}  \cite{HM2014}: if $f_N^0$ is the initial law of the distribution of the $N$ particles in $\T^\ds$ and if $f_N^0$ converges to some factorized law  $(\mu^0)^{\otimes N}$, then  the $k$-point marginals $f_{N;k}^t$ converge for all time to $(\mu^t )^{\otimes k}$. In general, quantitative results for one form of convergence do not \emph{a priori} carry over to the other, and we shall concern ourselves in this paper with rates for propagation of chaos, specifically measured in terms of relative entropy.
 
The topic of mean-field limits for singular interactions has experienced tremendous progress in recent years. As this article is exclusively concerned with the attractive log gas and general mean-field limits have been discussed extensively elsewhere, the reader will understand if we confine our discussion of the literature (in \cref{ssec:introMF} below) to the attractive log case. For a general review of the literature, we refer to the introductions of \cite{Serfaty2020, NRS2021, CdCRS2023} as well as the recent surveys \cite{CD2021, Golse2022ln}. 

\medskip

The main goal of the present paper is to establish \emph{uniform-in-time} propagation of chaos or convergence of solutions of \eqref{eq:SDEa} towards those of \eqref{eq:lima} in the setting of the torus. 
This has been previously studied by Bresch et al. in \cite{BJW2019crm, BJW2020} in the same  torus setting. By introducing a \emph{modulated free energy method}, combination of the relative entropy method of  \cite{JW2016, JW2018} and the modulated energy method of \cite{Serfaty2017,Duerinckx2016,Serfaty2020}, they were able to show local-in-time convergence in relative entropy, conditional on good estimates for the solutions of \eqref{eq:lima}.  Let us briefly define the main quantities here. We will discuss it in more detail in \cref{ssec:introMRrev}.

The  \emph{modulated free energy} defined by 
\begin{equation}\label{def:modulatedfreenrj}
E_N(f_N, \mu) \coloneqq \frac1\beta H_N\left(f_N \vert \mu^{\otimes N}\right) + \Eb_{f_N}\left[F_N(\ux_N, \mu)\right]
\end{equation} 
 is a combination of the normalized relative entropy 
\begin{equation}\label{eq:REdef}
H_N \left(f_N \vert \mu^{\otimes N}\right) \coloneqq \frac{1}{N} \int_{(\T^\ds)^N} \log\paren*{\frac{f_N}{\mu^{\otimes N}}} df_N,
\end{equation}
and the (expectation of the)  modulated energy from \cite{SS2015log, SS2015, RS2016, PS2017, Duerinckx2016, Serfaty2020, NRS2021},
\begin{equation}\label{def:modulatedenergy}
\Fr_N(\ux_N,\mu) \coloneqq \int_{(\T^\ds)^2\setminus\triangle} \g(x-y)d\paren*{\frac{1}{N}\sum_{i=1}^N\d_{x_i} - \mu}^{\otimes 2}(x,y),
\end{equation}
where $\triangle \coloneqq \{(x,x) \in (\T^\ds)^2\}$. Above, the modulated energy is averaged with respect to the joint law $f_N$ of the positions $\ux_N$. The mean-field convergence proof consists in showing a control of the time-derivative of $E_N(f_N^t , \mu^t)$ by $E_N(f_N^t, \mu^t)$ itself and applying  Gr\"onwall's lemma. The convergence in relative entropy is then deduced provided one can absorb the modulated energy into the relative entropy. This is the main part of the analysis, which in \cite{BJW2020} is done by splitting the interaction into a cutoff (regular) part and a  short range singular part. 
Working on the torus is helpful as it provides lower bounds for the density $\mu^t$, which in turn allows to easily control velocity terms appearing in the modulated free energy evolution.
 
  As the rates of convergence these authors obtain  deteriorate exponentially in time, one of the questions  posed in \cite{BJW2020}, which we answer, is the possibility of a uniform-in-time rate, hence global convergence.
  
Proving this requires us to  revisit quite thoroughly the analysis of equation \eqref{eq:lima} and 
establish for the first time a  global-in-time theory and asymptotic behavior of the solutions to this equation on the torus in any dimension, which is the second main accomplishment of our paper. We demonstrate that the behavior on the torus is quite different from that of the full space. First, on the torus  the equation always  has the uniform density $\muu\equiv 1$  as  a particular stationary solution, which has no equivalent on $\R^\ds$. 
 Moreover,  the critical temperature $\beta_c$, defined as the  threshold for lower boundedness of the free energy \eqref{eq:introFE} is no longer the threshold for blow up of all solutions.  We exhibit three more thresholds $\bels, \beu$, and $\bei$, whose values depend on the dimension.
 
The threshold $\bels= \frac{(2\pi)^\ds}{\cdd}$ is the exact critical $\beta$ for which the uniform solution $\muu$ loses both its linear and nonlinear stability, as we show in \cref{thm:mainlim} and \cref{thm:mainUT}.  For $\beta<\bels$ we show that solutions that are initially close enough to $\muu$ converge exponentially fast to $\muu$, see Theorem \ref{thm:mainlim}.
This threshold $\bels$ equals $\bec$ if $\ds=1$, is strictly larger than $\bec$ if $2\leq\ds\leq 10$, and is strictly smaller than $\bec$ if $\ds\geq 11$. If $\bels>\bec$, then for $\beta \in (\bec,\bels) $,   the uniform solution thus remains nonlinearly stable, which means that being above $\bec$ does not necessarily imply that all solutions blow up, in sharp contrast with the $\R^\ds$ case!  
 
The third threshold $\beu \le \min (\bels,\bec)$ is the inverse temperature below which $\muu$ is the unique minimizer of \eqref{eq:introFE}; above $\beu$ there may be others (and if so there is an infinite number of them by translation invariance). In particular, above $\beu$, there exists a nonuniform stationary solution to \eqref{eq:lima}.

Finally, the last threshold $\bei$ is the functional inequality threshold: we prove that for $\beta<\bei$, the modulated energy can be controlled by the relative entropy in the form of the \emph{modulated logarithmic Hardy-Littlewood-Sobolev} (or mLHLS) inequality
\begin{equation} 
\E_{f_N}\Big[ F_N(\XN, \mu)\Big] \le \frac1\beta H_N(f_N|\mu^{\otimes N} ) + o_N(1),
\end{equation}
valid for all probability densities $f_N$ on $(\T^\ds)^N$, and $\mu$ on $\T^\ds$ with $\|\log \mu\|_{L^\infty}$ small enough.
This functional inequality was not proved as such in \cite{BJW2020}. Instead, they proved it for a localized version of the interaction,  leaving another term which prevents the obtention of uniform-in-time convergence. The proof of the inequality relies on the analysis of what we may call a macroscopic \emph{modulated} free energy (compare to \eqref{eq:introFE}) 
\begin{equation}
\frac{1}{\beta} \int_{\T^{\ds}} \log\left( \frac{\nu}{\mu} \right)d\nu- \frac12\int_{\T^{\ds}} \g(x-y) d(\mu-\nu)^{\otimes 2}(x,y).
\end{equation}
The existence and uniqueness of minimizers of this functional, according to $\beta$, is an intriguing question which is directly related to the uniqueness of solutions to Kazhdan-Warner equations in the prescribed curvature problem and to several works in the literature. See \cref{sec:MLHLS} for a discussion. 

%In the next section, we present the statements of the main results. 

%\st{The purpose of the present article is two-fold: to study the well-posedness and long-time dynamics of the limiting equation \eqref{eq:lima} and to use this detailed knowledge of the limiting equation to determine over what time scales one can establish rates for entropic propagation of chaos, in particular whether one can establish rates which are \emph{uniform in time}. To the best of our knowledge, equation in our periodic setting and for arbitrary dimension has not been specifically studied. As we shall show, it is misleading to naively extrapolate results on $\R^\ds$ to $\T^\ds$. }

%\st{Propagation of chaos for the attractive log gas has been studied by Bresch et al. \cite{BJW2019crm, BJW2020}; we postpone a more detailed discussion of their result until \cref{ssec:introMF}. As the rates of convergence in \cite{BJW2019crm, BJW2020} deteriorate exponentially in time, one of the questions  posed in \cite{BJW2020}, which we answer, is the possibility of a uniform-in-time rate. Answering this question is, of course, contingent on having a satisfactory global theory for solutions of the limiting equation} 

\section{Main results}
In this section, we present the main results of this paper and comment on their proofs. The section is organized into several subsections. \cref{ssec:introlim} concerns the dynamics of the limiting equation, \cref{ssec:introMRrev} detours to discuss in detail the modulated free energy method and review the work \cite{BJW2019crm, BJW2020}, while \cref{ssec:introFI} concerns the mLHLS inequality, and \cref{ssec:introMF} concerns the mean-field convergence problem. %These subsections are intended to be read in the order presented.

\subsection{The limiting equation}\label{ssec:introlim}
On $\R^\ds$, it is an elementary computation (e.g., see the survey \cite{Blanchet2013}) involving the time derivative of the second moment that any nonnegative classical solution of \eqref{eq:lima} with finite second moment must blow up in finite time if $\beta>{2\ds}$.\footnote{The cited reference normalizes $\beta=1$ but allows the mass to vary, whereas we normalize the mass to be one and allow $\beta$ vary. These are equivalent through rescaling, as described in \cref{rem:timers}. Also, their interaction potential is $\frac{1}{\cdd}\g$, not $\g$ as here, but this again can be accounted for by rescaling.}\footnote{Recently, it has been shown by the second author and Staffilani \cite{RS2023} that this finite-time blow up disappears under certain stochastic perturbation of the dynamics and, in fact, one has exponential-in-time convergence to the uniform distribution \cite{ARS2023}.} On $\T^\ds$, this computation does not quite work, since the potential $\g$ is no longer exactly $-\log|x|$. However, by a localization argument and under an additional condition on the concentration of the density around a point, one can show a similar blow up result \cite{Nagai2001}. Moreover, it is not difficult to show if the temperature $\frac1\beta$ is too low, then finite-time blow up must occur (see \cite[Appendix I]{KX2016}). On $\T^{\ds}$, as we will show, it is \emph{not true} that all classical nonstationary solutions blow up in finite time if $\beta>{2\ds}$.% We also mention that for $\theta=0$, one has a sharp bound for the time of existence for sufficiently strong solutions to \eqref{eq:lima}, as well as exact solutions that provide an explicit example of finite-time collapse to a nontrivial measure \cite{BLL2012}.

As mentioned above, equation \eqref{eq:lima} is the gradient flow of the \emph{free energy} \eqref{eq:introFE} and therefore the free energy $\Ec_\beta$ decreases along the flow of solutions. But \emph{a priori} it is not even clear that the free energy is bounded from below---hence a useful quantity---since there is competition, as measured by the size of $\beta$, between the entropy (the repulsive part) and the interaction potential (the attractive part). The balance between the two parts is determined by the sharp constant in the \emph{logarithmic Hardy-Littlewood-Sobolev (HLS) inequality} (see \cref{lem:logHLS} below), and the free energy is bounded from below if and only if $\beta\leq 2\ds$. We will discuss this point more in \cref{sec:aLogFE}. %{\cb \st{Given the relevance of the threshold $\bec\coloneqq {2\ds}$ for the free energy and for global existence vs. finite-time blowup and following existing terminology in the literature, we classify the temperature as \textit{subcritical} ($\beta<\bec$), \textit{critical} ($\beta=\bec$), and \textit{supercritical} ($\beta>\bec$). The reader may view the subcritical and supercritical regimes as the high- and low-temperature regimes, respectively.}}

There has been extensive study of the equation \eqref{eq:lima} on $\R^2$, corresponding to the aforementioned PKS equation. A detailed discussion of these results is beyond the scope of the article. We briefly mention that finite-time blow up for strong solutions if $\be >\be_c$ is classical \cite{JL1992}. One has local well-posedness of mild solutions with finite free energy \cite{Blanchet2013gf, FM2016} and in fact for general probability measures with small atoms \cite{BM2014}. For initial data in $L^1$, one has a unique, global mild solution if and only if $\beta\leq \bec$ \cite{Wei2018}. For the asymptotic behavior of solutions, we refer to \cite{BDP2006, CD2014} ($\beta<\bec$), \cite{BCM2008, GM2018} ($\beta=\bec$), and \cite{Velazquez2002, Velazquez2004i, Velazquez2004ii} ($\beta>\bec$), and references therein. %In the case $\nu=0$, one has a sharp bound for the time of existence for compactly supported $L^\infty$ weak solutions to \eqref{eq:KS}, which are necessarily unique, as well as exact solutions that provide an explicit example of finite-time collapse to a nontrivial measure \cite{BLL2012}.

There are few works of which we are aware that even consider equation \eqref{eq:lima} posed on $\R^\ds$. \cite{CC2012} considers the $\ds=1$ and radial $\ds=2$ cases, and \cite{CPS2007} generalizes to arbitrary $\ds\geq 1$ properties of the free energy and the existence of certain global so-called free energy solutions. While there are some works (e.g., \cite{CPZ2004, BCL2009, CCE2012, OW2016, BKZ2018, SW2019, Naito2021}) considering high-dimensional versions of the parabolic-elliptic PKS equation, this corresponds to when $\g$ is the Coulomb, not log, potential, and therefore these are not relevant for us. This is an unfortunate gap in the literature, and part of the goal of this work is to address this gap for $\T^\ds$.

\medskip
%The local well-posedness is straightforward by a fixed point argument similar to the one used in the proof of \cref{prop:lwp}, which is ...Given initial datum $\mu^0\in L^1(\T^\ds)$, there is a unique mild solution $\mu$ to \eqref{eq:lima} in a certain a class, which moreover is $C^\infty$ for positive times. Moreover, as we are interested in probability density solutions, the sign and mass of $\mu^0$ are preserved by the evolution. Concerning long-time dynamics, if $\theta\geq \frac{1}{2d}$, then the solution $\mu$ is global. These results are summarized in the following proposition, the proof of which is divided over \ref{?}.

%We summarize the results contained in the preceding discussion in the statement of the following theorem.

By a contraction mapping argument, it is not difficult to show (see \cref{prop:LWPa}) that given an initial datum $\mu^0\in \P_{ac}(\T^\ds)$, the space of probability measures on $\T^{\ds}$ that are absolutely continuous with respect to the Lebesgue measure, there is a unique mild solution to \eqref{eq:lima} (see definition below) on some interval $[0,T]$ in the class of weakly continuous $\P_{ac}$-valid maps with the property that
\begin{align}
\sup_{0< t\leq T} (\beta t)^{\frac14}\|\mu^t\|_{L^{\frac{2\ds}{2\ds-1}}}<\infty.
\end{align}
The choice of an exponent in the last condition is not so important; it is convenient for reasons of scaling. An iteration argument implies an $L^\infty$ gain of integrability (see \cref{lem:hypequiv}), and moreover a gain of regularity to arbitrary order (see \cref{lem:greg}). In other words, solutions are instantaneously smooth. To show that the solution is global, one needs to show the norm $\|\mu^t\|_{L^{\frac{2\ds}{2\ds-1}}}$ cannot blow up in finite time. We accomplish this for $\be<\bec$ by exploiting \emph{a priori} bounds provided by the free energy and dissipation functionals on the Fisher information, discussed in \cref{sec:aLogFE}, which allows us to show that the $L^2$ norm of a solution can grow at most linearly in time (see \cref{lem:aLrbnd}). An $L^2$ bound suffices because $\frac{2\ds}{2\ds-1}\leq 2$. While the local theory makes no assumptions on the size of $\be$, in this last part, we crucially use that $\be<\bec$ to exploit the coercivity of the free energy.

The preceding allows us to conclude that solutions are globally smooth, but it does not say anything about their asymptotic dynamics as $t\rightarrow\infty$. Using that the free energy $\Ec_\beta$ is nonincreasing along the flow and a compactness argument (see \cref{lem:mutasyss}), it is not too hard to show that $\mu^t$ converges weakly to a probability density $\mu_\beta$ on which the dissipation functional vanishes. The latter implies that $\mu_\beta$ satisfies the equation, sometimes called the \emph{Kirkwood-Monroe equation} \cite{KM1941},
\begin{equation}\label{eq:introKM}
\mu_{\beta} = \dfrac{e^{\beta\g\ast\mu_{\beta}}}{Z_{\beta}}, \qquad Z_{\beta} \coloneqq \int_{\T^\ds}e^{\beta\g\ast\mu_\beta}dx,
\end{equation}
and therefore is a stationary solution to \eqref{eq:lima}. Define the \emph{uniqueness threshold}
\begin{align}\label{eq:beudef}
\beu \coloneqq \sup\{\be_0\geq 0: \forall  \be\in [0,\be_0], \ \text{$\muu$ is the unique solution} \eqref{eq:introKM}\}.
\end{align}
Note that by definition, for $\be<\beu$, uniqueness holds for equation \eqref{eq:introKM}, while there exists $\be>\beu$ for which equation \eqref{eq:introKM} has a nonuniform solution. Through a contraction argument (see \cref{prop:ssunq}), we prove that $\beu>0$. An explicit bound for $\beu$ is presented in \cref{ssec:Instabssunq}. In particular, for $\be<\beu$, this implies that $\muu$ is the unique minimizer of the free energy. If $\ds=2$, then by previous works, it is known that $\beu=\bec=4$ and that uniqueness holds for $\be\leq 4$ \cite{LL2006, GM2019, GGHL2021} and fails for all $\be>4$ \cite{RT1998, ST1998}. With uniqueness in hand, one can then show that both terms in $\Ec_\beta(\mu^t)$ converge to their respective values in $\Ec_\beta(\mu_\beta)$ as $t\rightarrow\infty$, which implies the vanishing of the relative entropy $H(\mu^t\vert\mu_\be)$. In turn, this implies the vanishing of $\|\mu^t-\mu_\be\|_{L^1}$ as $t\rightarrow\infty$ by Pinsker's inequality. Under the assumption that $\be<\bels$, we then show in \cref{lem:aL2exp} that if $\mu^t$ is close enough in $L^1$ to $\muu$ for all $t\geq T_0$, for some time $T_0\geq 0$, the quantity $\|\mu^t-\muu\|_{L^2}$ decays exponentially fast. In fact, we also show that under the stronger assumption that $\mu^t$ is close enough to $\muu$ in $L^2$ at some time $t=T_0$, then this is true for all future times $t\geq T_0$. Combining this result with \cref{lem:mutasyss}, we conclude that for $\be<\beu$, all solutions eventually become close enough in $L^1$ distance to $\muu$ for our $L^2$ exponential-in-time convergence result to apply. Through a time translation trick, we may then combine this with the local smoothing result \cref{lem:greg} to obtain exponential-in-time decay of $\|\nabla^{\otimes n}\mu^t\|_{L^\infty}$ for arbitrary $n\geq 1$ (see \cref{lem:nabLinfexp}). On the other hand, we can show (see \cref{prop:NLinstab}) that for $\be>\bels$ and $\ep>0$ arbitrarily small, there exists an $O(\ep)$ perturbation $\mu_\ep^0$ of $\muu$ such that in time $O(\log\frac1\ep)$, the associated solution $\mu_\ep$ to \eqref{eq:lima} is $\frac12$ distance from $\muu$. We summarize the above discussion with the following theorem. Explicit forms of the error $O(e^{-ct})$, in particular how it depends on the initial datum, and the time $t_\ep$ are available in \cref{sec:glob} and \cref{sec:Instab}, respectively.

\begin{mthm}[Global existence and trend to equilibrium]\label{thm:mainlim}
Let $\ds\geq 1$. Given $\mu^0\in \P_{ac}(\T^\ds)$, there exists a unique maximal lifespan mild solution $\mu$ to \eqref{eq:lima} in the class $C_{w}([0,T_{\max}), \P_{ac})$, the space of weakly continuous functions with values in $\P_{ac}$, such that
\begin{align}
\sup_{0<t\leq T} t^{\frac{1}{4}}\|\mu^t\|_{L^{\frac{2\ds}{2\ds-1}}} < \infty, \qquad \forall T\in (0,T_{\max}).
\end{align}
This solution has the property that $\mu\in C^\infty((0,T_{\max})\times \T^\ds)$. If $\beta< \bec$, then $T_{\max}=\infty$, and there exists a solution $\mu_\beta$ of \eqref{eq:introKM} such that the relative entropy $H(\mu^t\vert\mu_\beta)\rightarrow 0$ as $t\rightarrow\infty$.

If $\be<\bels \coloneqq \frac{(2\pi)^\ds}{\cdd}$, then there exists $\delta_\be>0$, depending on $\ds,\be$, such that if there exists $T_0\geq 0$ for which
\begin{align}
\|\mu^{t}-\muu\|_{L^1} \leq \delta_\be, \qquad \forall t\geq T_0,
\end{align}
or there exists $T_0\geq 0$ such that
\begin{align}
\|\mu^{T_0} - \muu\|_{L^2} \leq \delta_\be,
\end{align}
then
\begin{align}\label{eq:mainlimnabLinf}
\forall n\geq 0, \qquad \|\nabla^{\otimes n}(\mu^t-\muu)\|_{L^\infty} = O(e^{-c t}), \qquad {\text{as $t\rightarrow\infty$}},
\end{align}
for some $c>0$ depending on $\ds,\be$. Additionally, for any $\be<\beu$, the relative entropy $H(\mu^t\vert \muu)\rightarrow 0$ as $t\rightarrow\infty$ and \eqref{eq:mainlimnabLinf} holds. In particular, for $\ds=2$, \eqref{eq:mainlimnabLinf} holds for all $\beta<\bec$.

Additionally, {if $\be>\bels$}, then for all $\ep>0$ sufficiently small depending on $\ds,\be$, there exists a $C^\infty$ solution $\mu_\ep$ to \eqref{eq:lima} with $\|\mu_\ep^0-\muu\|_{W^{n,\infty}} = O(\ep)$, for any $n$, such that for time $t_\ep = O(\log\frac{1}{\ep})$,
\begin{align}
\frac12 \leq \|\mu_{\ep}^{t_\ep} - \muu\|_{L^1}.
\end{align} 
\end{mthm}
\medskip
We emphasize that we do not need to assume $\beta <\bec$ for the second part of the result. The threshold $\bels$, which we call the \emph{stability threshold} is such that for $\be<\bels$, the linearization $L_{\beta}$ of the right-hand side of equation \eqref{eq:lin} about the uniform distribution $\muu$ has no eigenvalues with positive real part, while for $\be>\bels$, $L_{\beta}$ has eigenvalues with positive real part, implying linear instability of the uniform distribution. We refer to \cref{sec:Instab} for the detailed calculation. Recall that a stationary state $\mu_*$ of a PDE is said to be \emph{nonlinearly stable} if for every $\ep>0$, there is a $\delta>0$ such that if one takes initial data $\mu^0$ such that $\|\mu^0-\mu_*\|_1<\delta$, for some norm $\|\cdot\|_1$, then $\sup_{t\geq 0} \|\mu^t-\mu_*\|_2 < \ep$, for some norm $\|\cdot\|_2$. We say that $\bmu$ is \emph{nonlinearly unstable} if it is not nonlinearly stable. \cref{thm:mainlim} implies that $\muu$ is nonlinearly stable if $\be<\bels$ and unstable if $\be>\bels$. %\st{}
To our knowledge, such a quantitative instability result for $\be>\bels$ has not been previously demonstrated, even for the $\ds=2$ case.

Instability of $\muu$ is intimately tied to the fact that $\muu$ \emph{does not minimize} the free energy $\Ec_{\be}$ when $\be>\bels$. If $\bec\leq\bels$, then this is implied by the unboundedness from below of the free energy (see \cref{lem:FEsc}). When $\be>\min(\bels,\bec)$, the minimal free energy $\Ec_\beta^*$ is strictly negative, and we compute an---admittedly suboptimal---upper bound for it in \cref{prop:ssnounq}. The minimal value of the free energy is also tied to the uniqueness of solutions of \eqref{eq:introKM} (i.e., stationary solutions of \eqref{eq:lima}). If $\bec<\bels$, then we show through the mountain pass geometry of $\Ec_\be$ that for $\be\in (\bec,\bels)$, there is a nonuniform solution (see \cref{prop:ssnounq2}). While if $\bels<\bec$, then for every $\be\in (\bels,\bec)$, the minimizer of $\Ec_\be$ exists and is nonuniform (see \cref{prop:ssnounq}). The $\ds=1$ case where $\bec=\bels$ will be treated in a separate work. See \cref{rem:1D} below for further comments. These results imply that the maximal uniqueness threshold $\beu\leq \min(\bels,\bec)$. We summarize our discussion on uniqueness with the next theorem.

\begin{mthm}\label{thm:mainstab}
If $\be>\bels$, then the minimal free energy, viewed as an element of $[-\infty,0]$, satisfies the upper bound
\begin{align}
\Ec_\beta^*\coloneqq \inf_{\mu \in \mathcal{P}_{ac}(\T^\ds)} \Ec_\beta(\mu)\leq -\frac{\be}{432}\left|\frac1\bels-\frac1\be\right|^3 < 0 = \Ec_\beta(\muu).
\end{align}
If $\be>\bec$, then $\Ec_\beta^* = -\infty$. If $\bels<\bec$, then for any $\be\in (\bels,\bec)$, the minimizer of $\Ec_\beta$ is nonuniform. If $\bec<\bels$, then for $\be\in (\bec,\bels)$, equation \eqref{eq:introKM} admits a nonuniform solution. {Therefore}, $\beu\leq \min(\bec,\bels)$.
\end{mthm}

\begin{remark}
\cite{BM2014} shows for equation \eqref{eq:lima} on $\R^2$ that if $\beta<4$, then for initial data $\mu^0\in \P(\R^2)$, there is a unique solution to the mild form of the equation. Combining this with the results of \cite{FM2016}, these solutions are global and, after passing to self-similar variables, converge in relative entropy to the unique minimizer of the free energy. We suspect that with substantially more work, one could prove global well-posedness on $\T^\ds$, general $\ds\geq 1$, for measure initial data and convergence to a steady state. The uniqueness of this steady state, though, is less clear. We have chosen not to pursue this direction as for applications, we are only interested in absolutely continuous initial data.
\end{remark}

\begin{remark}\label{rem:1D}
The $\ds=1$ case is actually exactly solvable in terms of the system of differential equations for the Fourier coefficients of the solution. Because this result is of interest in its own right, the details will be reported elsewhere \cite{CdCRS20231d}.
\end{remark}

\begin{remark}
\cref{thm:mainstab} leaves a gap concerning nonuniqueness of solutions to equation \eqref{eq:introKM} for dimensions $2\leq \ds\leq 10$ where $\bec<\bels$ and therefore, we do not know nonuniqueness for $\be\geq \bels$. The $\ds=2$ case has been treated for all $\be>\bec$ \cite{RT1998, ST1998}, and it would be interesting to determine if nonuniqueness holds for all $\be>\bec$.
\end{remark}

\subsection{The modulated free energy method}\label{ssec:introMRrev}
As previously mentioned, Bresch et al. \cite{BJW2019crm, BJW2020} considered the problem of mean-field convergence/propagation of chaos for the system \eqref{eq:SDEa}, proving the first quantitative result for the full range $\be<\bec$. Previous results \cite{CP2016ks, FJ2017} only show a weak form of mean-field convergence starting from the system \eqref{eq:SDEa}. Namely, \cite{FJ2017} shows that for sufficiently  small $\be$, any $N\rightarrow\infty$ limit point of the sequence of empirical measures to the system \eqref{eq:SDEa} is a type of weak solution to equation \eqref{eq:lima}. This result has been recently extended (subsequent to \cite{BJW2020}) up to and including the critical temperature in \cite{Tardy2022}. However, these results are not quantitative, do not characterize the limit point due to its being too weak a notion of solution to \eqref{eq:lima}---in principle, uniqueness may fail in the class which the limit point is shown to belong---nor do they imply propagation of chaos due to the necessity of a subsequence extraction. As one would expect, more results are known for mean-field convergence/propagation of chaos starting from regularized versions of the SDE system \eqref{eq:SDEa} \cite{HS2011, GQ2015, ORT2020}. We emphasize that all of the aforementioned results only give convergence which holds \emph{locally in time}.

The proof of \cite{BJW2019crm, BJW2020} is through their \emph{modulated free energy method}, discussed at some length in \cite[Subsection 1.2]{CdCRS2023}. We briefly recall that this method proceeds by showing a Gr\"onwall relation for the modulated free energy  defined in \eqref{def:modulatedfreenrj}.
%\st{\begin{equation}\label{def:modulatedfreenrj}
%E_N(f_N, \mu) \coloneqq \frac1\beta H_N\left(f_N \vert \mu^{\otimes N}\right) + \Eb_{f_N}\left[F_N(\ux_N, \mu)\right].
%\end{equation}
%which is a combination of the normalized relative entropy from \cite{JW2016, JW2018},
%\begin{equation}\label{eq:REdef}
%H_N \left(f_N \vert \mu^{\otimes N}\right) \coloneqq \frac{1}{N} \int_{(\T^\ds)^N} \log\paren*{\frac{f_N}{\mu^{\otimes N}}} df_N,
%\end{equation}
%and the modulated energy from \cite{SS2015log, SS2015, RS2016, PS2017, Duerinckx2016, Serfaty2020, NRS2021},
%\begin{equation}\label{def:modulatedenergy}
%\Fr_N(\ux_N,\mu) \coloneqq \int_{(\T^\ds)^2\setminus\triangle} \g(x-y)d\paren*{\frac{1}{N}\sum_{i=1}^N\d_{x_i} - \mu}^{\otimes 2}(x,y),
%\end{equation}
%where $\triangle \coloneqq \{(x,x) \in (\T^\ds)^2\}$. Since we work in the statistical setting of the Liouville equation \eqref{eq:Lioua}, the modulated energy $\Fr_N(\ux_N,\mu)$ is averaged with respect to the joint law $f_N$ of the positions $\ux_N$.}
As observed recently in \cite{RS2023lsi}, one can rewrite the modulated free energy as
\begin{align}
E_N(f_N, \mu) = \frac{1}{\beta} \( H_N(f_N| \mathbb{Q}_{N,\beta}(\mu))+ \frac{\log K_{N,\beta}(\mu)}{N}\),
\end{align}
where
\begin{align}\label{defQ}
\mathbb{Q}_{N,\beta}(\mu) &\coloneqq \frac{1}{K_{N,\beta}(\mu)} e^{\beta  N F_N(\XN, \mu)} d\mu^{\otimes N}(\XN),\\
K_{N,\beta}(\mu) &\coloneqq \E_{\mu^{\otimes N}}\Big[e^{\beta N F_N(\XN, \mu)}\Big] \label{defKNbe}
\end{align}
are the modulated Gibbs measure and partition function, respectively. This reformulation makes it clear that the modulated free energy is itself a relative entropy between two probability measures up to an error that is negligible under a certain smallness of free energy {\cb (that is $N^{-1}\log K_{N,\be}(\mu)$)} condition. 

Given a suitable solution (so-called entropy solutions discussed below) $f_N^t$ of the forward Kolmogorov equation \eqref{eq:Lioua} and solution $\mu^t$ of the limiting equation \eqref{eq:lima}, the modulated free energy satisfies the differential inequality
\begin{multline}\label{eq:introdtEN}
\frac{d}{dt} E_N(f_N^t, \mu^t) \leq  \frac{1}{2} \E_{f_N^t}\Bigg[\int_{(\T^\ds)^2\setminus\triangle} (u^{t} (x)-u^{t}(y))\cdot \nabla\g(x-y) d\left(\frac1N\sum_{i=1}^N\d_{x_i} - \mu^t\right)^{\otimes 2}(x,y)\Bigg]\\
-\frac{1}{\beta^2 N}  \int_{(\R^\ds)^N } \left |\nab \sqrt{   \frac{f_N^t}{\mathbb{Q}_{N,\beta} (\mu^t) } }\right|^2 d \mathbb{Q}_{N,\beta}(\mu^t),
\end{multline}
where $u^t \coloneqq \frac1\beta\nabla\log\mu^t - \nabla \g \ast \mu^t$ is the velocity field associated to the mean-field dynamics \eqref{eq:lima}.

At first pass, the  second term on the right-hand side of \eqref{eq:introdtEN}, which is a multiple of the relative Fisher information between $f_N^t$ and $\mathbb{Q}_{N,\beta} (\mu^t)$ and is nonpositive,  may be discarded. Expressions having the form of the first term on the right-hand side of \eqref{eq:introdtEN} appear naturally when computing the variations of the modulated energy along a transport field. These expressions  have been called ``commutator terms" because they have the structure of a quadratic form associated to a commutator. Functional inequalities have been developed to estimate such expressions and show they are bounded in terms of the modulated energy for general log/Riesz-type potentials \cite{LS2018, Serfaty2020, BJW2019edp, Serfaty2023, Rosenzweig2021ne, NRS2021, RS2022}. For the logarithmic-type potential we consider here, we are in an easier situation since, assuming $u^{t}$ is Lipschitz, we have
\begin{equation}
\paren*{u^{ t}(x)-u^{ t }(y)}\cdot\nabla\g(x-y) \in L_{x,y}^\infty((\T^\ds)^2),
\end{equation}
by the mean-value theorem. These functional inequalities can then be avoided in favor of a cruder argument combining the Donsker-Varadhan lemma with large deviation type estimates originating in \cite{JW2018} (see \cref{lem:logFI}). This way one can show that the right-hand side of \eqref{eq:introdtEN} is bounded by
\begin{equation}\label{eq:introMFErhs}
C^t\paren*{H_N(f_N^t \vert (\mu^t)^{\otimes N}) +\frac{C^t}{N}},
\end{equation}
for some constant $C^t>0$ depending on $\|\log\mu^\tau\|_{W^{2,\infty}}$. To close the Gr\"onwall loop, one needs to bound the relative entropy in terms of the modulated free energy $E_N$. Since the interaction is attractive, this is far from obvious and is the main difficulty of \cite{BJW2020}.

%The Gr\"onwall-Bellman lemma then allows to conclude that
%\begin{equation}\label{eq:introENgron}
%|E_N(f_N^t,\mu^t)| \leq \paren*{|E_N(f_N^0,\mu^0)| +  \frac{1}{N}\int_0^t C^\tau d\tau }
%\end{equation}

\begin{remark}\label{rem:TvsR}
Controlling  $\|\log\mu^t\|_{W^{2,\infty}}$ is delicate on Euclidean space due to the  decay of $\mu^t $ to $0$ at infinity. This issue, of course, disappears on the torus (likely more generally a bounded domain with appropriate boundary conditions)---and therefore motivates the restriction to the periodic case in all previous works using the modulated free energy---since the solution to \eqref{eq:lima} can be ensured to remain bounded from below (at least locally in time) provided its initial datum is. We refer to \cite{FW2023, Rosenzweig2023re} for some recent progress on extending entropy methods for propagation of chaos to Euclidean space. %Without some confining potential $\Vext$ added to the dynamics \eqref{eq:SDE} and \eqref{eq:lim} and for solutions which start near equilibrium (which is no longer uniform), we have doubts that the modulated free energy method can be implemented on $\R^\ds$ \cite{HRS2022}.
\end{remark}

\emph{A priori} it is not clear that the modulated free energy is a coercive quantity in the sense that it controls the relative entropy (up to some error vanishing as $N\rightarrow\infty$), since there is competition between the relative entropy and the modulated energy. Such coercivity is needed to obtain a closed estimate from \eqref{eq:introdtEN}, \eqref{eq:introMFErhs} and to show that modulated free energy metrizes propagation of chaos. In fact, Bresch et al. were unable to prove this coercivity. Instead, they introduce a truncation of the potential $\g$ into ``short-range'' (sr) and ``long-range'' (lr) parts:
\begin{equation}
\g = \g\chi_r + \g(1-\chi_r) \eqqcolon \g_{sr} + \g_{lr},
\end{equation}
where $\chi$ is a $C^\infty$ bump function supported in a neighborhood of the origin and $\chi_r(x)\coloneqq \chi(x/r)$, for $r>0$. Provided $r$ is sufficiently small depending on $\ds,\be$, for any $\beta<\bec$, Bresch et al. show that there is a constant $C>0$ depending on $\ds,\beta$ and norms of $\mu$ and an exponent $\ga>0$ depending on $\ds$ such that for any $N$ sufficiently large, $f_N \in \P_{ac}((\T^\ds)^N)$ and sufficiently regular $\mu\in \P_{ac}(\T^\ds)$ which is bounded from below, one has 
\begin{equation}\label{eq:introLMHLS}
\beta\Eb\left[\Fr_{N,sr}(\ux_N,\mu)\right] \leq H_N(f_N \vert \mu^{\otimes N}) + CN^{-\ga},
\end{equation}
where $\Fr_{N,sr}$ is the modulated energy \eqref{def:modulatedenergy} with $\g$ replaced by $\g_{sr}$. In analogy to the classical log HLS inequality (see \cref{lem:logHLS} below), we refer to the functional inequality \eqref{eq:introLMHLS} as a \emph{modulated logarithmic HLS (mLHLS) inequality}.

Their proof of the functional inequality \eqref{eq:introLMHLS} proceeds through large deviation arguments similar to those appearing in previous work of Bodineau and Guionnet \cite{BG1999}. The freedom to make $r$ small is essential. At one point, one has to show that the large deviation rate functional
\begin{equation}\label{eq:introRF}
\mathscr{I}(\mu) \coloneqq \sup_{\rho\in\P(\T^\ds)} \paren*{\be\int_{(\T^\ds)^2}\tl\g_{sr}(x-y)d(\rho-\mu)^{\otimes 2}(x,y) - \int_{\T^\ds}\log\paren*{\frac{d\rho}{d\mu}}d\rho} =0,
\end{equation}
where $\tl\g_{sr}$ is a modification of $\g_{sr}$ to remove the singularity at the origin. This is done by computing the Euler-Lagrange equation and showing uniqueness of solutions through a contraction argument; and this step requires that $\|\g_{sr}\|_{L^1}$ is sufficiently small, which can, of course, be guaranteed by taking $r$ sufficiently small. 

There is still the long-range contribution of $\g_{lr}$ that has to be accounted for. This is treated in \cite{BJW2020} directly: $\g_{lr}$ is $C^\infty$, and therefore one can directly perform a Gr\"onwall argument on the quantity
\begin{equation}
\Eb_{f_N^t}\left[\Fr_{N,lr}(\ux_N,\mu^t)\right],
\end{equation}
where $\Fr_{N,lr}$ is obtained from \eqref{def:modulatedenergy} by replacing $\g$ with $\g_{lr}$ defined above. %Using equation \eqref{eq:}, one computes (cf. \cite[]{RS2021})
%\begin{equation}
%\frac{d}{dt}\int_{(\T^\ds)^N}\Fr_{N,lr}(\ux_N,\mu^t)df_N^t(\ux_N) = 
%\end{equation}
%The right-hand side may be bounded an expression of the form \eqref{eq:introMFErhs} by the same arguments.
Unfortunately, this part is the source of trouble in obtaining a uniform-in-time estimate from their proof. Namely, one encounters a term of the form
\begin{align}
\E_{f_N^t}\Bigg[\int_{(\T^\ds)^2}\Delta\g_{lr}(x-y)d\left(\frac1N\sum_{i=1}^N \delta_{x_i}-\mu^t\right)^{\otimes 2}(x,y)\Bigg],
\end{align}
the bound for which (equation (2.16) in the cited work) has no clear means of decay in time.

\subsection{Functional inequalities}\label{ssec:introFI}
We observe and show that it is possible to prove this inequality \eqref{eq:introLMHLS} \emph{without truncation}, provided the temperature is \emph{sufficiently high}. One encounters a similar problem of uniqueness of minimizers of a large deviation functional; but the freedom to take the temperature high replaces the freedom to take $r$ small.

Perhaps even more interesting---and unanticipated---we show that a functional inequality of the form \eqref{eq:introLMHLS} \emph{cannot hold if the temperature is too low}. Namely, if $\beta>\min(\bels,\bec)$, then one can use the strict negativity of the minimal free energy to show that for the reference measure $\muu$, the expected modulated energy dominates the temperature-weighted relative entropy.

We summarize the results contained in the preceding discussion with the following theorem. An explicit form of the rather complicated error in \eqref{eq:mainMLHLS1} below is given in \cref{ssec:MLHLSprf}.

\begin{mthm}[mLHLS inequality] \label{thm:mainMLHLS}
Let $\ds\geq 1$. There exists a maximal $\bei>0$, such that for every $0\leq \be <\bei$, there exists a $\delta_{\be}>0$, tending to $\infty$ as $\be\rightarrow 0$, such that for all $\mu\in\P_{ac}(\T^\ds)$ with $\|\log\mu\|_{L^\infty}\leq \delta_{\be}$ and $f_N\in \P_{ac}((\T^\ds)^N)$, it holds that for all $N=N(\ds,\be)$ sufficiently large,
\begin{align}\label{eq:mainMLHLS1}
\E_{f_N}\left[\Fr_N(\XN,\mu)\right] \leq \frac{1}{\be}H_N(f_N \vert \mu^{\otimes N}) + \mathbf{C}(\beta,\|\log\mu\|_{W^{1,\infty}}) N^{-\ga},
\end{align}
where $\ga>0$ depends only on $\ds$ and $\mathbf{C}:[0,\infty)^2\rightarrow [0,\infty)$ is a continuous, increasing function of its arguments, vanishing if any of them are zero.
%If $\ds=2$, then there exist constants $C_1,C_2>0$, which depend only on $\ds$, such that the following holds. For any $\be\leq 2\ds$, $f_N \in \P_{ac}((\T^\ds)^N)$, and $\mu \in \P(\T^\ds)$, with $\log\mu \in L^\infty$ and ?, we have

\medskip
Conversely, if $\be> \min(\bec,\bels)$, then there exists a $\eta_\be>0$ such that the following holds: for any $N$, there is an $f_N \in \P_{ac}((\T^\ds)^N)\cap C^\infty((\T^\ds)^N)$ such that
\begin{align}\label{eq:mainMLHLS2}
\frac{1}{\be} H_N(f_N\vert \muu^{\otimes N})  \leq  \E_{f_N}\left[\Fr_N(\ux_N,\muu)\right] - \eta_\be.
\end{align}
\end{mthm}

In principle, there is a gap between $\bei$ and $\min(\bels,\bec)$. We expect that the affirmative portion \eqref{eq:mainMLHLS1} of the theorem should hold for $\be<\min(\bels,\bec)$, but are unable to prove this case, except for when $\ds=2$ and $\mu=\muu$. The difficulty is showing that the rate functional \eqref{eq:introRF} is zero by showing that $\muu$ is the unique solution of the Euler-Lagrange equation (see \eqref{eq:minmeasid}).

\begin{remark}\label{rem:MLHLSunif}
If $\ds=2$ and $\mu=\muu$, then \eqref{eq:mainMLHLS1} holds for any $\be<\bec=4$. This is because the rate functional \eqref{eq:introRF} with $\tl{\g}_{sr}$ replaced by $\g$ becomes equivalent to the minimal value of the free energy \eqref{eq:introFE}. By the sharp uniqueness result for the Kirkwood-Monroe equation \eqref{eq:introKM} discussed above, the minimal free energy is zero.
\end{remark}

\begin{comment}
{\cre
\begin{remark}\label{rem:MLHLSbeu}
If we define
\begin{align}
\bei\coloneqq \sup\{\be_0 : \forall \be\leq \be_0, \ \text{equation \eqref{eq:mainMLHLS1} holds for some $\gamma>0$}\},
\end{align}
then for all $\be<\bei$, we have uniqueness of solutions to equation \eqref{eq:introKM}. Indeed, given two solutions $\mu_\be^1,\mu_\be^2$, setting $\mu=\mu_\be^1$ and $f_N = (\mu_\be^2)^{\otimes N}$, we find from inserting into \eqref{eq:mainMLHLS1} and taking the limit as $N\rightarrow\infty$ of both sides that
\begin{align}
\int_{(\T^\ds)^2}\g(x-y)d(\mu_\be^1-\mu_\be^2)^{\otimes 2}(x,y) \leq \frac1\be H(\mu_\be^1\vert\mu_\be^2).
\end{align}
Using that $\mu_\be^1,\mu_\be^2$ are critical points of $\Ec_\be$, this implies that
\begin{align}
\Ec_\be(\mu_\be^1) - \Ec_\be(\mu_\be^2) =  \frac1\be H(\mu_\be^1\vert\mu_\be^2) - \int_{(\T^\ds)^2}\g(x-y)d(\mu_\be^1-\mu_\be^2)^{\otimes 2}(x,y) \geq 0.
\end{align}
Swapping the roles of $\mu_\be^1$ and $\mu_\be^2$ yields
\end{remark}
}
\end{comment}

%\begin{remark}
%The choice of norms of $\mu$ in the negative assertion of \cref{thm:mainMLHLS} is not really important. One can replace them with arbitrarily strong norms (e.g., Gevrey, analytic, etc.) and the inequality would still be false. 
%\end{remark}

\subsection{Uniform-in-time propagation of chaos}\label{ssec:introMF}
As alluded to above, there is an underlying reason---beyond aesthetics---for our revisiting the proof of the mLHLS lemma, which is the problem of uniform-in-time convergence. The possibility of such a result was posed as a question in \cite{BJW2020}, as their proof did not exploit the diffusive effect of temperature on the long-time dynamics of the solution to \eqref{eq:lima}---namely, the relaxation to steady states.
% Avoiding truncating the potential is essential because the treatment of the long-range contribution in \cite{BJW2020} spoils the possibility of a uniform-in-time estimate.

By using our results from the previous subsection for the relaxation of the solution $\mu^t$ to \eqref{eq:lima} to equilibrium, and following the strategy of our companion work \cite{CdCRS2023} for incorporating relaxation rates into the modulated free energy method, we are able to obtain a uniform-in-time bound for the modulated free energy. To the best of our knowledge, this is the first result of uniform-in-time propagation of chaos for systems with attractive, singular interactions. The coercivity of the modulated free energy then implies a uniform-in-time bound for the relative entropy, which by subadditivity and Pinsker's inequality also yields a rate for propagation of chaos in total variation distance. However, this rate is expected to be suboptimal (cf. \cite{Lacker2023, LlF2023, BJS2022,hCR2023}).

Additionally, we use our quantitative instability result of \cref{thm:mainstab} to show that if $\beta>\bels$, it is not possible to have a uniform-in-time relative entropy estimate of the following form:
\begin{align}\label{eq:EntUT}
\forall t\geq 0,\qquad H_N(f_N^t \vert (\mu^t)^{\otimes N}) \leq C\paren*{o_N(1) + H_N(f_N^0 \vert (\mu^0)^{\otimes N})},
\end{align}
where $f_N^t$ is any entropy solution to the $N$-particle forward Kolmogorov equation, $\mu^t$ is any smooth solution to equation \eqref{eq:lima}, $C>0$ is a constant depending only on $\ds,\be,\mu^0$, and $o_N(1)$ is some quantity, depending on $\ds,\be,\mu^0$, which vanishes as $N\rightarrow\infty$.

The above discussion leads us to our final main result. The reader may consult \cref{sec:MFE} for the explicit form of the errors in \eqref{eq:mainUTent} below. The value of the exponent $\gamma$ corresponds to the rate of convergence and is almost certainly not sharp. The restrictions on $\beta$ and $\mu$ in \cref{thm:mainUT} are inherited from \cref{thm:mainMLHLS}. Improving the range of applicability of the latter would improve the range of applicability of the former.

\begin{mthm}[Uniform and nonuniform propagation of chaos]\label{thm:mainUT}
Let $\ds\geq 1$ and $\beta<\bei$, where $\bei$ is as in \cref{thm:mainMLHLS}. There exists $\delta_\be>0$, depending on $\ds,\be$ and tending to $\infty$ as $\be\rightarrow 0$, and functions $\mathbf{C}_{1},\mathbf{C}_2: [0,\infty)^2\rightarrow [0,\infty)$, continuous, increasing, and vanishing if any of them are zero, and constant $\ga>0$, depending on $\ds$, such that the following holds. Let $\mu\in L^\infty([0,\infty), W^{2,\infty})$ be a solution to \eqref{eq:lima}  with $\|\log\mu^0\|_{L^\infty} \leq \delta_\be$. Let $f_N \in L^\infty([0,\infty), \P_{ac})$ be an entropy solution to the forward Kolmogorov equation \eqref{eq:Lioua}.\footnote{The existence of such a solution is sketched in \cite[Appendix 4.2]{BJW2020}.} Then
\begin{align}\label{eq:mainUTent}
\forall t\geq 0, \qquad H_N(f_N^t\vert (\mu^t)^{\otimes N}) \leq \mathbf{C}_1(\beta,\|\log\mu^0\|_{W^{2,\infty}})\left(E_N(f_N^0, \mu^0) + \mathbf{C}_2(\beta, \|\log \mu^0\|_{W^{2,\infty}}) N^{-\gamma}\right).
\end{align}

If $\be>\bels$, then for all $\ep>0$ sufficiently small depending on $\ds,\be$, there exists a $C^\infty$ solution $\mu_\ep$ to \eqref{eq:lima} with $\|\mu_\ep^0-\muu\|_{W^{n,\infty}} = O(\ep)$, for any $n$, such that for time $t_\ep = O(\log\frac{1}{\ep})$,
\begin{align}
\|f_{N;1}^{t_\ep} - \muu\|_{L^1} + \|f_{N;1}^{t_\ep}-\mu_{\ep}^{t_\ep}\|_{L^1} \geq \frac12,
\end{align}
where $f_{N}$ is an entropy solution of the forward Kolmogorov equation \eqref{eq:Lioua} with initial datum $(\mu_\ep^0)^{\otimes N}$. In particular, since $H_N(f_N^0\vert (\mu_\ep^0)^{\otimes N}) = 0$ and $H_N(f_N^0\vert \muu^{\otimes N}) = O(\ep)$, we see that \eqref{eq:EntUT} cannot hold.
\end{mthm}

\begin{remark}
If $\ds=2$ and $\mu^0=\muu$, then \eqref{eq:mainUTent} holds for all $\be<\bec=4$ as a consequence of \cref{rem:MLHLSunif}.
\end{remark}

%\begin{remark}
%As is by now well-known, the vanishing of the relative entropy or modulated energy, and therefore the modulated free energy, as $N\rightarrow\infty$, implies propagation of chaos. %We refer to \cref{rem:MFEcoer} below for further explanation.
%\end{remark}

\begin{remark}
It is also interesting to consider flows with both a repulsive Riesz and attractive log part. This can be treated by combining the results of this section with the results of \cite{CdCRS2023}. We leave this extension to the reader. We caution, though, that for such mixed repulsive-attractive systems, the critical temperature \emph{a priori} is not the same as in the purely attractive case.
\end{remark}

\subsection{Comparison of the thresholds}
Let us explain now the ordering between the various $\be_{(\cdot)}$ parameters introduced in the preceding subsections. We recall that $\bec=2\ds$ is the critical inverse temperature corresponding to the boundedness from below of the free energy $\Ec_\be$. $\bels =\frac{(2\pi)^\ds}{\cdd}$ is the threshold for stability of the uniform distribution. $\beu$ is the threshold for uniqueness of solutions to \eqref{eq:introKM}, i.e. uniqueness holds $\be<\beu$, while there exists at least one $\be>\beu$ for which uniqueness fails. $\bei$ is the maximal value such that for all $\be<\bei$ the mLHLS inequality holds for sufficiently small $\log\mu$ and sufficiently large $N$.

In low dimensions $\bels\geq \bec$, while in high dimensions, $\bec>\bels$. Indeed, recalling the definition of $\cdd$ from \eqref{eq:gdefa}, we see that
\begin{align}
\be<\bels =  \frac{2(2\pi)^{\ds}}{\Gamma(\ds/2) (4\pi)^{\ds/2}}= \frac{2\pi^{\ds/2}}{\Gamma(\ds/2)}.
\end{align}
%Since
%\begin{equation}
%\Gamma(\ds/2) = 
%\begin{cases}
%(\frac{\ds}{2}-1)!, & {\ds \ \text{even}} \\ \frac{(\ds-1)!}{4^\ds (\frac{\ds-1}{2})!}\sqrt{\pi}, & {\ds \ \text{odd}},
%\end{cases}
%\end{equation}
Using the lower bound $\Gamma(x)\geq (x/e)^{x-1}$ for all real $x\geq 2$, we see that for $\ds\geq 4$, 
\begin{align}
\frac{2\pi^{\ds/2}}{\Gamma(\ds/2)} \leq \frac{2\pi^{\ds/2}}{(\ds/2e)^{\frac{\ds}{2}-1}} = \frac{(2\pi e)^{\ds/2}}{e \ds^{\ds/2}}.
\end{align}
So for $\ds\geq 2\pi e$, we have
\begin{align}
\bels \leq \frac{1}{e} < 2\ds = \bec.
\end{align}
If $\ds$ is an integer and $\ds\leq 2\pi e$, then $\ds\leq 17$. By direct computation,\footnote{We used Mathematica to compute these finitely many cases.}  we obtain the following:
\begin{center}
\begin{tabular}{|cccccccccccccccccc|}
\hline
$\ds$ & 1 & 2 & 3 & 4 & 5 & 6 & 7 & 8 & 9 & 10 & 11 & 12 & 13 & 14 & 15 & 16 & 17\\
\hline
$\bels$ & $2$ & $2\pi$ & $4\pi$ & $2\pi^2$ & $\frac{8\pi^2}{3}$ & $\pi^3$ & $\frac{16\pi^3}{15}$ & $\frac{\pi^4}{3}$ & $\frac{32\pi^4}{105}$ & $\frac{\pi^5}{12}$ & $\frac{64\pi^5}{945}$ & $\frac{\pi^6}{60}$ & $\frac{128\pi^6}{10395}$ & $\frac{\pi^7}{360}$ & $\frac{256\pi^7}{135135}$ & $\frac{\pi^8}{2520}$ & $\frac{512\pi^8}{2027025}$ \\
\hline
\end{tabular}
\end{center}

Thus, for $\ds =1$, we have $\bels=\bec$; for $2\leq \ds \leq 10$, we have $\bels>\bec$; and  for $\ds > 10$, we have $\bels<\bec$. In particular, for $2\leq\ds\leq 10$, our \cref{thm:mainlim} yields global existence of solutions for initial data close in $L^2$ to $\muu$ for $\be \in (\bec,\bels)$. This is in sharp contrast to the behavior on $\R^2$, where all classical solutions with finite second moment blow up in finite time for $\beta>\bec$.

From \cref{thm:mainstab}, we know that $\beu\leq \min(\bec,\bels)$. We do not know the ordering between $\bei$ and $\beu$, but we know that $\min(\bels,\bec)$ is also an upper bound for $\bei$.  Thus, for dimension $\ds =1$, we have the ordering
\begin{align}
0<\bei \ {?} \ \beu\leq \bec=\bels < \infty.
\end{align}
For dimensions $2\leq \ds\leq 10$, we have the ordering
\begin{align}
0<\bei \ {?} \ \beu  \leq \bec < \bels < \infty.
\end{align}
While for dimensions $\ds>10$, we have the ordering
\begin{align}
0<\bei \ {?} \ \beu \leq \bels < \bec < \infty.
\end{align}
For the special case $\ds=2$, we can be explicit and sharp with respect to all the thresholds, except $\bei$.
\begin{center}
\begin{tabular}{|c c c c|}
\hline
$\bei$ & $\beu$ &  $\bec$ & $\bels$\\
\hline
? & $4$ & $4$  & $2\pi$\\
\hline
\end{tabular}

\end{center}

\subsection{Organization of paper}\label{ssec:introorg}
Let us briefly comment on the organization of the remaining body of the paper.

In \cref{sec:WP}, we study the basic local well-posedness and smoothing properties of the equation \eqref{eq:lima}. The main results are \cref{prop:LWPa}, \Cref{lem:hypequiv,lem:greg}. In \cref{sec:aLogFE}, we review the gradient flow structure of equation \eqref{eq:lima} and \emph{a priori} bounds provided by the free energy and dissipation functionals. The main results are \Cref{lem:FEcoer,lem:FEests,lem:DFI}. In \cref{sec:glob}, we combine the results of  \Cref{sec:WP,sec:aLogFE} to prove global existence and trend to equilibrium for the equation \eqref{eq:lima}. The main results are \Cref{lem:aLrbnd,lem:mutasyss,lem:nabLinfexp}, which together complete the proof of the global existence and trend to equilibrium portion of \cref{thm:mainlim}. In \cref{sec:Instab}, we study the role of the temperature in determining the linear and nonlinear stability  of the uniform distribution as a stationary solution of \eqref{eq:lima}. The main result is \cref{prop:NLinstab}, which together with the results of \cref{sec:glob} completes the proof of \cref{thm:mainlim}. In \cref{sec:Unq}, we study the role of temperature in determining the uniqueness of stationary solutions. The main results are \Cref{prop:ssnounq,prop:ssnounq2} (nonuniqueness) and \cref{prop:ssunq} (uniqueness), which together yield \cref{thm:mainstab}. In \cref{sec:MLHLS}, we consider the validity of the modulated log HLS inequality: \cref{ssec:MLHLSdprf} treats the counterexample for $\be>\min(\bels,\bec)$ and \cref{ssec:MLHLSprf} treats the proof for sufficiently small $\be \in (0,\min(\bels,\bec))$ (i.e., the proof that $\bei>0$). Together, this then yields \cref{thm:mainMLHLS}. Finally, in \cref{sec:MFE}, we present the main application to the modulated free energy method, giving a uniform-in-time bound for the relative entropy for $\beta<\bei$ and showing uniform-in-time entropic propagation of chaos is impossible if $\beta>\bels$. This then completes the proof of our last main result, \cref{thm:mainUT}.

\subsection{Acknowledgments}\label{ssec:introack}
This work was completed as a result of the first author's internship visit to the Courant Institute of Mathematical Sciences, NYU, which he thanks for its hospitality. The second author thanks Pierre-Emmanuel Jabin and Toan T. Nguyen for engaging discussion during a visit to Penn State that inspired several results in this paper. He also thanks Amir Moradifam for helpful discussion. 

\subsection{Notation}\label{ssec:intronot}
We close the introduction with the basic notation used throughout the article without further comment. We mostly follow the conventions of \cite{NRS2021,RS2021, CdCRS2023}.

Given nonnegative quantities $A$ and $B$, we write $A\lesssim B$ if there exists a constant $C>0$, independent of $A$ and $B$, such that $A\leq CB$. If $A \lesssim B$ and $B\lesssim A$, we write $A\sim B$. Throughout this paper, $C$ will be used to denote a generic constant which may change from line to line. Also, $(\cdot)_+$ denotes the positive part of a number, and $x\pm$ denotes $x\pm\ep$, for arbitrarily small $\ep>0$.

Given $N\in\N$ and points $x_{1,N},\ldots,x_{N,N}$ in some set $S$, $\ux_N=(x_{1,N},\ldots,x_{N,N})\in S^N$. Given $x\in\T^\ds$ and $r>0$, $B(x,r)$ and $\p B(x,r)$ respectively denote the ball and sphere centered at $x$ of radius $r$. Given a function $f$, we denote the support of $f$ by $\supp f$. The notation $\nabla^{\otimes k}f$ denotes the $k$-tensor field with components $(\p_{i_1\cdots i_k}^k f)_{1\leq i_1,\ldots,i_k\leq d}$.

$\P(\T^\ds)$ denotes the space of Borel probability measures on $\T^\ds$, with subscript $ac$ to denote absolutely continuous with respect to Lebesgue measure. $C(\T^\ds)$ denotes the Banach space of  continuous, bounded functions on $\T^\ds$ equipped with the uniform norm $\|\cdot\|_{\infty}$. $C^k(\T^\ds)$ denotes the Banach space of $k$-times continuously differentiable functions with bounded derivatives up to order $k$ equipped with the natural norm, and $C^\infty \coloneqq \bigcap_{k=1}^\infty C^k$. The subspace of smooth functions with compact support is denoted with a subscript $c$.

$\Dm=(-\D)^{-\frac12}$ denotes the Fourier multiplier with symbol $2\pi|k|$. Functions of $\Dm$ can be defined through the Fourier transform. For integers $n\in \N_0$ and exponents $1\leq p\leq\infty$, $W^{n,p}$ denotes the standard inhomogeneous Sobolev space. For general $\al\in \R$ and $1<p<\infty$, $W^{\al,p}$ denotes the Bessel potential space defined by
\begin{equation}
\left\{\mu \in \mathcal{D}'(\T^\ds) : \|(I-\D)^{\al/2}\mu\|_{L^p} < \infty\right\},
\end{equation}
i.e., the space of distributions $\mu$ such that $(I-\D)^{\al/2}\mu$ is an $L^p$ function. When $\al$ is a positive integer, then $W^{\al,p}$ coincides with the classical Sobolev space above. When $p=2$, we instead use the customary notation $H^\al$. A superscript $\dot{}$ indicates the corresponding homogeneous space.

\section{Well-posedness of the limiting equation}\label{sec:WP}
In this section, we study the well-posedness and regularity properties of the mean-field equation \eqref{eq:lima}. We show that for $L^1$ initial data, there exists a unique maximal lifespan local solution and instantaneously, this solution is $C^\infty$ on $(0,T_{\max})\times\T^\ds$. Later in \cref{sec:glob}, we show that these solutions are, in fact, global. Most of the results in this section are probably known on $\R^\ds$, at least for $\ds=2$; but comparable statements for the case of $\T^\ds$ do not appear present. In any case, our presentation perhaps differs from the existing literature.

\subsection{Periodic log potential and heat kernel}\label{ssec:WPhk}
Before proceeding to discuss well-posedness and regularity, we review some basic properties of $\g$ and the heat kernel $e^{t\D}$ on $\T^\ds$.

We recall from the introduction that $\g$ is the unique distributional solution to the equation
\begin{equation}
|\nabla|^{\ds}\g = \cdd(\d_0-1), \qquad x\in \T^\ds,
\end{equation}
subject to the constraint that $\int_{\T^\ds}\g=0$. Equivalently, $\g$ is the distribution with Fourier coefficients $\hat{\g}(k) = \cdd(2\pi|k|)^{-\ds}\indic_{k\neq 0}$ for $k\in\Z^\ds$. One can show that \cite{HSSS2017}
\begin{equation}\label{eq:ggE}
\g \in C^\infty(\T^\ds\setminus\{0\}) \quad \text{and} \quad \g+\log|x|\in C^\infty\paren*{B(0,\frac{1}{4})}.
\end{equation}
In particular,
\begin{equation}
\forall n\geq 1 , \ x\in\T^\ds\setminus\{0\}, \qquad |\nabla^{\otimes n}\g(x)| \lesssim_n |x|^{-n} + 1.
\end{equation}

\medskip
Let $e^{t\D}$ denote the Fourier multiplier on $\T^\ds$ with coefficients $(e^{-4\pi^2 t|k|^2})_{k\in\Z^\ds}$, and let $\K_t$ denote the convolution kernel of $e^{t\D}$. From the Fourier representation, one checks that $\K_t \in C^\infty(\T^\ds)$  and $\int_{\T^\ds}\K_t=1$, for every $t>0$. Furthermore, one can explicitly write $\K_t$ as the periodization of the Euclidean heat kernel \cite[Section 10.3]{BHS2019},
\begin{equation}
\K_t(x) = (4\pi t)^{-\ds/2}\sum_{n\in\Z^\ds} e^{-\frac{|x-n|^2}{4 t}}.
\end{equation}
Using these properties, one can show that if $m(D)$ is a Fourier multiplier, homogeneous of degree $\ka$, with symbol $m(k)\in C^\infty(\R^\ds\setminus\{0\})$, then
\begin{equation}
\|m(D)\paren*{\K_t-1}\|_{L^p} \lesssim_{n,\ds,p,m(D)} \min(t,1)^{-\frac{\ds}{2}\paren*{1-\frac{1}{p}}-\frac{\ka}{2}}e^{-C_{p,m(D)}\max(t,1)}.
\end{equation}
This implies that if $\int_{\T^\ds}\mu=0$, then
\begin{equation}\label{eq:mDhk}
\forall t>0 \qquad \|m(D) e^{t\D}\mu\|_{L^p} \lesssim_{n,\ds,p,q,m(D)} \|\mu\|_{L^q}\min(t,1)^{-\frac{\ds}{2}\paren*{\frac{1}{q}-\frac{1}{p}}-\frac{\ka}{2}}e^{-C_{p,m(D)}\max(t,1)}
\end{equation}
for any $1\leq q\leq p\leq \infty$.  We refer to \cite[Section 2]{CdCRS2023} for details. We will use \eqref{eq:mDhk}, called \emph{hypercontractivity}, in the remaining body of the paper without further comment.

\subsection{Local well-posedness}\label{ssec:WPloc}
We now show the local well-posedness of equation \eqref{eq:lima} using a fixed point argument (cf. \cite[Section 3.1]{CdCRS2023}) for the mild formulation of the equation,
\begin{equation}\label{eq:milda}
   \mu^t = e^{t\Delta/\beta}\mu^0 - \int_0^t e^{(t-\tau) \Delta/\beta} \div\paren*{\mu^\tau\nabla\g \ast \mu^\tau}d\tau.
\end{equation}
Although for our main application in \cref{sec:MFE}, we work with $W^{2,\infty}$ solutions, as it may be of independent use/interest, we prove here a more general local well-posedness result for initial data in $L^1$. If the initial data has additional integrability/regularity, then this persists along the evolution. We note that on $\R^\ds$, the $L^1$ norm is invariant under the scaling of equation \eqref{eq:lima},
\begin{equation}
\la>0, \qquad \mu^t(x) \mapsto \mu_\la^t(x) \coloneqq \lambda^\ds \mu^{\la^2 t}(\la x),
\end{equation}
and therefore $L^1$ is a \emph{critical} function space for the well-posedness of the equation.

We first record a lemma used in the proof of local well-posedness.
\begin{lemma}\label{lem:LWPlimsup}
Let $\ds\geq 1$, $\beta \in (0,\infty)$, and $\mu^0\in L^1(\T^\ds)$. Then
\begin{align}
\lim_{T\rightarrow 0^+} \sup_{0<t\leq T} (t/\beta)^{\frac14}\|e^{t\D/\beta}\mu^0\|_{L^{\frac{2\ds}{2\ds-1}}}=0.
\end{align}
\end{lemma}
\begin{proof}
Since $C^\infty(\T^\ds)$ is dense in $L^1(\T^\ds)$, given $\mu^0$ and $\ep>0$, let $\mu_\ep^0\in C^\infty(\T^\ds)$ such that $\|\mu^0-\mu_\ep^0\|_{L^1}<\ep$. By the triangle inequality,
\begin{align}
\left\|e^{t\D/\beta}\mu^0\right\|_{L^{\frac{2\ds}{2\ds-1}}} &\leq \left\|e^{t\D/\beta}\left(\mu^0-\mu_\ep^0\right)\right\|_{L^{\frac{2\ds}{2\ds-1}}}  + \left\|e^{t\D/\beta}\mu_\ep^0\right\|_{L^{\frac{2\ds}{2\ds-1}}} \nn\\
&\leq C\|\mu^0-\mu_\ep^0\|_{L^1}\min(t/\beta,1)^{-\frac14}e^{-C\max(t/\beta,1)} + \left\|\mu_\ep^0\right\|_{L^{\frac{2\ds}{2\ds-1}}} \nn\\
&<C\ep \min(t/\beta,1)^{-\frac14}e^{-C\max(t/\beta,1)}+ \left\|\mu_\ep^0\right\|_{L^{\frac{2\ds}{2\ds-1}}},
\end{align}
where the second line follows from applying \eqref{eq:mDhk} and Young's inequality to the first and second terms, respectively, of the first line. Hence,
\begin{align}
\lim_{T\rightarrow 0^+}\sup_{0<t\leq T} (t/\beta)^{\frac14}\|e^{t\D/\beta}\mu^0\|_{L^{\frac{2\ds}{2\ds-1}}} \leq C\ep + \lim_{T\rightarrow 0^+}  (T/\beta)^{\frac14}\left\|\mu_\ep^0\right\|_{L^{\frac{2\ds}{2\ds-1}}} =  C\ep.
\end{align}
Since $\ep>0$ was arbitrary, the proof is complete.
\end{proof}

Our main local well-posedness result is the following proposition.

\begin{prop}\label{prop:LWPa}
Let $\ds\geq 1$, $\beta<\infty$, and $\mu^0\in L^1(\T^\ds)$. There exists a constant $C>0$ depending only on $\ds$, such that if $T>0$ satisfies
\begin{equation}\label{eq:LWPaIC}
\sup_{0<t\leq T} (t/\beta)^{\frac14}\|e^{t\D/\beta}\mu^0\|_{L^{\frac{2\ds}{2\ds-1}}} \leq \frac{1}{C\beta},
\end{equation}
then there exists a unique solution $\mu$ to \eqref{eq:milda} in the class
\begin{equation}
\{\mu \in C_w([0,T], L^1(\T^\ds)) : \sup_{0< t\leq T}(t/\beta)^{\frac14}\|\mu^t\|_{L^{\frac{2\ds}{2\ds-1}}}<\infty\}
\end{equation}
and which satisfies
\begin{equation}
\sup_{0<t\leq T} (t/\beta)^{\frac14}\|\mu^t\|_{L^{\frac{2\ds}{2\ds-1}}} \leq C'\sup_{0<t\leq T} (t/\beta)^{\frac14}\|e^{t\D/\beta}\mu^0\|_{L^{\frac{2\ds}{2\ds-1}}},
\end{equation}
for some $C'>0$ depending only on $\ds$. Moreover, if $\mu_1, \mu_2$ are two solutions to \eqref{eq:milda} with initial data $\mu_1^0,\mu_2^0$, respectively, on $[0,T]$, then
\begin{equation}
\sup_{0<t\leq T} (t/\beta)^{\frac14} \|\mu_1^t-\mu_2^t\|_{L^{\frac{2\ds}{2\ds-1}}} \leq C'\left[\|\mu_1^0-\mu_2^0\|_{L^1} + \sup_{0<t\leq T}(t/\beta)^{\frac14}\|e^{\tau\D/\be}(\mu_1^0-\mu_2^0)\|_{L^{\frac{2\ds}{2\ds-1}}}\right].
\end{equation}
\end{prop}

\begin{proof}[Proof of \cref{prop:LWPa}]
Fix $\mu^0\in L^1(\T^\ds)$, and let $\Tc$ denote the operator defined by the right-hand side of \eqref{eq:milda}. For $T,R>0$ appropriately chosen, we aim to show that $\Tc$ is a contraction on the ball $B_R$ of the Banach space
\begin{equation}
\begin{split}
X \coloneqq \{\mu \in C_w([0,T], L^1(\T^\ds)) : \sup_{0<t\leq T} (t/\beta)^{1/4}\|\mu^t\|_{L^{\frac{2\ds}{2\ds-1}}} < \infty\}, \\
\|\mu\|_{X} \coloneqq \sup_{0<t\leq T}(t/\beta)^{1/4}\|\mu^t\|_{L^{\frac{2\ds}{2\ds-1}}}.
\end{split}
\end{equation}
For $t>0$, observe from the triangle/Minkowski's inequality plus the heat semigroup property \eqref{eq:mDhk},
\begin{align}
\|(\Tc\mu)^t\|_{L^{\frac{2\ds}{2\ds-1}}} &\leq \|e^{ t\D/\beta}\mu^0\|_{L^{\frac{2\ds}{2\ds-1}}} + \int_0^t \left\|e^{(t-\tau)\D/\beta}\div\paren*{\mu^\tau \nabla\g\ast\mu^\tau}\right\|_{L^{\frac{2\ds}{2\ds-1}}}d\tau \nn\\
&\leq \|e^{ t\D/\beta}\mu^0\|_{L^{\frac{2\ds}{2\ds-1}}}  + C\int_0^t \paren*{(t-\tau)/\beta}^{-\frac{3}{4}} \|\mu^\tau \nabla\g\ast\mu^\tau\|_{L^1} d\tau. \label{eq:muL2d2d-1}
\end{align}
By H\"older's inequality,
\begin{align}\label{eq:Holmunabgmu}
\|\mu^\tau \nabla\g\ast\mu^\tau\|_{L^1} \leq \|\mu^\tau\|_{L^{\frac{2\ds}{2\ds-1}}} \|\nabla\g\ast\mu^\tau\|_{L^{2\ds}} \lesssim \|\mu^\tau\|_{L^{\frac{2\ds}{2\ds-1}}}^2,
\end{align}
where if $\ds=1$, we have used that $\nabla\g\ast$ is the Hilbert transform, which is bounded on $L^2$, and if $\ds\geq 2$, then we have used the HLS lemma to obtain the final line. Therefore,
\begin{align}
\|(\Tc\mu)^t\|_{L^{\frac{2\ds}{2\ds-1}}} &\leq \|e^{ t\D/\beta}\mu^0\|_{L^{\frac{2\ds}{2\ds-1}}}  + C\int_0^t ((t-\tau)/\beta)^{-\frac34}(\tau/\beta)^{-\frac12} \paren*{\sup_{0< \tau'\leq t} (\tau'/\beta)^{\frac14}\|\mu^{\tau'}\|_{L^{\frac{2\ds}{2\ds-1}}}}^2d\tau \nn\\
&=\|e^{ t\D/\beta}\mu^0\|_{L^{\frac{2\ds}{2\ds-1}}} + C\beta(t/\beta)^{-\frac14}\int_0^1 (1-\tau)^{-\frac34}\tau^{-\frac12} \|\mu\|_{X}^2d\tau, \label{eq:fpgaincomp}
\end{align}
where the final line is by making the change of variable $\tau\mapsto \tau/t$. %Also, observe that for $t_2\geq t_1\geq 0$,
%\begin{multline}
%\left\|(\Tc\mu)^{t_2} - (\Tc\mu)^{t_1}\right\|_{L^1} \leq \|e^{t_2\Delta/\beta}\mu^0 - e^{t_1\Delta/\beta}\mu^0\|_{L^1} + \int_{t_1}^{t_2} \|e^{(t_2-\tau)\D/\be}\div(\mu^\tau\nabla\g\ast\mu^\tau)\|_{L^1}d\tau \\
%+ \int_0^{t_1} \|e^{(t_2-\tau)\D/\be}\div(\mu^\tau\nabla\g\ast\mu^\tau) - e^{(t_1-\tau)\D/\be}\div(\mu^\tau\nabla\g\ast\mu^\tau)\|d\tau
%\end{multline}
The preceding implies
\begin{align}
\|\Tc(\mu)\|_{X} \leq \sup_{0<t\leq T} (t/\be)^{\frac14}\|e^{ t\D/\beta}\mu^0\|_{L^{\frac{2\ds}{2\ds-1}}}  + C\beta\|\mu\|_{X}^2,
\end{align}
for some constant $C>0$ depending on $\ds$. Similar reasoning also shows that
\begin{align}
\|\Tc(\mu_1)-\Tc(\mu_2)\|_{X} \leq 2C\beta\paren*{\|\mu_1\|_{X} + \|\mu_2\|_{X}}\|\mu_1-\mu_2\|_{X}.
\end{align}

%Next, we claim that as $T\rightarrow 0^+$, $\sup_{0<t\leq T} (t/\beta)^{\frac14} \|e^{t\D/\beta}\mu^0\|_{L^{\frac{2\ds}{2\ds-1}}} \rightarrow 0$. Indeed, given any $\ep>0$, we may find $\varphi\in C_c^\infty(\T^\ds)$ with $\|\varphi-\mu^0\|_{L^1}\leq \ep$. Trivially,
%\begin{align}
%(t/\beta)^{\frac14} \|e^{t\D/\beta}\varphi\|_{L^{\frac{2\ds}{2\ds-1}}} \leq (t/\beta)^{\frac14} \|\varphi\|_{L^{\frac{2\ds}{2\ds-1}}}.
%\end{align}
%Hence by triangle inequality and using that $\|e^{ t\D/\beta}(\varphi-\mu^t)\|_{L^{\frac{2\ds}{2\ds-1}}} \lesssim (t/\beta)^{-\frac14} \|\varphi-\mu^t\|_{L^1}$,
%\begin{equation}	
%\limsup_{T\rightarrow 0^+}\sup_{0<t\leq T} (t/\beta)^{\frac14} \|e^{ t\D/\beta}\mu^0\|_{L^{\frac{2\ds}{2\ds-1}}} \lesssim\ep,
%\end{equation}
%which, since $\ep$ was arbitrary, implies the claim.

Let $R \coloneqq 2\sup_{0<t\leq T} (t/\be)^{\frac14} \|e^{t\D/\beta}\mu^0\|_{L^{\frac{2\ds}{2\ds-1}}}$, and appealing to \cref{lem:LWPlimsup}, suppose that $T$ is sufficiently small so that $4C\be R = \frac12$. Then we see that $\Tc(B_R) \subset B_R$ and
\begin{align}
\|\Tc(\mu_1) - \Tc(\mu_2)\|_{X} \leq \frac12\|\mu_1-\mu_2\|_{X},
\end{align}
showing that $\Tc$ is a contraction on $B_R$. By the Banach fixed point theorem, there is a unique solution $\mu$ of equation \eqref{eq:milda} in $X$. 

An analysis of the estimates above also shows that we have the Lipschitz continuity of the data-to-solution map: if $\mu_1,\mu_2$ are two solutions to \eqref{eq:milda} on $[0,T]$ with initial datum $\mu_1^0,\mu_2^0$, respectively, and such that $R:=\|\mu_i\|_{X} \leq \frac{1}{8C\be}$, then
\begin{equation}
\|\mu_1-\mu_2\|_{X} \leq 2C\left[\|\mu_1^0-\mu_2^0\|_{L^1} + \sup_{0<t\leq T}(t/\beta)^{\frac14}\|e^{\tau\D/\be}(\mu_1^0-\mu_2^0)\|_{L^{\frac{2\ds}{2\ds-1}}}\right].
\end{equation}
This completes the proof.
\end{proof}

\begin{remark}\label{rem:ablowup}
If $\mu$ is a solution to \eqref{eq:milda} with maximal lifespan $[0,T_{\max})$ for $T_{\max}<\infty$, then
\begin{equation}
\lim_{T\rightarrow T_{\max}^{-}} \sup_{0<t\leq T} (t/\beta)^{\frac14} \|\mu^t\|_{L^{\frac{2\ds}{2\ds-1}}} = \infty.
\end{equation}
Otherwise, the solution could be continued beyond $T_{\max}$ by applying \cref{prop:LWPa}.
%\st{
%there is an increasing sequence of times $t_n\rightarrow T_{\max}^{-}$ such that $\|\mu^{t_n}\|_{L^{\frac{2\ds}{2\ds-1}}} \leq \frac{C_0}{(T_{\max}/\beta)^{\frac 14}}$. By Young's inequality,
%\begin{align}
%\forall t\geq t_n,\qquad \|e^{(t-t_n)\D/\beta}\mu^{t_n}\|_{L^{\frac{2\ds}{2\ds-1}}} \leq \frac{C_0}{(T_{\max}/\beta)^{\frac 14}}.
%\end{align}
%Making the time translation, $\nu_n^\tau \coloneqq \mu^{t_n + \tau}$, let $T>0$ be such that
%\begin{equation}
%\sup_{0\leq \tau \leq T} (\tau/\beta)^{\frac14}\|e^{\tau\D/\beta}\nu_n^0\|_{L^{\frac{2\ds}{2\ds-1}}} \leq \frac{C_0 (T/\beta)^{\frac14}}{(T_{\max}/\beta)^{\frac 14}} = \frac{1}{C\beta},
%\end{equation}
%where $C$ is the same constant as in \eqref{eq:LWPaIC}. Having fixed $T$, we may choose $n$ sufficiently large so that $T_{\max}-t_n< T$. But then applying \cref{prop:LWPa} with initial datum $\nu_n^0$, we have extended the lifespan of the solution $\mu$ to $[0,t_{n}+T]\supset [0, T_{\max})$, contradicting the definition of $T_{\max}$. }
\end{remark}

\begin{remark}\label{rem:aclass}
From inspection of the proof of \cref{prop:LWPa}, we also see that for any integer $n\geq 0$, given $\mu^0\in W^{n,\infty}(\T^\ds)$, there is a unique solution $\mu \in W^{n,\infty}(\T^\ds)$ on some interval $[0,T]$. We can also construct a solution $\nu$ in the class $X$ introduced above with the same initial datum $\mu^0$ on some interval $[0,T']$. Letting $T_0\coloneqq \min(T,T')$, we see from uniqueness, that $\mu=\nu$. In particular, the Lipschitz continuity of the solution map implies that we can always approximate solutions in $X$ by classical solutions. This fact will be useful in the sequel to justify computations.
\end{remark}

\begin{remark}
Solutions obey conservation of mass and sign on their lifespans, implying that probability density initial data lead to probability density solutions. See \cite[Lemma 4.6]{CdCRS2023}.
\end{remark}

\begin{remark}\label{rem:timers}
By rescaling time, we may always normalize the mass to be {1} up to a change of temperature. More precisely, suppose that $\mu$ is a solution to \eqref{eq:lima}. Letting $\bar\mu = \int_{\T^\ds}\mu^0$, set $\nu^t \coloneqq \bmu^{-1}\mu^{t/\bmu}$. Then using the chain rule,
\begin{equation}
\p_t\nu^t = -\bmu^{-2}\div\paren*{\mu^{t/\bmu}\nabla\g\ast\mu^{t/\bmu} } + \frac{\bmu^{-2}}{\beta}\D\mu^{t/\bmu} = -\div\paren*{\nu^t\nabla\g\ast\nu^t} + \frac{1}{\tl\beta} \D\nu^t,
\end{equation}
where $\tl\beta \coloneqq \beta\bmu$. 
\end{remark}

\subsection{Gain of integrability}
The solutions constructed by \cref{prop:LWPa} have the  $L^1$-$L^{\frac{2\ds}{2\ds-1}}$ hypercontractive estimate
\begin{equation}
\forall t\in (0,T], \qquad \|\mu^t\|_{L^{\frac{2\ds}{2\ds-1}}} \lesssim t^{-\frac{1}{4}}\|\mu^0\|_{L^1}.
\end{equation}
This estimate, in fact, implies, for any $1\leq p\leq \infty$, the estimate
\begin{equation}
\forall t\in (0,T], \qquad \|\mu^t\|_{L^p} \lesssim t^{-\frac{\ds}{2}\left(1-\frac1p\right)} \|\mu^0\|_{L^1}.
\end{equation}
This has been shown on $\R^2$, for instance, in \cite[Section 3.2]{BM2014}; but their argument uses a self-similarity transformation that seems to have no analogue on the torus. Instead, we present an alternative proof---which would also work on $\R^\ds$---based on a simple iteration scheme.

\begin{lemma}\label{lem:hypequiv}
Let $\mu \in C_w([0,T_{\max}), L^1(\T^\ds))$ be a maximal lifespan solution to equation \eqref{eq:milda} in the sense of \cref{prop:LWPa}. Then
\begin{equation}\label{eq:hypeequiv}
\forall 1\leq p\leq\infty, \ 0<T<T_{\max}, \qquad \sup_{0<t\leq T} (t/\beta)^{\frac{\ds}{2}\left(1-\frac1p\right)} \|\mu^t\|_{L^p} < \infty.
\end{equation}
\end{lemma}
\begin{proof}
By \cref{rem:aclass}, we may assume without loss of generality that $\mu$ is a classical solution to \eqref{eq:milda}. Suppose that there is some $r_0\geq \frac{2\ds}{2\ds - 1}$ such that
\begin{equation}\label{eq:IHgainint}
\forall 1\leq r\leq r_0, \qquad \sup_{0<t\leq T} (t/\beta)^{\frac{\ds}{2}\left(1-\frac1r\right)}\|\mu^t\|_{L^r} < \infty.
\end{equation}
Arguing similarly as to in the proof of \eqref{eq:fpgaincomp}, we find for $1\leq p\leq q\leq\infty$,
\begin{align}
\|\mu^t\|_{L^{q}} &\lesssim (t/\beta)^{-\frac{\ds}{2}\left(1-\frac{1}{q}\right)} \|\mu^0\|_{L^1} + \int_0^t ((t-\tau)/\beta)^{-\frac{3}{4}-\frac{\ds}{2}\left(\frac1p-\frac{1}{q}\right)} \|\mu^\tau\nabla\g\ast\mu^\tau\|_{L^p}d\tau \nn\\
&\lesssim (t/\beta)^{-\frac{\ds}{2}\left(1-\frac{1}{q}\right)} \|\mu^0\|_{L^1} + \int_0^t ((t-\tau)/\beta)^{-\frac{3}{4}-\frac{\ds}{2}\left(\frac1p-\frac{1}{q}\right)}  \|\mu^\tau\|_{L^{r_1}} \|\mu^\tau\|_{L^{r_2}} d\tau,
\end{align}
where $\frac{1}{r_2} - (\frac1p - \frac{1}{r_1}) = \frac{\ds-1}{\ds}$. Choose $r_1=r_2=r_0$, which fixes $p$. Then we may choose $q$ arbitrarily large, subject to the constraint $\ds(\frac1p-\frac1q)<\frac12$. A little algebra shows that by doing so, we may take $q>r_0$. By iteration {on the value of $r_0$}, one sees that \eqref{eq:IHgainint} holds for any $r_0<\infty$.

To obtain the $L^\infty$ estimate $\sup_{0<t\leq T} (t/\beta)^{\frac{\ds}{2}} \|\mu^t\|_{L^\infty}<\infty$, we argue {using \eqref{eq:mDhk} that} for $q>\ds$,
\begin{align}
\|\mu^t\|_{L^\infty} &\lesssim (t/\beta)^{-\frac{\ds}{2}}\|\mu^0\|_{L^1} + \int_0^{\frac{t}{2}} ((t-\tau)/\beta)^{-\frac{\ds+1}{2}} \|\mu^\tau\nabla\g\ast\mu^\tau\|_{L^1}d\tau \nn\\
&\ph + \int_{\frac{t}{2}}^t ((t-\tau)/\beta)^{-\frac{1}{2}-\frac{\ds}{2q}} \|\mu^\tau\nabla\g\ast\mu^\tau\|_{L^q} d\tau \nn\\
&\lesssim (t/\beta)^{-\frac{\ds}{2}}\|\mu^0\|_{L^1} + \paren*{\sup_{0<\tau\leq \frac{t}{2}}(\tau/\beta)^{\frac14}\|\mu^\tau\|_{L^{\frac{2\ds}{2\ds-1}}}}^2 ( t/\beta)^{-\frac{\ds+1}{2}}\int_0^{\frac{t}{2}} (\tau/\beta)^{-\frac{1}{2}} d\tau \nn\\
&\ph + \int_{\frac{t}{2}}^t ((t-\tau)/\beta)^{-\frac{1}{2}-\frac{\ds}{2q}} \|\mu^\tau\nabla\g\ast\mu^\tau\|_{L^q} d\tau.
\end{align}
Making a change of variable $\tau \mapsto \tau/t$,
\begin{equation}
(t/\beta)^{-\frac{\ds+1}{2}}\int_0^{\frac{t}{2}}(\tau/\beta)^{-\frac12}d\tau = \beta^{\frac{\ds+2}{2}}t^{-\frac{\ds}{2}}\int_0^{\frac12}\tau^{-\frac12}d\tau.
\end{equation}
By H\"older's inequality and the HLS lemma ($\ds\geq 2$) or boundedness of Hilbert transform ($\ds=1$),
\begin{align}
\|\mu^\tau\nabla\g\ast\mu^\tau\|_{L^q} \leq \|\mu^\tau\|_{L^{r}} \|\nabla\g\ast\mu^\tau\|_{L^{r'}} \lesssim \|\mu^\tau\|_{L^{r}}^2,
\end{align}
provided we choose $r = \frac{2\ds q}{\ds q-q+\ds}$. Now,
\begin{align}
\int_{\frac{t}{2}}^t ((t-\tau)/\beta)^{-\frac{1}{2}-\frac{\ds}{2q}} \|\mu^\tau\|_{L^{r}}^2d\tau &\leq \paren*{\sup_{\frac{t}{2}\leq \tau\leq t} (\tau/\beta)^{\frac{\ds}{2}\left(1-\frac{1}{r}\right)} \|\mu^\tau\|_{L^{r}}}^2 \int_{\frac{t}{2}}^t  ((t-\tau)/\beta)^{-\frac{1}{2}-\frac{\ds}{2q}} (\tau/\beta)^{-\ds\left(1-\frac{1}{r}\right)}d\tau \nn\\
&\lesssim \paren*{\sup_{\frac{t}{2}\leq \tau\leq t} (\tau/\beta)^{\frac{\ds}{2}\left(1-\frac{1}{r}\right)} \|\mu^\tau\|_{L^{r}}}^2 \beta^{\frac{\ds+2}{2}} t^{-\frac{\ds}{2}}\int_{\frac12}^1  (1-\tau)^{-\frac12-\frac{\ds}{2q}} \tau^{-\frac{\ds+1}{2}+\frac{\ds}{2q}}d\tau.
\end{align}
After a little bookkeeping, we conclude that
\begin{multline}
\sup_{0<\tau \leq t}( t/\beta)^{\frac{\ds}{2}}\|\mu^t\|_{L^\infty} \lesssim \|\mu^0\|_{L^1} + \beta\paren*{\sup_{0<\tau\leq \frac{t}{2}}(\tau/\beta)^{\frac14}\|\mu^\tau\|_{L^{\frac{2\ds}{2\ds-1}}}}^2  \\
+\paren*{\sup_{\frac{t}{2}\leq \tau\leq t} (\tau/\beta)^{\frac{\ds}{2}\left(1-\frac{1}{r}\right)} \|\mu^\tau\|_{L^{r}}}^2 \beta\int_{\frac12}^1  (1-\tau)^{-\frac12-\frac{\ds}{2q}} \tau^{-\frac{\ds+1}{2}+\frac{\ds}{2q}}d\tau,
\end{multline}
which completes the proof.
\end{proof}

%With \cref{lem:hypequiv} in hand, an argument similar to the proof of \cref{lem:derdcys<} shows that if $\mu$ is a solution to \eqref{eq:milda} with maximal lifespan $[0,T_{\max})$, then for any $t_0>0$, $\mu \in C^\infty([t_0,T_{\max}) \times\T^\ds)$, \emph{a fortiori} $\mu$ is classical for positive times. We leave the proof of this assertion as an exercise to the reader.

\subsection{Gain of regularity}
We now show that the solutions belong to $C^\infty( (0,T_{\max})\times\T^\ds)$. In other words, solutions are instantaneously smooth. From \cref{lem:hypequiv}, we know that for any $0<T_1<T_2<T_{\max}$, a solution $\mu$ belongs to $C([T_1,T_2], L^\infty(\T^\ds))$. So by time translation, it suffices to show that if a solution $\mu \in C([0,T], L^\infty(\T^\ds))$, then $\mu \in C^\infty((0,T], \T^\ds)$. To show this, we use an induction argument similar to {our}  analysis in \cite[Section 4.1]{CdCRS2023}, which exploits the mild formulation of the equation and the smoothing properties of the heat kernel.

\begin{lemma}\label{lem:greg}
For each $n\in\N$, there exists a function $\W_{n}: [0,\infty)^2 \rightarrow [0,\infty)$ continuous, increasing in its arguments, and vanishing if any of its arguments are zero, such that the following holds. If $\mu$ is a smooth solution to equation \eqref{eq:lima} on $[0,T]$, then
\begin{align}
\forall t\in (0,\min(\beta, T)], \qquad \|\nabla^{\otimes n}\mu^t\|_{L^\infty} \leq (t/\beta)^{-\frac{n}{2}}\W_{n}(\beta,\sup_{0\leq \tau\leq t}\|\mu^\tau\|_{L^\infty}).
\end{align}
\end{lemma}

\begin{proof}
Let $\al = (\al_1,\ldots,\al_\ds) \in \N_0^\ds$ be a multi-index of order $|\al| = 1$.  The general case $|\al| \geq 1$ will be handled by induction. From the mild formulation, we see that
\begin{equation}\label{eq:mildLeibniz}
\p_{\al}\mu^t = e^{t\D/\beta}\p_{\al}\mu^0 +\int^t_0 e^{(t-\tau)\D/\beta} \div \p_\al \left(\mu^\tau \nabla\g\ast\mu^\tau  \right)d\tau.
\end{equation}
To deal with the singularity of the heat kernel at $\tau=t$, we divide the integration over $[0,t]$ into $[0,(1-\ep)t]$ and $[(1-\ep)t,t]$, for some $\ep \in (0,1)$ to be specified. Applying triangle and Minkowski's inequalities to the right hand side of \eqref{eq:mildLeibniz} gives 
\begin{multline}\label{eq:defJ123}
    \|\p_\al\mu^t\|_{L^\infty} \le \|e^{t\D/\beta}\p_\al\mu^0\|_{L^\infty} + \int_0^{t(1-\ep)} \left\| e^{(t-\tau)\D/\beta} \div \p_\al \left(\mu^\tau\nabla\g\ast\mu^\tau  \right)\right\|_{L^\infty} d\tau \\ 
    + \int_{t(1-\ep)}^t \left\| e^{(t-\tau)\D/\beta} \div \p_\al \left(\mu^\tau \nabla\g\ast\mu^\tau  \right)\right\|_{L^\infty} d\tau \eqqcolon J_1(t) + J_2(t) + J_3(t).
\end{multline}
%We respectively denote by $J_1(t), J_2(t), J_3(t)$ the three terms in the right hand side of the previous inequality and proceed to estimate each of them individually. 

$J_1(t)$ is straightforward consequence of heat kernel estimate \eqref{eq:mDhk} and Young's inequality:
\begin{equation}
\label{eq:J1estim}
    J_1(t) \lesssim \paren*{t/\beta}^{- \frac{1}{2}}\|\mu^0-1\|_{L^\infty},
\end{equation}
where we have implicitly used (and will subsequently use) that $t/\beta\leq 1$ by assumption.

Consider now $J_2(t)$. By \eqref{eq:mDhk} and H\"older's inequalities, we have for any $1\le r \leq\infty$,
\begin{align}
\left\| \p_\al e^{(t-\tau)\D/\beta} \div \left(\mu^\tau \nabla\g\ast\mu^\tau  \right)\right\|_{L^\infty} &\lesssim 	\paren*{(t-\tau)/\beta}^{-\frac{1}{2r}- 1}  \|\mu^\tau \nabla\g\ast\mu^\tau \|_{L^r} \nn \\
&\lesssim \paren*{(t-\tau)/\beta}^{-\frac{1}{2r}- 1}  \|\mu^\tau\|_{L^\infty} \| \nabla\g\ast\mu^\tau \|_{L^{r}}. \label{eq:J2t1}
\end{align}
If $\ds\geq 2$, then using that $\nabla\g\in L^1$, we may take $r=\infty$ and bound $\|\nabla\g\ast\mu^\tau\|_{L^\infty} \lesssim \|\mu^\tau-1\|_{L^\infty}$. If $\ds=1$, the operator $\nabla\g\ast$ is the Hilbert transform, which is unbounded on $L^\infty$. Instead, we take $1<r<\infty$, and bound $\|\nabla\g\ast\mu^\tau\|_{L^r} \lesssim \|\mu^\tau-1\|_{L^r}$. Thus, we have shown that for any $1<r<\infty$, 
\begin{multline}
\left\| \p_\al e^{(t-\tau)\D/\beta} \div \left(\mu^\tau \nabla\g\ast\mu^\tau  \right)\right\|_{L^\infty} \lesssim \paren*{(t-\tau)/\beta}^{- 1}\|\mu^\tau\|_{L^\infty}\\
    \times \Big(\|\mu^\tau-1 \|_{L^\infty}\indic_{\ds\geq 2}  +\paren*{(t-\tau)/\beta}^{-\frac{1}{2r}} \|\mu^\tau-1\|_{L^r} \indic_{\ds=1}\Big) .
\end{multline}
Recalling the definition of $J_2(t)$ from \eqref{eq:defJ123} and making the change of variable $\tau\mapsto \tau/t$, we arrive at
\begin{equation}\label{eq:J2estim}
J_2(t) \lesssim \beta{A_\ep} \sup_{0\leq\tau \leq t} \|\mu^\tau\|_{L^\infty}^2\left(\indic_{\ds\geq 2} + (t/\beta)^{-\frac{1}{2r}}\indic_{\ds=1}\right),
\end{equation}
where $1<r<\infty$ is arbitrary, and
\begin{equation}\label{eq:Aepdefs<bc}
    A_\ep \coloneqq \int_0^{1-\ep} (1-\tau)^{ -1} \left(\indic_{\ds \ge 2} + (1-\tau)^{-\frac{1}{2r}}\indic_{\ds=1}\right) d\tau.
\end{equation}

Finally, for $J_3(t)$, we have by \eqref{eq:mDhk}, product rule, and triangle inequality that for any $1\leq r\leq \infty$,
\begin{multline}
    \left\| \p_\al e^{(t-\tau)\D/\beta} \div \left(\mu^\tau \nabla\g\ast\mu^\tau  \right)\right\|_{L^\infty} \\
    \lesssim ((t-\tau)/\beta)^{-\hal -\frac{\ds}{2r}} \paren*{\|\mu^\tau \nabla\g\ast\p_\al\mu^\tau\|_{L^r} + \|\p_\al\mu^\tau \nabla\g\ast\mu^\tau\|_{L^r}} .
\end{multline}
If $\ds > 1$, then since $\nabla\g\in L^1$, we may take $r=\infty$ and crudely estimate
\begin{align}\label{eq:J3t0}
\|\mu^\tau \nabla\g\ast\p_\al\mu^\tau\|_{L^r} + \|\p_\al\mu^\tau \nabla\g\ast\mu^\tau\|_{L^r} \lesssim \|\mu^\tau\|_{L^\infty}\|\p_{\al}\mu^\tau\|_{L^\infty}.
\end{align}
If $\ds=1$, then we take $1< r< \infty$ and use H\"older's inequality and the boundedness of the Hilbert transform to estimate
\begin{align}
    \| \mu^\tau \nabla\g\ast\p_\al\mu^\tau  \|_{L^r} \lesssim \| \nabla\g\ast\p_\al\mu^\tau \|_{L^r} \| \mu^\tau \|_{L^\infty} \lesssim \|\p_\al\mu^\tau \|_{L^r}\| \mu^\tau \|_{L^\infty} \leq \|\p_\al\mu^\tau \|_{L^\infty}\| \mu^\tau \|_{L^\infty} \label{eq:J3t1}, \\
    \| \p_\al\mu^\tau \nabla\g\ast\mu^\tau  \|_{L^r} \le \| \p_\al\mu^\tau\|_{L^\infty} \|\nabla\g\ast\mu^\tau  \|_{L^r} \lesssim \| \p_\al\mu^\tau\|_{L^\infty} \|\mu^\tau -1\|_{L^{r}}.\label{eq:J3t2}
\end{align}
Combining the estimates  \eqref{eq:J3t0}, \eqref{eq:J3t1}, \eqref{eq:J3t2} and taking $r>\ds$ (so the integral in $\tau$ converges), we arrive at
\begin{equation}\label{eq:J3estim}
    J_3(t) \lesssim {\sup_{0\leq \tau\leq t}\|\mu^\tau\|_{L^\infty}} \int_{t(1-\ep)}^t  ((t-\tau)/\beta)^{-\hal}  \|\p_\al\mu^\tau \|_{L^\infty}\left( \indic_{\ds\geq 2} + \paren*{(t-\tau)/\beta}^{0-} \indic_{\ds=1}\right) d\tau.
\end{equation}
Combining the estimates \eqref{eq:J1estim}, \eqref{eq:J2estim}, \eqref{eq:J3estim}, we obtain
\begin{multline}\label{eq:prelemma}
	 \| \p_\al\mu^t \|_{L^\infty} \lesssim \|\mu^0-1\|_{L^\infty}(t/\beta)^{-\frac12} + \beta{A_\ep} \sup_{0\leq\tau \leq t} \|\mu^\tau\|_{L^\infty}^2\left(\indic_{\ds\geq 2} + (t/\beta)^{0-}\indic_{\ds=1}\right)\\	
	 +{\sup_{0\leq \tau\leq t}\|\mu^\tau\|_{L^\infty}} \int_{t(1-\ep)}^t  ((t-\tau)/\beta)^{-\hal}  \|\p_\al\mu^\tau \|_{L^\infty}\left( \indic_{\ds\ge 2} + \paren*{(t-\tau)/\beta}^{0-} \indic_{\ds=1}\right) d\tau,
\end{multline}
where $0- \coloneqq -\ep$, for $\ep>0$ arbitrarily small. To close the estimate for $\|\p_\al\mu^t\|_{L^\infty}$, we define
\begin{equation}
\forall 0<t\leq \min(\beta,T), \qquad \phi(t)\coloneqq \sup_{0<\tau\le t} (\tau/\beta)^{\frac{1}{2}} \| \p_\al\mu^\tau \|_{L^\infty}.
\end{equation}
Using this notation, we rearrange \eqref{eq:prelemma} and make the change of variable $\tau \mapsto \tau/t$ to obtain the inequality
\begin{multline}\label{eq:taddfac}
\phi(t) \le C\|\mu^0-1\|_{L^\infty} +  {C\beta A_\ep(t/\beta)^{\frac12}} \sup_{0\leq\tau \leq t} \|\mu^\tau\|_{L^\infty}^2\left(\indic_{\ds\geq 2} + (t/\beta)^{0-}\indic_{\ds=1}\right)\\
+{C\beta B_\ep(t/\beta)^\hal}\phi(t){\sup_{0\leq \tau\leq t}\|\mu^\tau\|_{L^\infty}} \left( \indic_{\ds\ge 2} + \paren*{t/\beta}^{0-} \indic_{\ds=1}\right),
\end{multline}
where $C>0$ depends only on $\ds$ and
\begin{equation}
B_\ep \coloneqq \int_{1-\ep}^1 (1-\tau)^{-\hal} \tau^{-\frac12}  \left( \indic_{\ds\ge 2} + (1-\tau)^{0-}\indic_{\ds=1}\right) d\tau.
\end{equation}
Since the integral in the definition of $B_\ep$ decreases monotonically to zero as $\ep\rightarrow 1^{-}$ and $t/\beta\leq 1$ by assumption, we may choose $\ep$ sufficiently small so that $C\beta B_\ep\sup_{0\leq\tau\leq t}\|\mu^\tau\|_{L^\infty} \le \frac12$. 
Thus, 
\begin{equation}\label{eq:estimstep2}
\phi(t) \le  2C\|\mu^0-1\|_{L^\infty} +{2C\beta A_\ep(t/\beta)^{\frac12}}\sup_{0\leq\tau \leq t} \|\mu^\tau\|_{L^\infty}^2\left(\indic_{\ds\ge 2} + (t/\beta)^{0-}\indic_{\ds=1}\right).
\end{equation}

\medskip 

Let us now induct from the case $|\alpha|=1$ to the case $n=|\al|\geq 1$. As our induction hypothesis, assume that
\begin{equation}\label{eq:s<ih}
    \forall |\ka|\le n-1, \ t\in (0,\min(\beta,T)], \qquad  \sup_{0<\tau\leq t} (\tau/\beta)^{\frac{|\ka|}{2}}\| \p_\ka\mu^\tau\|_{L^\infty} \leq \W_{k}(\beta, \sup_{0\leq \tau\leq t}\|\mu^\tau\|_{L^\infty})
\end{equation}
where $\W_{k}: [0,\infty)^2\rightarrow [0,\infty)$ is a continuous, increasing function, vanishing if any of its arguments are zero. %By shifting the initial time by an $\ep$, we may assume without loss of generality that $\sup_{0\leq \tau\leq \theta^{-1}}\|\p_\beta\mu^\tau\|_{L^p}<\infty$ for every $|\beta|\leq n-1$ and $1\leq p\leq \infty$.
Analogous to \eqref{eq:defJ123}, we have
\begin{multline}
    \|\p_\al\mu^t\|_{L^\infty} \le \|e^{t\D/\beta}\p_\al\mu^0\|_{L^\infty} + \int_0^{t(1-\ep)} \left\| e^{ (t-\tau)\D/\beta} \div \p_\al \left(\mu^\tau \nabla\g\ast\mu^\tau  \right)\right\|_{L^\infty} d\tau \\ 
    + \int_{t(1-\ep)}^t \left\| e^{(t-\tau)\D/\beta} \div \p_\al \left(\mu^\tau \nabla\g\ast\mu^\tau  \right)\right\|_{L^\infty} d\tau.
\end{multline}
Repeating the arguments for $J_1(t)$ and $J_2(t)$ above, we have 
\begin{align}
    J_1(t) \lesssim \paren*{ t/\beta}^{- \frac{n}{2}}\|\mu^0-1\|_{L^\infty},\label{eq:J1estimboot}
\end{align}
\begin{equation}
J_2(t) \lesssim  \frac{A_\ep\beta}{( t/\beta)^{\frac{n-1}{2}}} \sup_{0\leq\tau \leq t} \|\mu^\tau\|_{L^\infty}^2\left(\indic_{\ds\ge 2} + ( t/\beta)^{-\frac{1}{2r}}\indic_{\ds=1}\right) , \label{eq:J2estimboot}
\end{equation}
where now
\begin{equation}
A_\ep \coloneqq\int_0^{1-\ep} (1-\tau)^{ -\frac{n+1}{2}} \left(\indic_{\ds\ge 2} + (1-\tau)^{-\frac{1}{2r}}\indic_{\ds=1 }\right) d\tau.
\end{equation}
For $J_3$, we apply the Leibniz rule,
\begin{equation}
    \p_\al(\mu\nab\g\ast\mu) = \sum_{\ka\le\al} {{\al}\choose{\ka}} \p_\ka\mu\nab\g\ast\p_{\al-\ka}\mu,
\end{equation}
and note that estimates \eqref{eq:J3t0}, \eqref{eq:J3t1}, \eqref{eq:J3t2} also hold for the $\ka=0,\ka=\alpha$ terms. For the terms with $\ka\notin \{0,\alpha\}$, we use the induction hypothesis \eqref{eq:s<ih}. For $r>\ds$, we argue using H\"older's inequality and $\nabla\g\in L^1$ ($\ds\geq 2$) or the boundedness of the Hilbert transform ($\ds=1$),
\begin{align}
&\int_{t(1-\ep)}^t \left\| e^{(t-\tau)\D/\beta} \div\left(\p_\ka\mu^\tau\nabla\g\ast\p_{\al-\ka}\mu^\tau\right) \right\|_{L^\infty} d\tau  \nn\\
&\lesssim \int_{t(1-\ep)}^t((t-\tau)/\beta)^{-\hal -\frac{\ds}{2r}} \|\p_{\ka}\mu^\tau\|_{L^\infty} \|\p_{\al-\ka}\mu^\tau\|_{L^\infty}d\tau  \nn\\
&\leq \W_{|\ka|}(\beta,\sup_{0\leq \tau\leq t} \|\mu^\tau\|_{L^\infty})\W_{n-|\ka|}(\beta,\sup_{0\leq \tau\leq t} \|\mu^\tau\|_{L^\infty})\int_{t(1-\ep)}^{t}((t-\tau)/\beta)^{-\frac12-\frac{\ds}{2r}}(\tau/\beta)^{-\frac{n}{2}}d\tau \nn\\
&= \W_{|\ka|}(\beta,\sup_{0\leq \tau\leq t} \|\mu^\tau\|_{L^\infty})\W_{n-|\ka|}(\beta,\sup_{0\leq \tau\leq t} \|\mu^\tau\|_{L^\infty})\beta (t/\beta)^{-\frac{n}{2}}{( t/\beta)^{\frac12-\frac{\ds}{2r}}} \nn\\
&\ph\qquad\times\int_{1-\ep}^1 (1-\tau)^{-\frac12-\frac{\ds}{2r}}\tau^{-\frac{n}{2}}d\tau,
\end{align}
where we may take $r=\infty$ if $\ds\geq 2$ and arbitrary $r\in (\ds,\infty)$ if $\ds=1$. 

Combining the estimates above and following the same reasoning used to obtain \eqref{eq:estimstep2} in the case $|\al|=1$, one completes the proof of the induction step. Therefore, the proof of \cref{lem:greg} is complete.
\end{proof}

\section{Dissipation of free energy}\label{sec:aLogFE}
We recall from the introduction that the \emph{free energy} associated to equation \eqref{eq:lima} is
\begin{equation}\label{eq:FE}
\Ec_\beta(\mu) \coloneqq \frac{1}{\beta}\int_{\T^\ds}\log(\mu)d\mu - \frac{1}{2}\int_{(\T^\ds)^2}\g(x-y)d\mu^{\otimes 2}(x,y).
\end{equation}
Important questions {to understand}  are when the free energy is bounded from below and when it is coercive, in the sense that it controls the entropy. These questions are intimately linked with the sharp constant and extremals for the log HLS inequality (\cref{lem:logFI} below), first obtained independently by Carlen and Loss \cite{CL1992} and Beckner \cite{Beckner1993} on $\mathbb{S}^\ds, \R^\ds$, with an alternate proof later given in \cite{CCL2010}. There is also a log HLS inequality for compact manifolds \cite{SW2005}, which covers the flat torus $\T^\ds$, but we can always appeal to the Euclidean version by identifying measures on $\T^\ds$ as measures on $[-\frac{1}{2},\frac{1}{2}]^\ds \subset \R^\ds$.

\begin{lemma}\label{lem:logHLS}
Let $\ds\geq 1$. Then for any  $f\in \P_{ac}(\R^\ds)$ with $\int_{\R^\ds}\log(1+|x|^2)df(x)<\infty$, it holds that
\begin{equation}\label{eq:logHLS}
-\int_{(\R^\ds)^2}\log|x-y|df^{\otimes 2}(x,y) \leq \frac{1}{\ds}\int_{\R^\ds}\log(f)df + C_0,
\end{equation}
where
\begin{equation}\label{eq:C0logHLS}
C_0 \coloneqq \frac{\log \pi}{2}+\frac{1}{\ds}\log\paren*{\frac{\Gamma(\ds/2)}{\Gamma(\ds)} } + \frac{\psi(\ds)-\psi(\ds/2)}{2}
\end{equation}
and $\psi$ is the logarithmic derivative of the $\Gamma$ function.
\end{lemma}

Using \cref{lem:logHLS}, we obtain that the free energy $\Ec_\beta(\mu)$ is bounded from below provided $\beta\leq \bec={2\ds}$  and controls the relative entropy if $\beta<\bec$.

\begin{lemma}\label{lem:FEcoer}
Let $\ds\geq 1$ and $f\in \P_{ac}(\T^\ds)$ with $\int_{\T^\ds}\log(f)df<\infty$. There exists a constant $C>0$ depending only on $\ds$, such that
\begin{equation}\label{eq:FEcoer}
\paren*{\frac1\beta-\frac{1}{\bec}}\int_{\T^\ds}\log(f)df \leq \Ec_\beta(f) + C.
\end{equation}
Moreover, letting $\log_+(\cdot) \coloneqq \max(\log(\cdot),0)$,
\begin{equation}\label{eq:posentFE}
\paren*{\frac1\beta-\frac{1}{\bec}}\int_{\T^\ds}\log_+(f)df \leq C_1\paren*{\frac1\beta-\frac{1}{\bec}} + \paren*{\Ec_\beta(f) + C_2},
\end{equation}
for $C_1>0$ independent of $\ds$ and $C_2>0$ dependent on $\ds$.
\end{lemma}
\begin{proof}
Identifying $f$ {with} a probability density on $[-\frac12,\frac12]^\ds$, we have
\begin{multline}
\int_{(\T^\ds)^2}\g(x-y)df^{\otimes 2}(x,y) =- \int_{(\R^\ds)^2}\log|x-y|df^{\otimes 2}(x,y) \\
+  \int_{([-\frac12,\frac12]^\ds)^2}\paren*{\g(x-y)+\log|x-y|}df^{\otimes 2}(x,y).
\end{multline}
By bounding directly the integrand in the second line with \eqref{eq:ggE} and using that $f$ is a probability density,
\begin{align}
\int_{([-\frac12,\frac12]^\ds)^2}\left|\g(x-y)+\log|x-y|\right|df^{\otimes 2}(x,y) \leq C_{\ds}.
\end{align}
Therefore, using \cref{lem:logHLS},
\begin{align}
\int_{\T^\ds}\log(f)df &= \paren*{1-\frac{\beta}{\bec}}\int_{\T^\ds}\log(f)df + \frac{\beta}{\bec}\int_{\T^\ds}\log(f)df \nn\\
&\geq \paren*{1-\frac{\be}{\bec}}\int_{\T^\ds}\log(f)df -\frac{\be}{2}\int_{(\R^\ds)^2}\log|x-y|df^{\otimes 2}(x,y) - \frac{C_0\beta}{2}.
\end{align}
Combining estimates, we obtain
\begin{align}
\Ec_\beta(f) \geq \paren*{\frac1\beta-\frac{1}{\bec}}\int_{\T^\ds}\log(f)df - \frac{C_0}{2} - \frac{C_\ds}{2},
\end{align}
which, upon rearrangement and relabeling of constants, yields \eqref{eq:FEcoer}. %If $2d\theta>1$, then we may divide both sides by $(1-\frac{1}{2d\theta})$ to obtain an upper for the entropy in terms of the free energy and a dimension- and temperature-dependent constant.

Since $|x\log|x|| \leq C|x|^{1/2}$ for $|x|\leq 1$, we have the inequality
\begin{align}
\int_{\T^\ds}\log(f)df &\geq \int_{\T^\ds}\log_+(f)\indic_{f\geq 1} df - \int_{\T^\ds}|\log(f)|\indic_{f\leq 1}df \nn\\
&\geq \int_{\T^\ds}\log_+(f)\indic_{f\geq 1} df - C\int_{\T^\ds}f^{1/2}\indic_{f \leq 1} dx \nn\\
&\geq \int_{\T^\ds}\log_+(f) df - C. \label{eq:posEnt}
\end{align}
Combining this estimate with \eqref{eq:FEcoer} yields \eqref{eq:posentFE}.
\end{proof}

\begin{remark}\label{rem:nonpd}
If $f\geq 0$ is such that $\int_{\T^\ds}\log(f)df<\infty$ but that $\int_{\T^{\ds}}f \neq 1$, then letting $\bar{f}=\int_{\T^\ds}df$, we may apply estimate \eqref{eq:FEcoer} with $f/\bar{f}$ and rearrange to obtain
\begin{equation}
\frac{\paren*{\frac1\beta-\frac{1}{\bec}}}{\bar f}\int_{\T^\ds}\log(f)df \leq \frac{\Ec_{\beta\bar f}(f)}{\bar f} - \frac{\log \bar f}{\bec} + C.
\end{equation}
\end{remark}

In the case $\be>\bec$, the free energy is unbounded from below. This may be seen through a rescaling argument.

\begin{lemma}\label{lem:FEsc}
Let $\ds\geq 1$ and $\be>\bec$. Then the minimal free energy $\Ec_{\be}^* \coloneqq \inf_{\mu\in\P_{ac}(\T^\ds)} \Ec_\beta(\mu) = -\infty$.
\end{lemma}
\begin{proof}
Let $\varphi$ be a $C^\infty$ bump function on $\R^{\ds}$ supported on the ball $B(0,\frac18)$, such that $\int_{\R^\ds}\phi = 1$. Periodizing $\phi$, we may view it as an element of $C^\infty(\T^\ds)$. For $\la \in (0,1)$, set $\phi_\la \coloneqq \la^{-\ds}\phi(x/\la)$. Note that $\supp(\phi_\la)\subset B(0,\la/8)$ and $\int_{\T^\ds}\phi_\la = 1$. Now by making the changes of variable $x' = x/\la$ and $y'=y/\la$, we obtain
\begin{align}
\Ec_\beta(\phi_\la) &= \frac1\be\int_{\T^{\ds}}\log(\phi_\la)\phi_\la dx - \frac12\int_{(\T^\ds)^2}\g(x-y)\phi_\la(x)\phi_\la(y)dxdy \nn\\
&= \frac1\be\int_{\R^\ds}\log(\phi)\phi dx -\frac{\ds\log\la}{\be} - \frac12\int_{B(0,\frac18)^2}\g(\la(x-y))\phi(x)\phi(y)dxdy. \label{eq:FEunbdsc1}
\end{align}
Since for $x,y\in B(0,\frac18)$, $\la(x-y) \in B(0,\frac14)$, we may apply triangle inequality and \eqref{eq:ggE} to estimate
\begin{align}
 - \int_{B(0,\frac18)^2}\log|\la(x-y)|\phi(x)\phi(y)dxdy &\leq \int_{B(0,\frac18)^2} \left|\g(\la(x-y)) +\log|\la(x-y)|\right|\phi(x)\phi(y)dxdy\nn\\
&\ph + \int_{B(0,\frac18)^2}\g(\la(x-y))\phi(x)\phi(y)dxdy \nn\\
&\leq C_{\ds} + \int_{B(0,\frac18)^2}\g(\la(x-y))\phi(x)\phi(y)dxdy,
\end{align}
where $C_{\ds}>0$ is some constant depending only on $\ds$. Hence,
\begin{align}
\int_{B(0,\frac18)^2}\g(\la(x-y))\phi(x)\phi(y)dxdy \geq -\log\la - C_{\ds}- \int_{B(0,\frac18)^2}\log|x-y|\phi(x)\phi(y)dxdy.
\end{align}
Inserting this bound into \eqref{eq:FEunbdsc1}, we find
\begin{align}
\Ec_\beta(\phi_\la) \leq \frac1\be\int_{\R^\ds}\log(\phi)\phi dx -\frac{\ds\log\la}{\be} + \frac{\log \la}{2} + \frac{C_{\ds}}{2} + \frac12\int_{B(0,\frac18)^2}\log|x-y|\phi(x)\phi(y)dxdy.
\end{align}
Noting that $-\frac{\ds\log\la}{\be} + \frac{\log \la}{2} = \frac{\log\la}{2\be}(\be-\bec)$, we see that as $\la\rightarrow 0^+$, the preceding right-hand side tends to $-\infty$, which completes the proof.
\end{proof}

%\begin{remark}\label{rem:posEnt}
%\end{remark}

\medskip
It is known that the equation \eqref{eq:lima} is formally the gradient flow of the free energy \eqref{eq:FE} with respect to the 2-Wasserstein metric on the manifold of probability measures. And in fact, this gradient-flow structure has been used to produce solutions to the $\ds=2$ case of \eqref{eq:lima} on $\R^\ds$ (i.e., the PKS equation). We refer to \cite{Blanchet2013,Blanchet2013gf, BDP2006} for details. We do not need the full extent of this structure, but instead only that the free energy dissipates along the evolution of \eqref{eq:lima}. For this, we introduce the \emph{dissipation functional}
\begin{equation}\label{eq:FED}
\Dc_{\beta}(f) \coloneqq \int_{\T^\ds}\left|\frac1\beta\nabla\log f -\nabla\g\ast f \right|^2 df.
\end{equation}
Given any classical solution $\mu$ of \eqref{eq:lima}, a direct computation shows that
\begin{equation}\label{eq:FEid}
\forall t\geq 0, \qquad \Ec_\beta(\mu^t) + \beta\int_0^t \Dc_{\beta}(\mu^\tau)d\tau  = \Ec_\beta(\mu^0).
\end{equation}
Since $\Dc_{\beta}(\mu^\tau)\geq 0$, this relation implies that the free energy is nonincreasing. Combining the relation \eqref{eq:FEid} with \cref{lem:FEcoer}, we obtain the following useful \emph{a priori} estimate ({whose proof is immediate}).

\begin{lemma}\label{lem:FEests}
Let $\ds\geq 1$, $\beta<\bec$, and let $\mu$ be a probability density solution to equation \eqref{eq:lima} with maximal lifespan $[0,T_{\max})$. Then there exist constants $C_1>0$ independent of $\ds$ and $C_2>0$ dependent on $\ds$, such that
\begin{equation}
\forall t\in [0,T_{\max}), \qquad \int_{\T^\ds}\log_+(\mu^t)d\mu^t \leq C_1 + \paren*{\frac1\beta-\frac{1}{\bec}}^{-1}\paren*{\Ec_\beta(\mu^0) -\beta\int_0^t\Dc_{\beta}(\mu^\tau)d\tau + C_2}.
\end{equation}
\end{lemma}
%\begin{proof}
%The desired estimate is immediate from combining \eqref{eq:FEid} and \eqref{eq:posentFE}.
%\end{proof}

\begin{remark}\label{rem:posentFED}
Since the solutions given by \cref{prop:LWPa} are automatically classical for positive times, by time translation, we see that if $\mu \in C_w([0,T_{\max}), L^1(\T^\ds))$ is a maximal lifespan solution, then for any $0<t_0<T_{\max}$, \cref{lem:FEests} implies that
\begin{equation}\label{eq:posentFED}
\forall t\in [t_0,T_{\max}), \qquad \int_{\T^\ds}\log_+(\mu^t)d\mu^t \leq C_1 + \paren*{\frac1\beta-\frac{1}{\bec}}^{-1}\paren*{\Ec_\beta(\mu^{t_0}) -\beta\int_{t_0}^t\Dc_{\beta}(\mu^\tau)d\tau + C_2}.
\end{equation}
\end{remark}

For $f\in\P_{ac}(\T^\ds)$ with $\nabla\log f\in L^2(df)$, let us introduce the \emph{Fisher information}
\begin{equation}\label{eq:Finfdef}
I(f) \coloneqq \int_{\T^\ds}\frac{|\nabla f|^2}{f}dx = \int_{\T^\ds}|\nabla\log f|^2 df.
\end{equation}
At least when $\ds=2$, it is known that the entropy, mass, and dissipation functional control the Fisher information. This is still true for any $\ds\geq 1$. Before we prove this, we record some useful estimates involving the Fisher information which are consequences of Sobolev embedding and interpolation (cf. \cite[Lemma 2.1]{FM2016}).

\begin{lemma}\label{lem:FIests}
For $\ds\geq 1$ and $1\leq q<2$, with $q\leq \frac{\ds}{\ds-1}$ if $\ds\geq 3$, it holds for $f\in \P_{ac}(\T^\ds)$ that
\begin{equation}\label{eq:FIgrad}
 \|\nabla f\|_{L^q} \lesssim I(f)^{\frac{1}{2-\theta_q}} \|f-1\|_{L^1}^{\frac{1-\theta_q}{2-\theta_q}} + I(f)^{\frac{1}{2}}, \qquad \text{where} \ \theta_q \coloneqq \frac{2(q-1)}{q(1-(\frac{1}{q}-\frac1\ds))}.
\end{equation}
%Consequently, if
%\begin{equation}
%\begin{cases} {1\leq p\leq \infty}, & {d=1} \\ {1\leq p<\infty}, & {d=2} \\ 1\leq p \leq \frac{d}{d-2}, & {d\geq 3}  \end{cases}
%\end{equation}
%then
%\begin{equation}
%\|f-\bar f\|_{L^p} \lesssim 
%\end{equation}
We also have for any $\frac{\ds-2}{2}\leq r<\infty$,
\begin{equation}\label{eq:fLr+1}
\|f-1\|_{L^{r+1}}^{r+1} \lesssim \|f-1\|_{L^1}^{\frac{2(r+1)-\ds}{2+\ds(r-1)}} \| \nabla |f-1|^{\frac{r}{2}} \|_{L^{2}}^{\frac{2\ds r}{2+\ds(r-1)}},
\end{equation}
and if $\ds\geq 3$ and 
\begin{equation}\label{eq:fLrHLS}
\int_{\T^\ds} |f-1|^{2} |\Dm^{2-\ds}(f-1)| dx  \lesssim \|f-1\|_{L^1} \| \nabla f\|_{L^{2}}^2.
\end{equation}
\end{lemma}
\begin{proof}
Suppose that $1\leq q<2$. Then using H\"older's inequality with conjugate exponents $\frac{2}{q}$ and $\frac{2}{2-q}$,
\begin{align}
\int_{\T^\ds}|\nabla f|^qdx = \int_{\T^\ds}\left|\frac{\nabla f}{\sqrt{|f|}}\right|^q |f|^{\frac{q}{2}} dx \leq \paren*{\int_{\T^\ds}\frac{|\nabla f|^2}{|f|}}^{\frac{q}{2}} \paren*{\int_{\T^\ds} |f|^{\frac{q}{2-q}}dx }^{\frac{2-q}{2}} = I(f)^{\frac{q}{2}} \|f\|_{L^{\frac{q}{2-q}}}^{\frac{q}{2}}. \label{eq:Ifpresub}
\end{align}
By Gagliardo-Nirenberg interpolation,
\begin{align}
\|f-1\|_{L^{\frac{q}{2-q}}} \lesssim \|f-1\|_{L^1}^{1-\theta} \|\nabla f\|_{L^q}^\theta,
\end{align}
where $\theta = \frac{2(q-1)}{q(1-(\frac{1}{q}-\frac1\ds))}$. If $\ds\geq 3$, then in this last step, we also need to assume $q\leq \frac{\ds}{\ds-1}$. Using the triangle inequality and inserting the preceding estimate into the right-hand side of \eqref{eq:Ifpresub}, we find
\begin{align}
\int_{\T^\ds}|\nabla f|^qdx &\lesssim I(f)^{\frac{q}{2}}\paren*{\|f-1\|_{L^1}^{1-\theta}\|\nabla f\|_{L^q}^{\theta} + 1 }^{\frac{q}{2}}\nn\\
& \leq I(f)^{\frac{q}{2}}\paren*{\paren*{\|f-1\|_{L^1}^{1-\theta}\|\nabla f\|_{L^q}^{\theta}}^{\frac{q}{2}} + 1}. \label{eq:IfpreYoung}
\end{align}
Using Young's product inequality,
\begin{align}
I(f)^{\frac{q}{2}}\paren*{\|f-1\|_{L^1}^{1-\theta}\|\nabla f\|_{L^q}^{\theta}}^{\frac{q}{2}} \leq \paren*{2I(f)^{\frac{q}{2}}\|f-1\|_{L^1}^{\frac{q(1-\theta)}{2}} }^{\frac{2}{2-\theta}} + 2^{-\frac{2}{\theta}} \|\nabla f\|_{L^q}^q .
\end{align}
Inserting this estimate into the right-hand side of \eqref{eq:IfpreYoung} and rearranging the resulting inequality, we find
\begin{align}
\|\nabla f\|_{L^q} \lesssim I(f)^{\frac{1}{2-\theta}} \|f-1\|_{L^1}^{\frac{1-\theta}{2-\theta}} + I(f)^{\frac{1}{2}}.
\end{align}
Substituting in the value of $\theta$ now yields the desired \eqref{eq:FIgrad}.

By Gagliardo-Nirenberg interpolation,
\begin{align}
\|f-1\|_{L^{r+1}} = \| |f-1|^{\frac{r}{2}}\|_{L^{\frac{2(r+1)}{r}}}^{\frac{2}{r}} \lesssim \paren*{\| |f-1|^{\frac{r}{2}}\|_{L^2}^{1-\theta} \|\nabla |f-1|^{\frac{r}{2}}\|_{L^2}^\theta }^{\frac{2}{r}}, \label{eq:GNLr1psub}
\end{align}
where
\begin{align}
\frac{r}{2(r+1)} = \theta\paren*{\frac12 - \frac1\ds} + \frac{1-\theta}{2} \Leftrightarrow \theta = \frac{\ds}{2(r+1)}.
\end{align}
Next, by H\"older's inequality,
\begin{align}
\| |f-1|^{\frac{r}{2}}\|_{L^2}^{\frac{2}{r}} = \| f-1\|_{L^r} \leq \|f-1\|_{L^1}^{\ga} \|f-1\|_{L^{r+1}}^{1-\ga}
\end{align}
where
\begin{align}
\frac{1}{r} = \ga + \frac{1-\ga}{r+1} \Leftrightarrow \ga = \frac{1}{r^2}.
\end{align}
Inserting the preceding inequality into the right-hand side \eqref{eq:GNLr1psub} and rearranging, we find
\begin{align}
\|f-1\|_{L^{r+1}} \lesssim \paren*{\|f-1\|_{L^1}^{(1-\theta)\ga} \|\nabla |f-1|^{\frac{r}{2}}\|_{L^{2}}^{\frac{2\theta}{r}}}^{\frac{1}{1-(1-\theta)(1-\ga)}}.
\end{align}
Now $1-\theta = \frac{2(r+1)-\ds}{2(r+1)}$ and $1-\ga = \frac{r^2-1}{r^2}$, therefore
\begin{align}
1-(1-\theta)(1-\ga) = 1-\frac{2(r+1)-\ds}{2}\cdot \frac{r-1}{r^2} = \frac{2r^2-2(r^2-1)+\ds(r-1)}{2r^2} = \frac{2+\ds(r-1)}{2r^2}.
\end{align}
Hence,
\begin{align}
\frac{2\theta}{r\paren*{1-(1-\theta)(1-\ga)}} = \frac{2\ds r}{(r+1)(2+\ds(r-1))} , \\
\frac{\ga(1-\theta)}{\paren*{1-(1-\theta)(1-\ga)}} = \frac{(2(r+1)-\ds)}{(r+1)(2+\ds(r-1))},
\end{align}
which completes the proof of the desired estimate.

Finally, if $\ds\geq 3$, then we may use H\"older's inequality to estimate
\begin{align}\label{eq:f1rpreHLS}
\int_{\T^\ds} |f-1|^{r} |\Dm^{2-\ds}(f-1)| dx \leq \|f-1\|_{L^{rp}}^r \|\Dm^{2-\ds}(f-1)\|_{L^{\frac{p}{p-1}}}.
\end{align}
Let $q$ be the HLS conjugate to $\frac{p}{p-1}$, that is
\begin{equation}
\frac{\ds-2}{\ds} = \frac{1}{q} - \frac{p-1}{p},
\end{equation}
and we require that $p \in (1,\frac{\ds}{\ds-2})$, equivalently $\frac{p}{p-1} \in (\frac{\ds}{2}, \infty)$, which allows for any $q\in (1,\frac{\ds}{\ds-2})$. Applying the HLS lemma,
\begin{align}
\|\Dm^{2-\ds}(f-1)\|_{L^{\frac{p}{p-1}}} \lesssim \|f-1\|_{L^q}.
\end{align}
Now by Gagliardo-Nirenberg interpolation,
\begin{align}\label{eq:f1Lrp}
\|f-1\|_{L^{rp}} = \| |f-1|^{\frac{r}{2}}\|_{L^{2p}}^{2/r} \lesssim \Big(\| |f-1|^{\frac{r}{2}}\|_{L^1}^{1-\theta_p} \|\nabla |f-1|^{\frac{r}{2}}\|_{L^2}^{\theta_p}\Big)^{2/r} 
\end{align}
and
\begin{align}\label{eq:f1Lq}
\|f-1\|_{L^q} = \| |f-1|^{\frac{r}{2}}\|_{L^{\frac{2q}{r}}}^{2/r} \lesssim \Big(\| |f-1|^{\frac{r}{2}}\|_{L^1}^{1-\theta_q}\|\nabla |f-1|^{\frac{r}{2}}\|_{L^2}^{\theta_q}  \Big)^{2/r},
\end{align}
where $\theta_q,\theta_p$ are defined by
\begin{align}
\frac{1}{2p} = \theta_p(\frac12-\frac1\ds) + (1-\theta_p),  \qquad \theta_p = \frac{2\ds}{\ds+2}(1-\frac{1}{2p})\\
\frac{r}{2q} = \theta_q(\frac12-\frac1\ds) + (1-\theta_q), \qquad \theta_q = \frac{2\ds}{\ds+2}(1-\frac{r}{2q}).
\end{align}
Note that since $r<\frac{2\ds}{\ds-2}$, we may choose $p$ sufficiently close to $1$, so that $q$ is sufficiently close to $\frac{\ds}{\ds-2}$ to imply $\frac{2q}{r}>1$, i.e. $\frac{r}{2}<q$. By H\"older's inequality,
\begin{align}
\| |f-1|^{\frac{r}{2}}\|_{L^1} = \|f-1\|_{L^{\frac{r}{2}}}^{r/2} \leq \Big(\|f-1\|_{L^1}^{\vartheta_p} \|f-1\|_{L^{rp}}^{1-\vartheta_p}\Big)^{r/2}, \label{eq:f1L1r21}\\
\| |f-1|^{\frac{r}{2}}\|_{L^1} = \|f-1\|_{L^{\frac{r}{2}}}^{r/2} \leq \Big(\|f-1\|_{L^1}^{\vartheta_q} \|f-1\|_{L^q}^{1-\vartheta_q}\Big)^{r/2}, \label{eq:f1L1r22}
\end{align}
where
\begin{align}
\vartheta_p + \frac{1-\vartheta_p}{rp} = \frac{2}{r}, \qquad \vartheta_p =\frac{2p-1}{rp-1} \\
\vartheta_q  + \frac{1-\vartheta_q}{q} = \frac{2}{r}, \qquad \vartheta_q = \frac{2q-r}{r(q-1)}
\end{align}
with the convention that $\vartheta_p,\vartheta_q=1$ if $r\leq 2$. Combining the estimates \eqref{eq:f1Lrp}, \eqref{eq:f1Lq}, \eqref{eq:f1L1r21}, \eqref{eq:f1L1r22}, we obtain
\begin{align}
\|f-1\|_{L^{rp}} \lesssim \|f-1\|_{L^1}^{\vartheta_p(1-\theta_p)} \|f-1\|_{L^{rp}}^{(1-\theta_p)(1-\vartheta_p)} \|\nabla|f-1|^{\frac{r}{2}}\|_{L^2}^{2\theta_p/r}, \\
\|f-1\|_{L^q} \lesssim \|f-1\|_{L^1}^{\vartheta_q(1-\theta_q)} \|f-1\|_{L^{rp}}^{(1-\theta_q)(1-\vartheta_q)} \|\nabla|f-1|^{\frac{r}{2}}\|_{L^2}^{2\theta_q/r}.
\end{align}
Rearranging, we get
\begin{align}
\|f-1\|_{L^{rp}} \lesssim  \|f-1\|_{L^1}^{\frac{\vartheta_p(1-\theta_p)}{1-(1-\theta_p)(1-\vartheta_p)}} \|\nabla|f-1|^{\frac{r}{2}}\|_{L^2}^{\frac{2\theta_p}{r(1-(1-\theta_p)(1-\vartheta_p))}}, \\
\|f-1\|_{L^{q}} \lesssim \|f-1\|_{L^1}^{\frac{\vartheta_q(1-\theta_q)}{1-(1-\theta_q)(1-\vartheta_q)}} \|\nabla|f-1|^{\frac{r}{2}}\|_{L^2}^{\frac{2\theta_q}{r(1-(1-\theta_q)(1-\vartheta_q))}}.
\end{align}
Recalling our starting point \eqref{eq:f1rpreHLS}, we have shown that
\begin{multline}
\int_{\T^\ds} |f-1|^{r} |\Dm^{2-\ds}(f-1)| dx  \leq  \|f-1\|_{L^1}^{\frac{r\vartheta_p(1-\theta_p)}{1-(1-\theta_p)(1-\vartheta_p)} + \frac{\vartheta_q(1-\theta_q)}{1-(1-\theta_q)(1-\vartheta_q)}} \\
\times \|\nabla|f-1|^{\frac{r}{2}}\|_{L^2}^{\frac{2\theta_p}{(1-(1-\theta_p)(1-\vartheta_p))} + \frac{2\theta_q}{r(1-(1-\theta_q)(1-\vartheta_q))}}.
\end{multline}

In particular, if $r=2$, then $\vartheta_p=\vartheta_q=1$, and therefore,
\begin{align}
\frac{r\vartheta_p(1-\theta_p)}{1-(1-\theta_p)(1-\vartheta_p)} + \frac{\vartheta_q(1-\theta_q)}{1-(1-\theta_q)(1-\vartheta_q)} &= 2(1-\theta_p) + (1-\theta_q) \nn\\
&= 3 - \frac{2\ds}{\ds+2}\Big[2-\frac1p + 1-\frac1q\Big] \nn\\
&= 3 - \frac{2\ds}{\ds+2}\Big[1 +\frac{2}{\ds}\Big] \nn\\
&=1,
\end{align}
where we have used that $1-\frac1q = \frac{2}{\ds}-1 + \frac1p$ by HLS conjugacy. Similarly,
\begin{align}
{\frac{2\theta_p}{(1-(1-\theta_p)(1-\vartheta_p))} + \frac{2\theta_q}{r(1-(1-\theta_q)(1-\vartheta_q))}} =2\theta_p + \theta_q = 2.
\end{align}
Since $|\nabla|f-1|| \leq |\nabla(f-1)| = |\nabla f|$ (e.g., see \cite[Theorem 6.17]{LL2001}), the desired conclusion follows.

\begin{comment}
We choose $p$ so that $rp=q$, which implies the relation
\begin{equation}
\frac{\ds-2}{\ds} = \frac{1}{rp} - \frac{p-1}{p} \Leftrightarrow \frac{2(\ds-1)}{\ds} = \frac{r+1}{rp} \Leftrightarrow p = \frac{\ds(r+1)}{2r(\ds-1)}.
\end{equation}
Therefore, by the HLS lemma,
\begin{align}
\|f-1\|_{L^{rp}}^r \|\Dm^{2-\ds}(f-1)\|_{L^{\frac{p}{p-1}}} \lesssim \|f-1\|_{L^{rp}}^{r+1} = \|f-1\|_{L^{\frac{\ds(r+1)}{2(\ds-1)}}}^{r+1}.
\end{align}
Now by H\"older's inequality,
\begin{align}
\|f-1\|_{L^{\frac{\ds(r+1)}{2(\ds-1)}}} \leq \|f-1\|_{L^1}^{\vartheta} \|f-1\|_{L^{r+1}}^{1-\vartheta},
\end{align}
where
\begin{equation}
\frac{2(\ds+1)}{\ds(r+1)} = \vartheta + \frac{(1-\vartheta)}{r+1} \Leftrightarrow \frac{\ds-2}{\ds r} = \vartheta.
\end{equation}
Since $1-\vartheta = \frac{2+\ds(r-1)}{\ds r}$, estimate \eqref{eq:fLr+1} implies
\begin{align}
\|f-1\|_{L^{r+1}}^{(r+1)(1-\vartheta)} \lesssim \|f-1\|_{L^1}^{\frac{2(r+1)}{2+\ds(r-1)}} \| \nabla |f-1|^{\frac{r}{2}} \|_{L^{2}}^{\frac{2\ds r}{2+\ds(r-1)}} \lesssim \|f-1\|_{L^1}^{\frac{2(r+1)}{\ds r}} \| \nabla |f-1|^{\frac{r}{2}} \|_{L^{2}}^2.
\end{align}
After a little bookkeeping, we arrive at the desired estimate \eqref{eq:fLrHLS}.
\end{comment}
\end{proof}

We now use \cref{lem:FIests} to show that the Fisher information is controlled by the entropy and dissipation functional.

\begin{lemma}\label{lem:DFI}
There exists a constant $C>0$ depending only on $\ds$, such that for $f\in\P_{ac}(\T^\ds)$, it holds that if $\ds=1$, then
\begin{equation}\label{eq:DFI1d}
\frac{1}{\be^2} I(f) \leq C\Big[\Dc_{\beta}(f) + {\paren*{1+ \be^6}}\exp\Big({C}\Big(1+{\beta^4}\int_{\T^\ds}\log_+(f)df\Big)\Big)\Big],
\end{equation}
and if $\ds\geq 2$, then
\begin{equation}\label{eq:DFI2d+}
\frac{1}{\be^2} I(f) \leq C\Big[\Dc_{\beta}(f) + {\frac1\beta}\exp\Big({C}\Big(1+{\beta^2}\int_{\T^\ds}\log_+(f)df\Big)\Big)\Big].
\end{equation}
\end{lemma}
\begin{proof}
Expanding the square,
\begin{align}
\Dc_{\beta}(f) &= \int_{\T^\ds}\left|\frac1\beta\nabla\log f -\nabla\g\ast f\right|^2 df \nn\\
&= \frac1{\beta^2}\int_{\T^\ds}|\nabla \log f|^2 df - \frac2\beta\int_{\T^\ds}\paren*{\nabla\log f\cdot \nabla\g\ast f}df + \int_{\T^\ds}|\nabla\g\ast f|^2 df \nn\\
&\geq \frac1{\beta^2} I(f) - \frac{2\cdd}{\beta}\int_{\T^\ds}\Dm^{2-\ds}(f) df ,
\end{align}
where the final line follows from integrating by parts the second term and discarding the third term on the second line. We need to control the modulus of the second term in the final line, which we do by considering cases of $\ds$.

\medskip
Suppose $\ds=2$. Then for $A>1$, writing $f=f\indic_{f\leq A} + f\indic_{f>A}$ and using interpolation together with Sobolev embedding and the estimate \eqref{eq:FIgrad} from \cref{lem:FIests} with $q=\frac65$ and $\theta_q = \frac12$, we see that
\begin{align}
\int_{\T^\ds}\left|\Dm^{2-\ds}(f)\right| df \leq \|f\|_{L^2}^2 &\leq 2\left[\|f\indic_{f>A}\|_{L^2}^2 + A^2\right] \nn\\
&\leq 2\left[\|f\indic_{f>A}\|_{L^1}^{\frac12} \|f\|_{L^3}^{\frac32} + A^2\right]\nn\\
&\leq C \paren*{\frac{1}{\log A}\int_{\T^2}\log_+(f)df}^{\frac12} \paren*{\|\nabla f\|_{L^{\frac{6}{5}}}^{\frac32} + 1} + A^2\nn\\
&\leq C\paren*{\frac{1}{\log A}\int_{\T^2}\log_+(f)df}^{\frac12} \paren*{I(f) + 1} + A^2. \label{eq:DFlb}
\end{align}
We choose $A$ according to
\begin{equation}
(\log A)^{\frac12} = { 32\cdd C}\paren*{1+{\be}\paren*{\int_{\T^2}\log_+(f)df}^{\frac12}} \Rightarrow A \leq  {\exp\paren*{{C'}\paren*{1+\beta^2\int_{\T^2}\log_+(f)df}}},
\end{equation}
where $C'>0$ depends on $\ds$. With this choice for $A$, we find
\begin{equation}
\frac{2\cdd}{\be}\int_{\T^2}\left|\Dm^{2-\ds}(f)\right|  df \leq \frac{1}{2\be^2}I(f) + \frac1\be\exp\paren*{{C''}{\paren*{1+\be^2\int_{\T^2}\log_+(f)df}}} ,
\end{equation}
where $C''\geq C'$. Applying this estimate to \eqref{eq:DFlb} and rearranging, we arrive at
\begin{equation}\label{eq:DFI2dfin}
\frac{1}{\be^2} I(f) \leq 2\Big[\Dc_{\beta}(f) + \frac1\beta\exp\paren*{{C''}{\paren*{1+{\be^2}\int_{\T^2}\log_+(f)df }} }\Big].
\end{equation}

\medskip
If $\ds\geq 3$, then by Cauchy-Schwarz,
\begin{align}
\int_{\T^\ds}|\Dm^{2-\ds}(f)| df  &\leq \paren*{\int_{\T^\ds} |\Dm^{2-\ds} f|^{2}df}^{1/2} \paren*{\int_{\T^\ds}\indic_{f>A} df}^{1/2} + A\int_{\T^\ds}|\Dm^{2-\ds}f|dx\nn\\
&\leq  \paren*{\frac{1}{\log A}\int_{\T^\ds}\log_+(f) df}^{\frac12}\paren*{\int_{\T^\ds} |\Dm^{2-\ds} f|^{2}df}^{1/2} + A\int_{\T^\ds}|\Dm^{2-\ds}f|dx. \label{eq:DFI3dpsub}
\end{align}
By H\"older's inequality,
\begin{align}
\int_{\T^\ds} |\Dm^{2-\ds} f|^{2}df \leq \|\Dm^{2-\ds}f\|_{L^{2p}}^2 \|f\|_{L^{\frac{p}{p-1}}}.
\end{align}
We choose $p$, so that the HLS conjugate $r$ of $2p$, defined by $\frac{\ds-2}{\ds} = \frac{1}{r} - \frac{1}{2p}$, equals $\frac{p}{p-1}$. This leads to the choice $p=\frac{3\ds}{4}$, implying $r=\frac{p}{p-1} = \frac{3\ds}{3\ds-4}$. Since
\begin{equation}
\ds\left(\frac1q - \frac{3\ds-4}{3\ds}\right) = 1 \Leftrightarrow q = \frac{3\ds}{3\ds-1},
\end{equation}
we may apply Sobolev embedding followed by the estimate \eqref{eq:FIgrad} from \cref{lem:FIests} with $\theta_q = \frac12$ in order to obtain
\begin{align}
\int_{\T^\ds} |\Dm^{2-\ds} f|^{2}df \lesssim \|f\|_{L^{\frac{3\ds}{3\ds-4}}}^3 &\lesssim \|f-1\|_{L^{\frac{3\ds}{3\ds-4}}}^3 + 1 \nn\\
&\lesssim \|\nabla f\|_{L^{\frac{3\ds}{3\ds-1}}}^3 + 1 \nn\\
%&\lesssim I(f)^{2}  + I(f)^{\frac{3}{2}} + 1 \nn\\
&\lesssim I(f)^{2} + 1.\label{eq:DFI3d1}
\end{align}
Since the kernel of $\Dm^{2-\ds}$ is in $L^1$, Young's inequality implies
\begin{align}
A\int_{\T^\ds}|\Dm^{2-\ds}f|dx \lesssim A\|f\|_{L^1}. \label{eq:DFI3d2}
\end{align}
Substituting the estimates \eqref{eq:DFI3d1}, \eqref{eq:DFI3d2} into the right-hand side of \eqref{eq:DFI3dpsub}, we find there is a constant $C>0$ depending only on $\ds$ such that
\begin{equation}
 \int_{\T^\ds}|\Dm^{2-\ds}(f)| df  \leq CA+ \paren*{\frac{C}{\log A}\int_{\T^\ds}\log_+(f) df}^{1/2} \paren*{I(f)  + 1} .\label{eq:DFI3dposub}
\end{equation}
Choosing $A$ according to
\begin{equation}
(\log A)^{\frac12} = {32\cdd}{C}\paren*{1+{\be^2}\int_{\T^\ds}\log_+(f)df}^{\frac12} \Rightarrow A \leq \exp\paren*{{C'}{\paren*{1+{\be^2}\int_{\T^\ds}\log_+(f)df}}},
\end{equation}
it follows that
\begin{align}
\frac{2\cdd}{\be}\int_{\T^\ds}|\Dm^{2-\ds}(f)| df \leq \frac{1}{2\be^2}I(f) + \frac{1}{\be}\exp\paren*{{C''}\paren*{1+{\be^2}\int_{\T^\ds}\log_+(f)df}},
\end{align}
from which we deduce an estimate analogous to \eqref{eq:DFI2dfin}.

\medskip
The $\ds=1$ case is a little trickier because $\Dm^{2-\ds}=\Dm$ is now the half-Laplacian. We instead undo the integration by parts and with $A$ as above, decompose
\begin{equation}
\int_{\T}\paren*{\nabla\log f\cdot \nabla\g\ast f}df  = \int_{\T}\paren*{\nabla\log f\cdot \nabla\g\ast f}\indic_{f\geq A}df + \int_{\T}\paren*{\nabla\log f\cdot \nabla\g\ast (f}\indic_{f\leq A}df.
\end{equation}
%+ \int_{\T^\ds}\paren*{\nabla\log f\cdot \nabla\g\ast (f\indic_{f>A})}\indic_{f\leq A} df.
%By Cauchy-Schwarz and Plancherel's theorem,
%\begin{align}
%\int_{\T^\ds}\left|\paren*{\nabla\log f\cdot \nabla\g\ast (f\indic_{f\leq A}}\right|\indic_{f\leq A}df \leq I(f)^{1/2} A^{1/2} \|\nabla\g%\ast(f\indic_{f\leq A})\|_{L^2} &\leq \frac{A}{\cdd} I(f)^{1/2}. \label{eq:1dDFI1}
%\end{align}
By H\"older's inequality,
\begin{align}
\int_{\T^\ds}\left|\nabla\log f\cdot \nabla\g\ast f\right|\indic_{f\geq A}df \leq I(f)^{1/2} \|\nabla\g\ast f\|_{L^4} \||f|^{1/2}\indic_{f\geq A}\|_{L^4}.
\end{align}
Since $\cdd^{-1}\nabla\g\ast$ is the Hilbert transform, which is bounded on $L^p$ for $p\in (1,\infty)$, we have
\begin{align}
\|\nabla\g\ast f\|_{L^4} \lesssim \|f\|_{L^4} \leq \|f\|_{L^1}^{\frac14}\|f\|_{L^\infty}^{\frac{3}{4}} &\lesssim { \|\nabla f\|_{L^1}^{\frac34} + 1 } \lesssim {I(f)^{\frac38}  + 1}, \label{eq:gradgL4}
\end{align}
where we have also used interpolation and Sobolev embedding (the $L^\infty$ embedding is fundamental theorem of calculus if $\ds=1$) together with estimate \eqref{eq:FIgrad} of \cref{lem:FIests} with $q=1$. Similarly,
\begin{align}
\||f|^{1/2}\indic_{f\geq A}\|_{L^4} = \|f\indic_{f\geq A}\|_{L^2}^{\frac{1}{2}} &\leq \|f\indic_{f\geq A}\|_{L^1}^{\frac{1}{4}}\paren*{\|f-1\|_{L^\infty}^{\frac14} + 1} \nn\\
&\lesssim \paren*{\frac{1}{\log A}\int_{\T^\ds}\log_+(f)df}^{\frac14}\paren*{I(f)^{\frac18} + 1}. \label{eq:fL4}
\end{align}
Estimates \eqref{eq:gradgL4},\eqref{eq:fL4} together imply that
\begin{align}
\int_{\T^\ds}\left|\nabla\log f\cdot \nabla\g\ast f\right|\indic_{f\geq A}df \lesssim  \paren*{\frac{1}{\log A}\int_{\T^\ds}\log_+(f)df}^{\frac14}(I(f) + I(f)^{\frac12}). \label{eq:1dDFI2}
\end{align}
By similar arguments as to the previous case,
\begin{align}
\int_{\T^\ds}\left|\nabla\log f\cdot \nabla\g\ast f\right|\indic_{f\leq A} df &\leq  I(f)^{\frac12} \|\nabla\g\ast f\|_{L^4} \||f|^{\frac12}\indic_{f\leq A}\|_{L^4} \nn\\
&\lesssim A^{\frac12} I(f)^{\frac12}\paren*{I(f)^{\frac38}  + 1} \label{eq:1dDFI3} .
\end{align}
Combining the estimates \eqref{eq:1dDFI2}, \eqref{eq:1dDFI3}, we find there is a constant $C>0$ depending only $\ds$ such that
\begin{align}
\left|\int_{\T^\ds}\paren*{\nabla\log f\cdot \nabla\g\ast f}df\right| \leq C \paren*{\frac{1}{\log A}\int_{\T^\ds}\log_+(f)df}^{\frac14}  (I(f) +I(f)^{\frac12}) + CA^{\frac12} I(f)^{\frac12}\paren*{I(f)^{\frac38}  + 1}.
\end{align}
We choose $A$ so that
\begin{equation}
{(\log A)^{\frac14}} = {32\cdd}{ C}\paren*{1+ { \be^4} \int_{\T^\ds}\log_+(f)df}^{\frac14} \Longrightarrow A \leq \exp\paren*{{C'}\paren*{1+ { \be^4}\int_{\T^\ds}\log_+(f)df}}.
\end{equation}
Now using Young's product inequality multiple times, we obtain
\begin{align}
\frac{2}{\be}\left|\int_{\T^\ds}\paren*{\nabla\log f\cdot \nabla\g\ast f}df\right| \leq \frac{1}{2\be^2}I(f) + { (1+\be^6)}\exp\paren*{C''\paren*{1+ { \be^4}\int_{\T^\ds}\log_+(f)df}},
\end{align}
from which \eqref{eq:DFI1d} follows.
\end{proof}

\begin{remark}\label{rem:FIlinbnd}
The control of the Fisher information provided by \cref{lem:DFI} together with the estimate \eqref{eq:posentFED} from \Cref{rem:posentFED} implies that if $\mu$ is a probability density solution to \eqref{eq:lima}, then for a constant $C_\beta$, increasing with respect to $\beta$,
\begin{align}
\frac{1}{\beta}\int_0^t I(\mu^\tau)d\tau &\leq C\int_0^t\Bigg(\be\Dc_{\beta}(\mu^\tau) + C_{\beta}\exp\paren*{C_{\beta}\paren*{1+\int_{\T^\ds}\log_+(\mu^\tau)d\mu^\tau}} \Bigg)d\tau \nn\\
&\leq C\Ec_\beta(\mu^0) + CC_{\beta}t\exp\paren*{ C_1 + \paren*{1-\frac{\beta}{\bec}}^{-1}\paren*{\Ec_\beta(\mu^{0}) + {C_2}{\beta}}},
\end{align}
where in the final line we have also used $\be\int_0^\infty\Dc_{\beta}(\mu^\tau)d\tau \leq \Ec_\beta(\mu^0)$. This gives a linear-in-time bound for the Fisher information, which will be important to the proof of \cref{lem:aLrbnd} in the next subsection.
\end{remark}

%By Plancherel, $\cdd\int_{\T^\ds}\Dm^{2-d}(f)df = \|f\|_{\dot{H}^{1/2}}^2$, and it is not clear how to control this seminorm as above for $d\geq 2$.

\begin{comment}
For the remaining assertion of the lemma, observe that if $d\leq 2$, then we can always find $q<2$ such that the Sobolev conjugate $q^* = \frac{dq}{d-q}>p$. So by Gagliardo-Nirenberg interpolation,
\begin{equation}
.
\end{equation}

Similarly, if $d\geq 3$, then by $p=?$ is the Sobolev conjugate of $\frac{d}{d-1}$; so we may again always find $q\leq \frac{d}{d-1}$ such that $p=q^*$. The proof is now complete.
\end{comment}

\begin{comment}
\begin{remark}
Since $\frac{2\ds}{2\ds-1} < \frac{d}{d-1}$, \cref{lem:FIests} implies that
\begin{equation}
\|f\|_{L^{\frac{2d}{d-1}}} \lesssim 
\end{equation}
We will use this observation together with the free energy dissipation identity in the next section to conclude that the solutions of \cref{prop:LWPa} are global.
\end{remark}
\end{comment}

\section{Global existence and relaxation}\label{sec:glob}
We now use the results of \Cref{sec:WP,sec:aLogFE} to show that the local solutions previously constructed are global (i.e., $T_{\max}=\infty$), if { $\beta <\bec$ or $\be<\bels$ and the initial data is sufficiently close to $\muu$ in $L^2$}. Moreover, if $\be<\beu$ or { $\be<\bels$ and initial data is sufficiently close to $\muu$ in $L^2$}, then $\mu^t$ converges at an exponential rate to the uniform distribution as $t\rightarrow\infty$ in Sobolev spaces of arbitrarily high regularity.

\subsection{Global existence}\label{ssec:globE}
Let $\mu \in C_w([0,T_{\max}), L^1(\T^\ds))$ be a maximal solution to \eqref{eq:milda} with initial datum $\mu^0 \in L^1(\T^\ds)$. Throughout this subsection, we assume that $\mu^0 \in \P_{ac}(\T^\ds)$. %By \ref{lem:?}, we know that for $0<T<T_{\max}$ and every $1\leq p\leq \infty$,
%\begin{equation}
%\sup_{0<t\leq T} t^{\frac{\ds}{2}\left(1-\frac1p\right)}\|\mu^t\|_{L^p} <\infty.
%\end{equation}
In order to show that $\mu$ is global, it suffices by the blow-up criterion of \cref{rem:ablowup} to show that if $T_{\max}<\infty$, then there there is a bound for
\begin{equation}
\sup_{0<t\leq T} t^{\frac{1}{4}}\|\mu^t\|_{L^{\frac{2\ds}{2\ds-1}}}, \qquad T<T_{\max},
\end{equation}
which is independent of $T$. This is a consequence of the Fisher information estimates from \cref{lem:FIests} and the control of Fisher information in terms of the free energy by \cref{rem:FIlinbnd}.

\begin{comment}
Suppose that $T_{\max}<\infty$. Then there exists an increasing sequence of times $t_n\rightarrow T_{\max}^-$ such that $\|\mu^{t_n}\|_{L^{\frac{2\ds}{2\ds-1}}} \rightarrow \infty$.

By H\"older's inequality, Sobolev embedding and \eqref{eq:FIgrad} with $q=1$,
\begin{align}
\|\mu^{t}\|_{L^{\frac{2\ds}{2\ds-1}}} \lesssim \|\mu^t\|_{L^{\frac{\ds}{\ds-1}}} \lesssim 1+ \|\nabla\mu^t\|_{L^1} \lesssim 1 + I(\mu^t)^{\frac12}.
\end{align}
Recalling the estimates \eqref{eq:muL2d2d-1}, \eqref{eq:Holmunabgmu}, we find that
\begin{align}
\|\mu^{t}\|_{L^{\frac{2\ds}{2\ds-1}}} &\lesssim (t/\beta)^{-\frac{1}{4}}\|\mu^0\|_{L^1} + \int_0^t \paren*{(t-\tau)/\beta}^{-\frac{3}{4}} \|\mu^\tau\|_{L^{\frac{2\ds}{2\ds-1}}}^2 d\tau \nn\\
&\lesssim (t/\beta)^{-\frac{1}{4}}\|\mu^0\|_{L^1} + \int_0^t \paren*{(t-\tau)/\beta}^{-\frac{3}{4}}\left(1+ I(\mu^\tau)\right)d\tau.
\end{align}

This estimate implies that $I(\mu^{t_n})\rightarrow \infty$. But by \cref{rem:FIlinbnd}, $\int_0^{T} I(\mu^t)dt$ is bounded for any $T>0$. Hence, $I(\mu^t)<\infty$ for a.e. $t\in [0,T_{\max})$ and since $\mu\in C^\infty((0,T_{\max})\times\T^\ds)$, we in fact have $I(\mu^t)<\infty$ for every $t\in [0,T_{\max})$. 
\end{comment}
%Next, we use the free energy dissipation identity \eqref{eq:FED} to obtain an exponential rate of decay of $\Ec_\beta(\mu^t)$ as $t\rightarrow\infty$, which by Pinsker's inequality implies an exponential rate of decay for $\|\mu^t-1\|_{L^1}$.
Since $\frac{2\ds}{2\ds-1}\leq 2$, it suffices to prove an $L^2$ estimate. In the statement of \cref{lem:aLrbnd} below, we assume that the initial datum $\mu^0\in L^2$, but this is not restrictive: the reader will recall from \cref{lem:greg} that solutions are instantaneously smooth.

\begin{lemma}\label{lem:aLrbnd}
{ Let $\beta<\bec$}. Let $\mu$ be a solution of equation \eqref{eq:lima}. Then there exist constants $C,C_1,C_2>0$, depending only on $\ds$, such that if $\ds=1$, then for $t\in [0,T_{\max})$,
\begin{align}\label{eq:aLrbnd1d}
\|\mu^t-1\|_{L^2}^2  { + 1}\leq \paren*{\|\mu^{0}-1\|_{L^2}^2 + { 1}}e^{C\beta\int_0^t I(\mu^\tau)d\tau},
\end{align}
and if $\ds\geq 2$, then for $t\in [0,T_{\max})$,
\begin{multline}\label{eq:aLrbnd2d+}
\|\mu^t-1\|_{L^2}^2 \leq e^{-\frac{4\pi^2}{\beta}t}\|\mu^0-1\|_{L^2}^2 + C\int_0^t  e^{-\frac{4\pi^2}{\beta}(t-\tau)} I(\mu^\tau)d\tau\\
+\frac{\beta}{4\pi^2}\exp\left(C\be\paren*{C_{1} + \paren*{1-\frac{\beta}{\bec}}^{-1}\paren*{\Ec_\beta(\mu^{0}) + {C_2}{\beta}}}\right),
\end{multline}
In particular, by \cref{rem:FIlinbnd}, the right-hand sides are locally bounded in time as a function of $\beta,\|\mu^0-1\|_{L^2}, \Ec_\beta(\mu^0)$.
\end{lemma}
\begin{proof}
Using equation \eqref{eq:lima}, integration by parts, and that $\div(\g\ast) = \cdd\Dm^{2-\ds}$, we compute
\begin{multline}\label{eq:adtLr}
\frac{d}{dt}\|\mu^t-1\|_{L^2}^2 = \cdd\int_{\T^\ds}|\mu^t-1|^2 \Dm^{2-\ds}(\mu^t-1) dx \\
+ 2\cdd\int_{\T^\ds}(\mu^t-1)\Dm^{2-\ds}(\mu^t-1)dx  - \frac{2}{\beta}\int_{\T^\ds}\left|\nabla (\mu^t-1)\right|^2dx.
\end{multline}

If $\ds=1$, then we use H\"older's inequality and fractional Leibniz rule to obtain
\begin{align}
\left|\int_{\T}|\mu^t-1|^2 \Dm(\mu^t-1) dx\right| &= \left|\int_{\T}\Dm\left((\mu^t-1)^2\right)(\mu^t-1) dx\right| \nn\\
&\lesssim \|\nabla\mu^t\|_{L^2} \|\mu^t-1\|_{L^2} \|\mu^t-1\|_{L^\infty} \nn\\
&\lesssim \|\nabla\mu^t\|_{L^2} \|\mu^t-1\|_{L^2} \|\nabla\mu^t\|_{L^1} \nn\\
&\lesssim \|\nabla\mu^t\|_{L^2} \|\mu^t-1\|_{L^2} I(\mu^t)^{\frac{1}{2}} , \label{eq:LrHLST1d1}
\end{align}
where the penultimate line is by Sobolev embedding (remember the endpoint $L^\infty$ estimate if $\ds=1$ holds by {fundamental theorem of calculus}) and the ultimate line is by estimate \eqref{eq:FIgrad} of \cref{lem:FIests}.

If $\ds=2$, then we have
\begin{align}
\left|\int_{\T^\ds}|\mu^t-1|^2 \Dm^{2-\ds}(\mu^t-1) dx\right| = \left|\int_{\T^\ds}  (\mu^t-1)^3\right| &\leq \|\mu^t-1\|_{L^{3}}^{3}.
\end{align}
A direct application of estimate \eqref{eq:fLr+1} of \cref{lem:FIests} to the right-hand side yields an estimate that will not close unless $\|\mu^t-1\|_{L^1}$ is sufficiently small uniformly in $t$. To get around this, we use a vertical displacement trick that seems to be fairly well-known (e.g., see the proofs of \cite[Proposition 3.3]{BDP2006}, \cite[Lemma 2.2]{BCM2008}, \cite[Lemma 2.6]{FM2016}). Given $A>0$, write
\begin{equation}\label{eq:mubmuA}
|\mu^t-1| = \min(|\mu^t-1|,A) + (|\mu^t-1|-A)_+.
\end{equation}
Then by the triangle inequality and convexity of the function $z\mapsto z^{3}$, we have
\begin{align}\label{eq:2dLr1split}
\|\mu^t-1\|_{L^{3}}^{3} \lesssim \| (|\mu^t-1|-A)_{+} \|_{L^{3}}^{3} + A^{3}.
\end{align}
By application of estimate \eqref{eq:fLr+1} of \cref{lem:FIests}, { we get}
\begin{align}
\| (|\mu^t-1|-A)_{+} \|_{L^{3}}^{3} &\lesssim \|(|\mu^t-1|-A)_{+}\|_{L^1}\| \nabla (|\mu^t-1|-A)_{+} \|_{L^{2}}^{2} \nn\\
&\lesssim \|(|\mu^t-1|-A)_{+}\|_{L^1} \|\nabla \mu^t \|_{L^{2}}^{2}, \label{eq:2dLr1A}
\end{align}
where {the} final expression follows from a classical estimate for the gradient of the max of two functions (e.g., see \cite[Corollary 6.18]{LL2001}). Now suppose that $A>2$, so that $|\mu^t-1|\geq A$ implies $|\mu^t| \geq A-1 > 1$. Then,
\begin{align}
\|(|\mu^t-1|-A)_{+}\|_{L^1} &\leq \| \paren*{|\mu^t-1|-A}\indic_{|\mu^t|\geq A-1} \|_{L^1} \nn\\
&\leq \frac{1}{\log(A-1)} \int_{\T^\ds}|\mu^t-1| \log(\mu^t) \indic_{|\mu^t|\geq A-1} dx \nn\\
&\lesssim \frac{1}{\log(A-1)}\paren*{ \int_{\T^\ds}\mu^t\log_+(\mu^t)dx + 1}.
\end{align}
Substituting this estimate into the right-hand side of \eqref{eq:2dLr1A}, then applying the resulting estimate to the right-hand side of \eqref{eq:2dLr1split}, we obtain
\begin{equation}
\|\mu^t-1\|_{L^{3}}^{3}  \lesssim A^{3} +\paren*{ \frac{1}{\log(A-1)}\paren*{\int_{\T^\ds}\mu^t\log_+(\mu^t)dx +1}}\|\nabla \mu^t \|_{L^{2}}^{2}. \label{eq:LrHLSd2fin}
\end{equation}

If $\ds\geq 3$, then we use the decomposition \eqref{eq:mubmuA} and the estimate \eqref{eq:fLrHLS} from \cref{lem:FIests} to instead obtain,
\begin{align}
\left|\int_{\T^\ds}|\mu^t-1|^2 \Dm^{2-\ds}(\mu^t-1) dx\right| &\lesssim  {\left|\int_{\T^\ds}|(|\mu^t-1|-A)_+|^2 \Dm^{2-\ds}((|\mu^t-1|-A)_+) dx\right|+  A^{3} }\nn\\
&\lesssim \|(|\mu^t-1|-A)_+\|_{L^1} \| \nabla(|\mu^t-1|-A)_+ \|_{L^{2}}^2 +A^{3} \nn\\
&\lesssim A^{3} +\paren*{ \frac{1}{\log(A-1)}\paren*{\int_{\T^\ds}\mu^t\log_+(\mu^t)dx + 1}} \|\nabla\mu^t \|_{L^{2}}^{2}, \label{eq:LrHLSd>2fin}
\end{align}
which agrees with our $\ds=2$ estimate.

We consider the remaining term in \eqref{eq:adtLr}. 
If $\ds=1$ then by Cauchy-Schwarz,
\begin{align}
\int_{\T}\left|(\mu^t-1)\Dm(\mu^t-1)\right| dx &\leq \|\mu^t-1\|_{L^2} \|\nabla \mu^t\|_{L^2} \lesssim \|\mu^t-1\|_{L^1}^{\frac12} I(f)^{\frac14} \|\nabla\mu^t\|_{L^2} , \label{eq:LrT21dr=2}
\end{align}
where the final inequality is by Gagliardo-Nirenberg interpolation and estimate \eqref{eq:FIgrad} from \cref{lem:FIests} with $q=1$. If $\ds\geq 2$, then { by H\"older and HLS lemma,}
%\begin{align}
%\left|\int_{\T^\ds}(\mu^t-1)\Dm^{2-\ds}(\mu^t-1)dx\right| \leq \|\mu^t-1\|_{L^2}^2,
%\end{align}
%and if $\ds\geq 3$, then using H\"older's inequality together with the fact that the kernel of $\Dm^{2-\ds}$ is in $L^1$, 
%\begin{align}
%\left|\int_{\T^\ds}(\mu^t-1)\Dm^{2-\ds}(\mu^t-1)dx\right| \leq  \| (\mu^t-1)\|_{L^2} \|\Dm^{2-\ds}(\mu^t-1)\|_{L^2}&\lesssim \|\mu^t-1\|_{L^2}^2. \label{eq:LrT2d>3}
%\end{align}
%By H\"older's inequality, then Hardy-Littlewood-Sobolev, followed by Sobolev embedding, we find
\begin{align} \label{eq:LrT22d}
\left|\int_{\T^\ds}(\mu^t-1)\Dm^{2-\ds}(\mu^t-1)dx\right| \leq \|\mu^t-1\|_{L^{\frac{\ds}{\ds-1}}} \|\Dm^{2-\ds}(\mu^t-1)\|_{L^{\ds}} \lesssim \|\mu^t-1\|_{L^{\frac{\ds}{\ds-1}}}^2 \lesssim I(\mu^t),
\end{align}
where the final inequality is by {Sobolev embedding and} estimate \eqref{eq:FIgrad} from \cref{lem:FIests} with $q=1$.

\medskip
The conclusion of the proof is now a matter of bookkeeping and appropriately choosing $A$. We first dispense with the $\ds=1$ case. Combining the estimates \eqref{eq:adtLr}, \eqref{eq:LrHLST1d1}, \eqref{eq:LrT21dr=2}, we have shown that there is a constant $C>0$, independent of $\be$, such that
\begin{multline}
\frac{d}{dt}\|\mu^t-1\|_{L^2}^2  \leq - \frac{2}{\beta}\int_{\T^\ds}\left|\nabla\mu^t\right|^2dx + C\|\nabla\mu^t\|_{L^2} \|\mu^t-1\|_{L^{2}}I(\mu^t)^{\frac12}\\
+ C\|\mu^t-1\|_{L^1}^{\frac12} I({ \mu^t})^{\frac14} \|\nabla\mu^t\|_{L^2}.
\end{multline}
Since $\|\mu^t-1\|_{L^1} \leq \|\nabla\mu^t\|_{L^1} \lesssim I(\mu^t)^{\frac12}$ by Poincar\'{e} and estimate \eqref{eq:FIgrad} with $q=1$, we find from Young's product inequality that there is a $C'\geq C$ such that
\begin{align}
\frac{d}{dt}\|\mu^t-1\|_{L^2}^2  \leq - \frac{1}{\beta}\left\|\nabla \mu^t\right\|_{L^2}^2 + C'\beta I(\mu^t)\left(\|\mu^t-1\|_{L^2}^2 +1\right).
\end{align}
It follows from the Gr\"onwall-Bellman lemma that for $t\in [0, T_{\max})$,
\begin{align}\label{eq:1dL2gronfin}
 \|\mu^t-1\|_{L^2}^2+1  \leq  \left(\|\mu^{0}-1\|_{L^2}^2+1 \right) e^{C'\beta\int_0^t I(\mu^\tau)d\tau}.
\end{align}
%Since the exponents of the $I(\mu^\tau)$ factors are all $\leq 1$, by H\"older's inequality and \ref{rem:}, all their respective integrals grow at most linearly in time.

Now we consider the $\ds\geq 2$ case. Combining the estimates \eqref{eq:adtLr}, \eqref{eq:LrHLSd2fin}, \eqref{eq:LrHLSd>2fin}, \eqref{eq:LrT22d}, we see that
\begin{multline}\label{eq:2d+L2gronpreA}
\frac{d}{dt}\|\mu^t-1\|_{L^2}^2  \leq  C\paren*{ \frac{1}{\log(A-1)}\paren*{\int_{\T^\ds}\mu^t\log_+(\mu^t)dx +1}} \|\nabla \mu^t \|_{L^{2}}^{2}\\
+C A^{3} + CI(\mu^t) - \frac{2}{\beta}\left\|\nabla\mu^t\right\|_{L^2}^2.
\end{multline}
Using estimate \eqref{eq:posentFED} from \cref{rem:posentFED}, we see that there exist constants $C_{1},C_2>0$, depending on $\ds$, such that
\begin{multline}
C\paren*{ \frac{1}{\log(A-1)}\paren*{\int_{\T^\ds}\mu^t\log_+(\mu^t)dx +1}}  \\
\leq C\Bigg(\frac{1}{\log(A-1)}\paren*{C_{1} + \paren*{1-\frac{\beta}{\bec}}^{-1}\paren*{\Ec_\beta(\mu^{0}) + {C_2}{\beta}}}\Bigg).
\end{multline}
Choose $A>2$ sufficiently large so that the right-hand side equals $\frac{1}{\beta}$. Then inserting this choice for $A$ into \eqref{eq:2d+L2gronpreA}, we obtain
\begin{align}
\frac{d}{dt}\|\mu^t-1\|_{L^2}^2  \leq - \frac{1}{\beta}\left\|\nabla\mu^t\right\|_{L^2}^2 + C'I(\mu^t) +\exp\left(C'\be\paren*{C_{1}' + \paren*{1-\frac{\beta}{\bec}}^{-1}\paren*{\Ec_\beta(\mu^{0}) + {C_2'}{\beta}}}\right),
\end{align}
for $C_1'\geq C_1$, $C_2'\geq C_2$, and $C'\geq C$. Since by the Poincar{\'e} inequality,
\begin{equation}
\left\|\nabla\mu^t\right\|_{L^2}^2 \geq 4\pi^2\|\mu^t-1\|_{L^2}^2,
\end{equation}
{an} application of Gr\"{o}nwall's lemma then yields
\begin{multline}
\|\mu^t-1\|_{L^2}^2 \leq e^{-\frac{4\pi^2}{\beta}t}\|\mu^0-1\|_{L^2}^2 + C'\int_0^t  e^{-\frac{4\pi^2}{\beta}(t-\tau)} I(\mu^\tau)d\tau\\
+\exp\left(C'\be\paren*{C_{1}' + \paren*{1-\frac{\beta}{\bec}}^{-1}\paren*{\Ec_\beta(\mu^{0}) + {C_2'}{\beta}}}\right)\int_0^t e^{-\frac{4\pi^2}{\beta}(t-\tau)}d\tau,
\end{multline}
which, upon simplification of the right-hand side, completes the proof of the lemma.
\end{proof}

\begin{comment}
Fix some $t_0 \in (0,T)$. Then using equation \eqref{eq:milda} time translated by $t_0$, the triangle and Minkowski inequalities, and the smoothing of the heat kernel, we find
\begin{equation}
\forall t\in [t_0,T], \qquad \|\mu^t\|_{L^{\frac{2\ds}{2\ds-1}}} \lesssim (t-t_0)^{-\frac14}\|\mu^{t_0}\|_{L^1} + \int_{t_0}^t (t-\tau)^{-\frac{3}{4}} d\tau
\end{equation}
\end{comment}

\subsection{Relaxation to equilibrium}\label{ssec:globR}
\cref{lem:aLrbnd} tells us that if $\be<\bec$, then any probability density solution to equation \eqref{eq:lima} must exist globally and, moreover, by \cref{lem:greg}, is classical for positive times. But it does not say anything about the asymptotic behavior, {nor} rule out ``infinite-time blowup.''  We now aim to show that this cannot happen.%and, in fact, $\|\mu^t-1\|_{L^r}$ eventually becomes nonincreasing and, in fact, decreases exponentially fast to zero.

The first result is $L^1$ (qualitative) convergence to a steady state $\mu_\beta$ as $t\rightarrow\infty$ for $\be<\bec$. {If $\ds=2$, then $\mu_\be$ is uniform. While if $\ds\neq2$, then $\mu_\be$ is uniform provided $\be$ is below the uniqueness threshold $\beu$.}

\begin{lemma}\label{lem:mutasyss}
Let $\ds\geq 1$ and $\beta<\bec$. There exists a $C^\infty$ solution $\mu_{\beta}$  to the equation
\begin{equation}\label{eq:muinfeq}
\mu_{\beta} = \dfrac{e^{\beta\g\ast\mu_{\beta}}}{\int_{\T^\ds}e^{\beta\g\ast\mu_\beta}dx},
\end{equation}
such that $H(\mu^t\vert\mu_{\beta})\rightarrow 0$ as $t\rightarrow\infty$. In particular, $\mu_{\beta}$ is a stationary solution to \eqref{eq:lima}. {For $\be<\beu$, we have $\mu_\be=\muu$. In particular, if $\ds=2$, then for all $\be< \bec$, we have $\mu_\be=\muu$.}
\end{lemma}
\begin{proof}
The convergence of $\mu^t$ to a solution $\mu_{\beta}$ of \eqref{eq:muinfeq} as $t\rightarrow\infty$ follows from a slight variation on \cite[Section 4]{BDP2006}. Essentially, one uses the decreasing property of the free energy and compactness coming from the entropy bound to prove the existence of a limit $\mu_{\beta}$ along a subsequence of times $t_k\rightarrow\infty$. One can then show that {the}  dissipation functional must vanish at $\mu_{\beta}$, which allows one to show that it solves \eqref{eq:muinfeq}. One can then show that both the energy and entropy in $\Ec_\beta(\mu^{t_k})$ converge to their corresponding values for $\mu_{\beta}$, as $k\rightarrow\infty$, and that the limit $\mu_{\beta}$ is independent of the sequence $t_{k}$.  We omit the details.

The regularity of $\mu_{\beta}$ is not explicitly considered in \cite[Section 4]{BDP2006}, but this follows from standard bootstrap arguments, using equation \eqref{eq:muinfeq} and the $L\log L$ bound coming from the free energy. Strict positivity is then immediate since $\g\ast\mu_{\beta}$ is bounded.

That $\mu_{\beta} =\muu$ if $\ds=2$ follows from the works \cite{LL2006, GM2019, GGHL2021} showing uniqueness of solutions to \eqref{eq:muinfeq} for inverse temperatures $\beta\leq \bec$, which is sharp. For sufficiently small $\be$ in any dimension $\ds\geq 1$, uniqueness of solutions to \eqref{eq:muinfeq} is proven in \cref{prop:ssunq} below.
\end{proof}

\begin{remark}
The uniqueness of solutions to equation \eqref{eq:muinfeq} is equivalent to the assertion that every solution $\mu^t$ to \eqref{eq:lima} converges as $t\rightarrow\infty$ to the uniform distribution. Indeed, the forward implication is clear. For the reverse implication, we can take a solution $\mu_{\beta}$ to \eqref{eq:muinfeq} as our initial datum $\mu^0$ for \eqref{eq:lima} and evolve it forward. Since $\mu_{\beta}$ is a stationary solution to \eqref{eq:lima}, this then yields that $\mu_{\beta}\equiv 1$. 

Knowing uniqueness for equation \eqref{eq:muinfeq} then implies that there is a unique minimizer of $\Ec_\beta$ given by the uniform distribution, again assuming that $\beta<\bec$.\footnote{If $\ds=2$, then the existence of a minimizer is also known in the critical case $\beta=\bec$ \cite{DJL1997, NT1998}.} Indeed, this can either be seen by taking the Euler-Lagrange equation for $\Ec_\beta$ or by using the monotonicity of the free energy together with the fact that all $t\rightarrow\infty$ limits are the uniform distribution. In particular, since $\Ec_\beta(\muu)=0$, uniqueness implies that the minimum value of $\Ec_\beta$ is zero. The minimal value of $\Ec_\beta$ is discussed in more detail in \cref{ssec:Instabssunq}.
\end{remark}

With \cref{lem:mutasyss} in hand, we may revisit the proof of \cref{lem:aLrbnd} in to obtain an exponential-in-time $L^2$ decay estimate. The reader will recall that $\bels = \frac{(2\pi)^\ds}{\cdd}$, where $\cdd$ is defined in \eqref{eq:gdefa}.

\begin{lemma}\label{lem:aL2exp}
Let $\mu$ be a solution to \eqref{eq:lima}. Suppose that $\beta<\bels$. Then there exists constants $\d_\be,{c}_{\be}>0$ depending on $\ds,\beta$ such that if for some $T_0\geq 0$,
\begin{equation}\label{eq:aL2expcond}
 \|\mu^{t}-1\|_{L^1} \leq \d_{\beta}, \qquad \forall t\geq T_0,
\end{equation}
then
\begin{equation}\label{eq:aL2expbnd}
\forall t\geq T_0, \qquad \|\mu^t -1\|_{L^2} \leq \|\mu^{T_0}-1\|_{L^2}e^{-{c}_{\be}(t-T_0)}.
\end{equation}
Additionally, if there exists $T_0\geq 0$, such that
\begin{align}\label{eq:aL2expcond'}
 \|\mu^{T_0}-1\|_{L^2} \leq \d_{\beta},
\end{align}
then \eqref{eq:aL2expbnd} holds for all $t\geq T_0$.
\end{lemma}
\begin{proof}
By time translation, we may assume without loss of generality that $T_0 = 0$. Recalling \eqref{eq:adtLr},
\begin{multline}\label{eq:adtL2'}
\frac{d}{dt}\|\mu^t-1\|_{L^2}^2 = \cdd\int_{\T^\ds}|\mu^t-1|^2\Dm^{2-\ds}(\mu^t-1)dx + 2\cdd\int_{\T^\ds}(\mu^t-1)\Dm^{2-\ds}(\mu^t-1)dx \\
-\frac{2}{\beta}\int_{\T^\ds}|\nabla\mu^t|^2dx.
\end{multline}

Consider the first term on the right-hand side of \eqref{eq:adtL2'}. Suppose $\ds=1$. Similar to the reasoning leading to \eqref{eq:LrHLST1d1}, we see that
\begin{align}
\left|\int_{\T}|\mu^t-1|^2\Dm(\mu^t-1)dx\right| &\leq \|\Dm (\mu^t-1)\|_{L^2} \|\mu^t-1\|_{L^4}^2 \leq {C}\|\nabla \mu^t\|_{L^2}^2 \|\mu^t-1\|_{L^1}.
\end{align}
If $\ds\geq 2$, then using the estimate \eqref{eq:fLr+1} or \eqref{eq:fLrHLS} from \cref{lem:FIests},
\begin{align}
\left|\int_{\T^\ds}|\mu^t-1|^2 \Dm^{2-\ds}(\mu^t-1) dx\right| &\leq {C} \|\mu^t -1\|_{L^1} \|\nabla\mu^t\|_{L^2}^2.
\end{align}

Consider the second and third terms on the right-hand side of \eqref{eq:adtL2'}. By Plancherel's theorem,
\begin{align}
2\cdd\int_{\T^\ds}(\mu^t-1)\Dm^{2-\ds}(\mu^t-1)dx = 2\cdd \|\mu^t\|_{\dot{H}^{\frac{2-\ds}{2}}}^2 {\leq \frac{2\cdd}{(2\pi)^{\ds}} \|\nabla\mu^t\|_{L^2}^2}.
\end{align}
So if $\frac{1}{\be} > \frac{(2\pi)^\ds}{\cdd}$ (i.e., $\be<\bels$), then
\begin{equation}\label{eq:diss}
2\cdd\int_{\T^\ds}(\mu^t-1)\Dm^{2-\ds}(\mu^t-1)dx -\frac{2}{\be}\int_{\T^\ds}|\nabla\mu^t|^2dx \leq {-2\left(\frac1\be-\frac{1}{\bels}\right) \|\nabla\mu^t\|_{L^2}^2} < 0.
\end{equation}
{After a little bookkeeping, we have shown that
\begin{align}\label{eq:dtL2preL1small}
\frac{d}{dt}\|\mu^t-1\|_{L^2}^2 \leq \Big(C\cdd\|\mu^t -1\|_{L^1} - 2\Big(\frac1\be-\frac{1}{\bels}\Big)\Big) \|\nabla\mu^t\|_{L^2}^2.
\end{align}
}

{Now suppose that $\|\mu^t-1\|_{L^1}$ is sufficiently small so as to satisfy
\begin{equation}
\frac2\be - C\cdd\|\mu^t-1\|_{L^1}  > \frac{2}{\bels}
\end{equation}
or equivalently,
\begin{equation}\label{eq:mutL1smallcon}
\|\mu^t-1\|_{L^1} <  \d_{\be} \coloneqq \frac{2}{\cdd C}\paren*{\frac1\be - \frac1\bels}.
\end{equation}
Then
\begin{align}\label{eq:aL2rhsbnd}
\Big(C\cdd\|\mu^t -1\|_{L^1} - 2\Big(\frac1\be-\frac{1}{\bels}\Big)\Big) \|\nabla\mu^t\|_{L^2}^2 &\leq-\underbrace{\paren*{2\paren*{\frac1\be - \frac1\bels}  - \delta_\be C\cdd }}_{>0}\|\nabla\mu^t\|_{L^2}^2  \nn\\
&\leq -4\pi^2\paren*{2\paren*{\frac1\be - \frac1\bels}  - \delta_\be C\cdd }\|\mu^t-1\|_{L^2}^2.
\end{align}
}
Applying Gr\"onwall's lemma, we obtain the desired conclusion \eqref{eq:aL2expbnd}.

We now use a continuity argument to show that if $\|\mu^0-1\|_{L^2}$ satisfies the smallness condition \eqref{eq:mutL1smallcon}, then this condition holds for all $t\geq 0$. Continuity in time of the solution implies that either $\|\mu^t-1\|_{L^2} < \d_{\be}$ for every $t\geq 0$, or there exists a minimal time $T_*>0$ such that
\begin{align}
\forall t\in [0,T_*), \quad \|\mu^t-1\|_{L^2} < \d_\be \qquad \text{and} \qquad  \|\mu^{T_*}-1\|_{L^2} = \d_{\be}.
\end{align}
Since $\|\mu^t-1\|_{L^1}\leq \|\mu^t-1\|_{L^2}$, we see from that \eqref{eq:aL2rhsbnd} that $\|\mu^t-1\|_{L^2}$ is decreasing on $[0,T_*)$, which implies that $\|\mu^{T_*}-1\|_{L^2} \leq \|\mu^0-1\|_{L^2} < \d_\be$, a contradiction. Hence, no such $T_*$ can exist, and since $\|\mu^t-1\|_{L^2}\geq \|\mu^t-1\|_{L^1}$, we obtain the desired conclusion.
\end{proof}

We may combine \Cref{lem:aLrbnd,lem:aL2exp} to obtain the following global $L^2$ bound.

\begin{lemma}\label{lem:L2glob}
Let $\mu$ be a solution to \eqref{eq:lima}, and suppose that $\be<\min(\bec,\bels)$. Let $T_0\geq 0$ be the minimal time such that $\|\mu^{t}-1\|_{L^1} \leq \d_{\be}$ for all $t\geq T_0$, where $\d_\be$ is as in \cref{lem:aL2exp}. Then
\begin{align}\label{eq:L2globRHS}
\forall t\geq 0, \qquad \|\mu^t-1\|_{L^2}^2 \leq \W(\be,\|\mu^0-1\|_{L^2}, \Ec_\beta(\mu^0), e^{T_0})e^{-2{c}_{\be}(t-T_0)_+},
\end{align}
where ${c}_{\be}>0$ is an in \cref{lem:aL2exp} and $\W: [0,\infty)^4 \rightarrow [0,\infty)$ is a continuous, increasing function, which vanishes if any of its arguments are zero.
\end{lemma}

We can now combine the global $L^2$ bound \cref{lem:L2glob} with the local smoothing of \Cref{lem:hypequiv,lem:greg} to obtain exponential-in-time decay of arbitrarily large Sobolev norms. This then completes the proof of \cref{thm:mainlim}.

\begin{lemma}\label{lem:nabLinfexp}
Let $\ds\geq 1$. If $\be<\min(\bec,\bels)$, let $T_0$ be the minimal time such that $\|\mu^{t}-1\|_{L^1} \leq \d_{\be}$ for all $t\geq T_0$. Then for any integer $n\geq 1$, there exists a $c>0$, depending on $\ds,\be$ and a function ${\W}_n: [0,\infty)^5\rightarrow [0,\infty)$, which is continuous, increasing, and vanishes if and of its arguments are zero, such that for any $0<t_0<t_1< t$,
\begin{multline}\label{eq:nabLinfexp}
\|\nabla^{\otimes n}\mu^t\|_{L^\infty} \leq \W_n\Big(\beta, \|\mu^{t_0}-1\|_{L^2}, \Ec_\beta(\mu^0), e^{T_0},  \min((t_1-t_0)/\beta,1)^{-1}\Big)\\
\times \min((t-t_1)/\beta,1)^{-\frac{2n+\ds}{2}}e^{-c(t-T_0)_+}.
\end{multline}

{Now suppose $\be<\bels$. If $\|\mu^0-1\|_{L^2}\leq \d_\be$, then for any $0<t_1< t$, 
\begin{align}
\|\nabla^{\otimes n}\mu^t\|_{L^\infty} \leq \tl{\W}_n\Big(\be, \|\mu^0-1\|_{L^2}, \min(t_1/\beta,1)^{-1}\Big)\min((t-t_1)/\beta,1)^{-\frac{2n+\ds}{2}} e^{-ct},
\end{align}\label{eq:nabLinfexp'}
where $\tl{\W}_n: [0,\infty)^3\rightarrow [0,\infty)$ is a continuous, increasing function, vanishing if any of its arguments are zero.}
\end{lemma}
\begin{proof}
Choose some $t_0>0$. By \cref{lem:hypequiv}, $\mu^{t_0} \in L^2$, and so by estimate \eqref{eq:L2globRHS} of \cref{lem:L2glob},
\begin{align}
\forall t\geq t_0, \qquad \|\mu^t-1\|_{L^2} \leq \W(\beta, \|\mu^{t_0}-1\|_{L^2}, \Ec_\beta(\mu^{t_0}), e^{T_0})e^{-c(t-T_0)_+},
\end{align}
where $c>0$ depends on $\ds,\be$; $T_0$ is the minimal time such that $\|\mu^{t}-1\|_{L^1} \leq \d_{\be}$ for all $t\geq T_0$; and $\W$ is a continuous, increasing function of its arguments, vanishes if any of its arguments are zero.

We now use a time translation trick (for instance, cf. \cite[Lemmas 2.8 and 2.11]{CdCRS2023}). Given $t\geq t_0$, let $\frac{t_*-t_0}{\be} \coloneqq \min(\frac12, \frac{t-t_0}{2\be})$. Then by translating the initial time $t_0$ to zero, \cref{lem:hypequiv} implies
\begin{align}
\|\mu^t\|_{L^\infty} &\leq ((t-t^*)/\be)^{-\frac{\ds}{2}}\W'(\beta,\sup_{t_0\leq \tau \leq t} (\tau/\be)^{\frac14}\|\mu^\tau\|_{L^{\frac{2\ds}{2\ds-1}}}) \nn\\
&\leq \min((t-t_0)/2\beta, 1/2)^{-\frac{\ds}{2}}\W'\Big(\beta, 1+\W(\beta, \|\mu^{t_0}-1\|_{L^2}, \Ec_\beta(\mu^{t_0}), e^{T_0})e^{-c(t-T_0)_+}\Big) \nn\\
&\leq  \min((t-t_0)/\beta, 1)^{-\frac{\ds}{2}}\W''(\beta, \|\mu^{t_0}\|_{L^2}, e^{T_0}), \label{eq:t0Linfglob}
\end{align}
where $\W', \W''$ is another function of the same type as $\W$. Now let $t_1>t_0$. Note that $\sup_{t\geq t_1}\|\mu^t\|_{L^\infty}$ is controlled by the right-hand side of \eqref{eq:t0Linfglob}. Then by translating $t_1$ to the origin, we may apply \cref{lem:greg} to find that for any $n\geq 1$, for all $t\geq t_1$,
\begin{align}
\|\nabla^{\otimes n}\mu^t\|_{L^\infty} &\leq \min((t-t_1)/\beta,1/2)^{-\frac{n}{2}} \W_n(\beta, \sup_{t_1\leq \tau\leq t} \|\mu^\tau\|_{L^\infty}) \nn\\
&\leq  \min((t-t_1)/\beta,1/2)^{-\frac{n}{2}} \W_n\Big(\beta, \Big( \min((t_1-t_0)/\beta, 1)^{-\frac{\ds}{2}}\W''(\beta, \|\mu^{t_0}\|_{L^2}, e^{T_0})\Big)\Big). \label{eq:nabnmuLinfpre}
\end{align}
We now conclude the proof by interpolation. Given $n\geq 1$, from  Plancherel's theorem and Cauchy-Schwarz, it follows that
\begin{align}
\|\nabla^{\otimes n}\mu^t\|_{L^\infty} \leq \|\mu^t-1\|_{L^2}^{\frac12}\|\nabla^{\otimes 2n+\ds}\mu^t\|_{L^2}^{\frac12}.
\end{align}
Using that $\|\nabla^{\otimes 2n+\ds}\mu^t\|_{L^2}\leq \|\nabla^{\otimes 2n+\ds}\mu^t\|_{L^\infty}$ and applying \eqref{eq:nabnmuLinfpre}, \eqref{eq:L2globRHS}, we find that for all $t\geq t_1$
\begin{multline}
\|\nabla^{\otimes n}\mu^t\|_{L^\infty} \leq  \Big[\W(\beta,\|\mu^{t_0}-1\|_{L^2},\Ec_\beta(\mu^0))e^{-c(t-T_0)+}\Big]^{\frac12}\Big[\min((t-t_1)/\beta,1/2)^{-\frac{2n+\ds}{2}}\\
\times \W_{2n+\ds}\Big(\beta, \Big( \min((t_1-t_0)/\beta, 1)^{-\frac{\ds}{2}}\W''(\beta, \|\mu^{t_0}\|_{L^2}, e^{T_0})\Big)\Big)\Big]\\
\leq \W'''\Big(\beta, \|\mu^{t_0}-1\|_{L^2}, \Ec_\beta(\mu^0), e^{T_0},  \min((t_1-t_0)/\beta,1)^{-1}\Big)\min((t-t_1)/\beta,1/2)^{-\frac{2n+\ds}{2}} e^{-c(t-T_0)_+/2}
\end{multline}
where $\W''':[0,\infty)^5\rightarrow [0,\infty)$ is a continuous, increasing function, vanishing if any of its arguments are zero. This completes the proof {of the first assertion.}

{For the second assertion, we just need to modify the preceding argument. By \cref{lem:aL2exp}, we have
\begin{align}\label{eq:aL2expit}
\forall t\ge 0,\qquad \|\mu^t -1\|_{L^2} \leq \|\mu^{0}-1\|_{L^2}e^{-{c}t}.
\end{align}
Given $t\geq 0$, let $\frac{t_*}{\be} \coloneqq \min(\frac12, \frac{t}{2\be})$. Then \cref{lem:hypequiv} and \eqref{eq:aL2expit} imply
\begin{align}
\|\mu^t\|_{L^\infty} &\leq ((t-t^*)/\be)^{-\frac{\ds}{2}}\tl{\W}(\beta,\sup_{t_0\leq \tau \leq t} (\tau/\be)^{\frac14}\|\mu^\tau\|_{L^{\frac{2\ds}{2\ds-1}}}) \nn\\
&\leq  \min(t/\beta, 1)^{-\frac{\ds}{2}}\tl{\W}(\beta,\|\mu^0\|_{L^2}).
\end{align}
For $t_1>0$, we use \cref{lem:greg} to find that for any $n\geq 1$, for all $t\geq t_1$,
\begin{align}
\|\nabla^{\otimes n}\mu^t\|_{L^\infty} &\leq \min((t-t_1)/\beta,1/2)^{-\frac{n}{2}} \W_n(\beta, \sup_{t_1\leq \tau\leq t} \|\mu^\tau\|_{L^\infty}) \nn\\
&\leq \min((t-t_1)/\beta,1/2)^{-\frac{n}{2}} \W_n\Big(\beta, \min(t_1/\beta, 1)^{-\frac{\ds}{2}}\tl{\W}(\beta,\|\mu^0\|_{L^2})\Big).
\end{align}
Interpolating as before yields the desired conclusion.
}
\end{proof}

For the purposes of the application in \cref{sec:MFE}, we record the following remark concerning the case when $\mu^0\in W^{2,\infty}$.

\begin{remark}\label{rem:W2infglob}
Since $\mu^0\in W^{2,\infty}$, an examination of the proof of \cref{prop:LWPa} shows that there exists a $T_{LWP} = T_{LWP}(\|\mu^0\|_{W^{2,\infty}}, \be)>0$, such that the solution of \eqref{eq:lima} with initial datum satisfies $\|\mu\|_{C([0,T_{LWP}], W^{2,\infty})} \leq 2\|\mu^0\|_{W^{2,\infty}}$.

{Suppose first that $\be<\min(\bels,\bec)$ and there exists a $T_0>0$ such that for all $t\geq T_0$, $\|\mu^t-1\|_{L^1}\leq\delta_\be$ (the smallness threshold from above).} Choose $t_0=T_{LWP}/4$ and $t_1=T_{LWP}/2$. Then for any $t\geq T_{LWP}$, it follows from \eqref{eq:nabLinfexp} that for $n\leq 2$,
\begin{align}
\|\nabla^{\otimes n}(\mu^t-1)\|_{L^\infty} &\leq  {\W}_n\Big(\beta, \|\mu^{T_{LWP}/4}-1\|_{L^2}, \Ec_\beta(\mu^0), e^{T_0},  \min(T_{LWP}/4\beta,1)^{-1}\Big) \nn\\
&\ph\times \min(T_{LWP}/2\beta,1)^{-\frac{2n+\ds}{2}}e^{-c(t-T_0)_+} \nn\\
&\leq \tl{\W}(\beta , \|\mu^0\|_{W^{2,\infty}}, \Ec_\beta(\mu^0), e^{T_0})e^{-c(t-T_0)_+},
\end{align}
where $\tl{\W}:[0,\infty)^4 \rightarrow [0,\infty)$ is a continuous, increasing function, vanishing if any of its arguments are zero. Therefore, for all $t\geq 0$,
\begin{align}
\|\mu^t-1\|_{W^{2,\infty}} \leq 2\|\mu^0\|_{W^{2,\infty}}\indic_{t< T_{LWP}} + \tl{\W}(\beta , \|\mu^0\|_{W^{2,\infty}}, \Ec_\beta(\mu^0), e^{T_0})e^{-c(t-T_0)_+}\indic_{t\geq T_{LWP}}.
\end{align}

{Now suppose that $\be<\bels$ and $\|\mu^0-1\|_{L^2}\leq \d_\be$. Then choosing $t_1=T_{LWP}/2$, it follows from \eqref{eq:nabLinfexp'} that for $n\leq 2$ and $t\geq T_{LWP}$,
\begin{align}
\|\nabla^{\otimes n}(\mu^t-1)\|_{L^\infty} &\leq \tl{\W}_n\Big(\beta, \|\mu^0-1\|_{L^2}, \min(T_{LWP}/2\beta,1)^{-1}\Big)\min(T_{LWP}/2\beta,1)^{-\frac{2n+\ds}{2}}e^{-ct}.
\end{align}
Therefore, for all $t\geq 0$,
\begin{align}
\|\mu^t-1\|_{W^{2,\infty}} \leq 2\|\mu^0\|_{W^{2,\infty}}\indic_{t< T_{LWP}} +\tl{\W}(\beta , \|\mu^0\|_{W^{2,\infty}})e^{-ct}\indic_{t\geq T_{LWP}},
\end{align}
where now $\tl{\W}: [0,\infty)^2 \rightarrow [0,\infty)$ is a continuous, increasing function of its arguments, vanishing if any of them are zero.
}
\end{remark}

\section{Stability}\label{sec:Instab}
It is clear that the uniform distribution $\muu \coloneqq 1$ is a stationary solution of \eqref{eq:lima}. We linearize the equation around $\muu$, obtaining
\begin{equation}\label{eq:lin}
\begin{cases}
\p_t\nu^t = \paren*{\frac1\beta\D+\cdd\Dm^{2-\ds}}\nu^t\\
\nu^t|_{t=0} = \nu^0.
\end{cases}
\end{equation}
Letting $L_{\beta} \coloneqq \paren*{\frac1\beta\D-\cdd\Dm^{2-\ds}}$, we find for $k\in \Z^\ds\setminus\{0\}$,
\begin{align}
L_{\beta}e^{2\pi i k\cdot x} = \paren*{\cdd (2\pi|k|)^{2-\ds} - \frac{4\pi^2}{\beta} |k|^2}e^{2\pi i k\cdot x}.
\end{align}
Solving
\begin{align}
\paren*{\cdd (2\pi |k|)^{2-\ds} - \frac{4\pi^2}{\beta} |k|^2} {\leq } 0 \Longleftrightarrow |k| \geq \paren*{\frac{\beta\cdd}{(2\pi)^\ds}}^{1/\ds} = \paren*{\frac{\be}{\bels}}^{1/\ds}.
\end{align}
Since $|k|\geq 1$, we see that $\beta \geq \bels$ is necessary to have any eigenvalues with positive real part (all eigenvalues of $L_\beta$ are real, so we may drop the real part). We see that if $\beta>\bels$, then there exists at least one eigenfunction $e^{2\pi i k\cdot x}$ of $L_{\beta}$ with eigenvalue $\la_{\beta,|k|} \coloneqq \cdd (2\pi |k|)^{2-\ds} - \frac{4\pi^2}{\beta} |k|^2 >0$. Obviously, the eigenvalue only depends on $k$ through $|k|$, so there is no ambiguity in our notation $\la_{\beta,|k|}$. {If $\ds\geq 2$, we note that the eigenvalues $\la_{\beta,|k|}$ are strictly decreasing in $|k|$. If $\ds=1$, then noting that the function $\varphi(r) = 2r - \frac{2\pi r^2}{\be}$, for $r\geq 0$, has a global maximum at $r= \frac{\be}{2\pi}$, we see that $\la_{\be,|k|}$ is strictly increasing for $0\leq |k|< \frac{\be}{2\pi}$ and strictly decreasing for $|k|>\frac{\be}{2\pi}$. For later reference, we note that for all $\ds\geq 1$, we have that $\la_{\be,n}<n\la_{\be,1}$ for any positive integer $n$.
}

\subsection{Nonlinear instability}
We now show that $\muu$ is nonlinearly unstable for equation \eqref{eq:lima} if $\beta>\bels$.

\begin{prop}\label{prop:NLinstab}
Let $\ds\geq 1$ and $\beta>\bels$. There exist constants $C,\ep_0>0$ depending on $\ds,\beta$, such that if $0< \ep\leq \ep_0$, then there is a solution $\mu$ to equation \eqref{eq:lima} with initial datum $1+2\ep \cos(2\pi k\cdot x)$, for wave vector $|k|=1$, and lifespan $[0,C^{-1}\log(1/\ep)]$, such that
\begin{equation}
\frac12 \leq \|\mu^{C^{-1}\log(1/\ep)} - 1\|_{L^1}.
\end{equation}
\end{prop}

The proof of \cref{prop:NLinstab} relies on an iterative argument originally due to Grenier \cite{Grenier2000nlinstab}, which we learned of from \cite{HanKwan2016b} (see also \cite{HkN2016} for other applications in kinetic theory). We divide the proof over several lemmas.

\medskip
Let $\nu_1^0 = 2\cos(2\pi k\cdot x)$, for any $|k|=1$, and set $\nu_1^t \coloneqq e^{\la_{\be,1}t}\nu_1^0$. For $\ep>0$, considering a formal solution to equation \eqref{eq:lima} given by $\mu^t = \sum_{j=0}^\infty \ep^j \nu_j^t$, where $\nu_0 \coloneqq 1$, we see from substitution into the equation and matching powers of $\ep$ on both sides that the $\nu_j$, for $j\geq 2$, solve
\begin{equation}\label{eq:nuj}
\begin{cases}
\displaystyle\paren*{\p_t - L_{\be}}\nu_j^t = -\sum_{m=1}^{j-1} \div\paren*{\nu_{j-m}^t\nabla\g\ast\nu_{j}^t } \\
\nu_{j}^t|_{t=0} = 0.
\end{cases}
\end{equation}
In other words, $\nu_j$ solves a linear inhomogeneous PDE with a source term determined by lower-order functions and with zero initial datum. Of course, there is no reason for the series $\sum_{j=0}^\infty \ep^j \nu_j$ to converge, so we must truncate it to some order $n$, that is only consider an approximate solution
\begin{align}\label{eq:muadef}
\mu_{\text{app}} \coloneqq \sum_{j=0}^n \ep^j \nu_j.
\end{align}
{Later in proving \cref{prop:NLinstab}, we will take $n=2$.} One checks that
\begin{equation}\label{eq:muapp}
\p_t\mua^t + \div(\mua^t\nabla\g\ast\mua^t) - \frac1\be \D\mua^t = \sum_{1 \leq j,\ell \leq n : j+\ell\geq n+1} \ep^{j+\ell}\div(\nu_j^t\nabla\g\ast\nu_\ell^t) \eqqcolon  \Ra^t. 
\end{equation}

Using Duhamel's formula on \eqref{eq:nuj}, we see that $\nu_j^t$ is explicitly given by
\begin{equation}\label{eq:nujDuh}
\nu_j^t = \int_0^t e^{(t-\tau)L_{\be}}\paren*{-\sum_{m=1}^{j-1} \div\paren*{\nu_{j-m}^\tau\nabla\g\ast\nu_{m}^\tau } } d\tau.
\end{equation}
Let us start with $j=2$. Then using the explicit form of $\nu_1^t$, we compute
\begin{align}
\div\left(\nu_1^\tau \nabla\g\ast \nu_1^\tau\right) &= e^{2\la_{\be,1}\tau}\Big(\div(e^{2\pi i k\cdot x}\nabla\g\ast e^{2\pi i k\cdot x}) + \div(e^{-2\pi i k\cdot x}\nabla\g\ast e^{-2\pi i k\cdot x})\Big) \nn\\
&=-4\cdd e^{2\la_{\be,1}\tau}(2\pi|k|)^{2-\ds}\Big(e^{4\pi i k\cdot x}+e^{-4\pi i k\cdot x}\Big).
\end{align}
And therefore
\begin{align}
\nu_2^t &= 4\cdd|2\pi k|^{2-\ds}\Big(e^{4\pi i k\cdot x} +e^{-4\pi i k\cdot x} \Big)\int_0^t e^{(t-\tau)\la_{\be,2}}e^{2\la_{\be,1}\tau}d\tau \nn\\
&= 4\cdd|2\pi k|^{2-\ds}\Big(e^{4\pi i k\cdot x} +e^{-4\pi i k\cdot x} \Big)e^{t\la_{\be,2}}\frac{e^{(2\la_{\be,1}-\la_{\be,2})t}-1}{2\la_{\be,1}-\la_{\be,2}}.
%\begin{cases} t,  & {2\la_{\be,1} = \la_{\be,2}}  \\ \frac{e^{(2\la_{\be,1}-\la_{\be,2})t}-1}{2\la_{\be,1}-\la_{\be,2}} , & {2\la_{\be,1} \neq \la_{\be,2}}. \end{cases}
\end{align}
Recall that $n\la_{\be,1}>\la_{\be,2}$ for all $n\in\N$. We note that $\nu_2^t$ has the (less explicit) form
\begin{equation}\label{eq:nu2form}
\nu_2^t = \sum_{\ell=-2}^{2} C_{2,\ell}^t e^{\ell(2\pi i)k\cdot x}, \quad \text{where} \quad \forall t\geq 0, \ \ 0\leq C_{2,\ell}^t \leq \mathsf{C}_{2,\ell} e^{\ell\la_{\be,2}t},
\end{equation}
$C_{2,\ell}^t= C_{2,-\ell}^t$, $C_{2,0}^t=0$, and the constant $\mathsf{C}_2>0$ depends only on $\ds,\la_{\be,1},\la_{\be,2}$. As we show in the next lemma, all the higher-order functions $\nu_j$, for $j\geq 2$, satisfy a form analogous to \eqref{eq:nu2form}.

\begin{lemma}\label{lem:nujbnd}
For every $j\geq 2$, there exist smooth functions $C_{j,\ell}^t: [0,\infty) \rightarrow [0,\infty)$, for $-j\leq \ell\leq j$, such that
\begin{align}
\nu_j^t = \sum_{\ell=-j}^j C_{j,\ell}^t e^{\ell(2\pi i)k\cdot x},
\end{align}
where $C_{j,\ell}^t = C_{j,-\ell}^t$, $C_{j,0}^t=0$, and
\begin{equation}
\forall t\geq 0, \qquad C_{j,\ell}^t \leq \mathsf{C}_{j,\ell} e^{\ell\la_{\be,1}t},
\end{equation}
for a constant $\mathsf{C}_{j,\ell}>0$ depending only on $\ds,\la_{\be,1},\ldots,\la_{\be,j}$. 
\end{lemma}
\begin{proof}
The proof is by induction on $j$. We have shown above the base case $j=2$. We make the following induction hypothesis: there exists $n\geq 2$ such that for every $1\leq j\leq n$, it holds that
\begin{equation}\label{eq:nujrepIH}
\nu_j^t = \sum_{\ell=-j}^{j} C_{j,\ell}^t e^{\ell (2\pi i)k\cdot x}, \quad \text{where}  \quad  \forall t\geq 0, \ \ 0\leq C_{j,\ell}^t \leq \mathsf{C}_{j,\ell} e^{\ell\la_{\be,1} t},
\end{equation}
$C_{j,\ell}^t = C_{j,-\ell}^t$, $C_{j,0}^t=0$, and the constant $\mathsf{C}_{j,\ell}>0$ depends only on $\ds,\la_{\be,1},\ldots,\la_{\be,j}$. We will show that \eqref{eq:nujrepIH} also holds for $j=n+1$.

Applying the induction hypothesis \eqref{eq:nujrepIH} to the formula \eqref{eq:nujDuh}, we find
\begin{align}
\nu_{n+1}^t  &= - \sum_{m=1}^{n}\sum_{\ell=-m}^m \sum_{\ell'=-(n+1-m)}^{n+1-m}\int_0^t C_{n+1-m,\ell'}^\tau C_{m,\ell}^\tau e^{(t-\tau)L_{\be}}\div\paren*{ e^{\ell'(2\pi i)k\cdot()} \nabla\g\ast e^{\ell(2\pi i)k\cdot() } }d\tau \nn\\
&=\sum_{m=1}^{n}\sum_{\ell,\ell' : \ell+\ell'\neq 0}\int_0^t C_{n+1-m,\ell'}^\tau C_{m,\ell}^\tau e^{(t-\tau)\la_{\be,\ell+\ell'}}\cdd\Big(\frac{ \ell' |2\pi \ell k|^{2-\ds}}{\ell} \nn\\
&\ph\qquad + |2\pi\ell k|^{2-\ds} \Big) e^{2\pi i (\ell+\ell')k\cdot x} d\tau.
\end{align}
Observe that for $1\leq m\leq n$, $-m\leq \ell\leq m$, $-(n-m+1)\leq \ell'\leq (n-m+1)$,
\begin{align}
&\int_0^t C_{n+1-m,\ell'}^\tau C_{m,\ell}^\tau e^{(t-\tau)\la_{\be,\ell+\ell'}}d\tau \nn\\
&\leq {\mathsf{C}_{n+1-m,\ell'}\mathsf{C}_{m,\ell}}\int_0^t e^{(t-\tau)\la_{\be,\ell+\ell'}} e^{(\ell+\ell')\la_{\be,1}\tau} d\tau \nn\\
&= {\mathsf{C}_{n+1-m,\ell'}\mathsf{C}_{m,\ell}}e^{t\la_{\be,\ell+\ell'}} \frac{e^{\left((\ell+\ell')\la_{\be,1}  - \la_{\be,\ell+\ell'}\right)t}-1}{(\ell+\ell')\la_{\be,1}- \la_{\be,\ell+\ell'}}.
%\begin{cases}t, & {(n+1)\la_{\be,1} = \la_{\be,n+1}} \\
%\displaystyle \frac{e^{\left((n+1)\la_{\be,1}  - \la_{\be,n+1}\right)t}}{(n+1)\la_{\be,1}- \la_{\be,n+1}} , & {(n+1)%\la_{\be,1} \neq \la_{\theta,n+1}}.
%\end{cases}
\end{align}
For $-n-1 \leq j\leq n+1$, define the function
\begin{multline}
C_{n+1,j}^t \coloneqq \sum_{m=1}^n \sum_{\substack{-m\leq \ell \leq m, -(n-m+1)\leq \ell'\leq (n-m+1) \\ \ell+\ell' = j}}\\
\cdd\Big(\frac{ \ell' |2\pi \ell k|^{2-\ds}}{\ell} + |2\pi\ell k|^{2-\ds} \Big)\int_0^t C_{n+1-m,\ell'}^\tau C_{m,\ell}^\tau e^{(t-\tau)\la_{\be,\ell+\ell'}}d\tau,
\end{multline}
which is evidently smooth, it follows from above that $C_{n+1,j}^t = C_{n+1,-j}^t$ and $0\leq C_{n+1,j}^t \leq \mathsf{C}_{n+1,j} e^{j\la_{\be,1} t}$, for some constant $\mathsf{C}_{n+1,j}>0$ depending only on $\ds,\la_{\be,1},\ldots,\la_{\be,j}$. This completes the induction step and therefore the proof of the lemma.
\end{proof}

\begin{remark}\label{rem:muaRaSob}
\cref{lem:nujbnd} shows that for every integers $m\geq 0$ and $j\geq 1$, we have the Sobolev bound
\begin{align}
\|\nu_j^t\|_{\dot{W}^{m,\infty}} \lesssim_{\ds,m} \sum_{\ell=-j}^j (2\pi \ell |k|)^m \mathsf{C}_{j,\ell} e^{\ell\lambda_{\be,1} t} \lesssim (2\pi j|k|)^m \mathsf{C}_j e^{j\la_{\be,1} t}.
\end{align}
Recalling the definition \eqref{eq:muapp} of $\mua$ and the triangle inequality, we then have
\begin{align}\label{eq:muaSob}
\|\mua-1\|_{W^{m,\infty}} &\leq \sum_{j=1}^n \ep^{j} \|\nu_j\|_{W^{m,\infty}} \leq 1+C\sum_{j=1}^n \ep^j\mathsf{C}_j e^{j\la_{\be,1}t} \leq  8C\paren*{\max_{1\leq j\leq n}\mathsf{C}_j}\ep e^{\la_{\be,1}t},
\end{align}
where the final inequality holds provided $\ep e^{\la_{\be,1}t} \leq \frac12$. Here, $C$ depends only on $\ds,m$. Similarly, if $\ep e^{\la_{\be,1}t} \leq \frac12$, then
\begin{align}
\|\Ra^t\|_{W^{m,\infty}} &\leq \sum_{1\leq j,j' \leq n: j+j' \geq n+1} \ep^{j+\ell}\|\div(\nu_j^t\nabla\g\ast\nu_{j'}^t)\|_{W^{m,\infty}} \nn\\
&\leq C\sum_{1\leq j,j' \leq n: j+j' \geq n+1} \ep^{j+j'} \mathsf{C}_{j}\mathsf{C}_{j'} e^{(j+j')\la_{\be,1}t} \nn\\
&\leq 8\paren*{\max_{1\leq j\leq n}\mathsf{C}_j}^2 \ep^{n+1} e^{(n+1)\la_{\be,1}t}. \label{eq:RaSob}
\end{align}
\end{remark}

\medskip
Next, we show that there is a genuine solution $\mu$ to equation \eqref{eq:lima} with initial datum $1+\ep\nu_1^0$, such that $\|\mu^t-\mu_{\text{app}}^t\|_{L^p} = O(\ep^{n+1})$ on time scales of order $\log(1/\ep)$. It is convenient to choose the exponent $p=2$, though any choice $p\in (1,\infty)$ would work with more effort. For $\ds\geq 2$, we could even choose $p=1$; but for $\ds=1$, this does not quite work, since $\nabla\g\ast$, which is the Hilbert transform, is unbounded on $L^1$. With more effort, it would be possible to replace $L^p$ with a Sobolev space $W^{k,p}$, but we have no need for such generality.   %  For $\ds\geq 2$, it is convenient to choose the exponent $p=1$. For $\ds=1$, this does not work for technical reasons (if $\ds=1$, $\nabla\g\ast$ is unbounded as a convolution operator on $L^1$, since it is the Hilbert transform); so it is instead convenient to choose $p=2$, though arbitrary $1<p<\infty$ would work.  %

\begin{lemma}\label{lem:mumua}
Suppose $\ds\geq 1$ and $\be>\bels$. Let $\mu$ be the unique solution to equation \eqref{eq:lima} with initial datum $\mu^0 = 1+\ep\nu_1^0$. Then the maximal time of existence for $\mu$ satisfies
\begin{equation}\label{eq:TmaxT*}
T_{\max} \geq T_* \coloneqq \min\paren*{\frac{\log(1/2\ep)}{\la_{\be,1}}, \frac{1}{2(C+(n+1)\la_{\be,1})}\log(\frac{(n+1)\la_{\be,1}}{C\ep^{2(n+1)}})}
\end{equation}
for a constant $C>0$ depending only on $\ds,n,\be$. Moreover,
\begin{align}
\forall t\in [0,T_*], \qquad \|\mu^t-\mua^t\|_{L^2}^2  \leq C\frac{\ep^{2(n+1)}e^{2((n+1)\la_{\be,1} + C)t}}{2(n+1)\la_{\be,1}}.
\end{align}
\end{lemma}
\begin{proof}
Using equations \eqref{eq:lima} and \eqref{eq:muapp}, we compute
\begin{multline}
\p_t(\mu-\mua) = -\div\paren*{\mua\nabla\g\ast(\mu-\mua)} - \div\paren*{(\mu-\mua)\nabla\g\ast\mu }\\
 + \frac1\be\D(\mu^t-\mua^t) {-} \Ra.
\end{multline}
Observe from integration by parts that
\begin{align}
-\int_{\T^\ds}|\mu^t-\mua^t|^{p-2}(\mu^t-\mua^t)\div\paren*{({\mu^t-\mua^t})\nabla\g\ast{\mu^t} } dx &= \cdd\int_{\T^\ds}|\mu^t-\mua^t|^p\Dm^{2-\ds}(\mu^t)dx \nn\\
&\ph -\frac{1}{p}\int_{\T^\ds}\nabla|\mu^t-\mua^t|^p \cdot\nabla\g\ast(\mu^t) dx \nn\\
&= {\cdd}\left(1-\frac{1}{p}\right)\int_{\T^\ds}|\mu^t-\mua^t|^{p}\Dm^{2-\ds}(\mu^t)dx.
\end{align}
Similarly,
\begin{align}
\frac1\be \int_{\T^\ds}|\mu^t-\mua^t|^{p-2}(\mu^t-\mua^t)\D(\mu^t-\mua^t)dx = -\frac{4(p-1)}{p^2\be}\int_{\T^\ds} |\nabla |\mu^t-\mua^t|^{\frac{p}{2}} |^2 dx.
\end{align}
Hence,
\begin{multline}
\frac{d}{dt} \|\mu^t-\mua\|_{L^p}^p = {\cdd}(p-1)\int_{\T^\ds}|\mu^t-\mua^t|^{p}\Dm^{2-\ds}(\mu^t)dx\\
-\frac{4(p-1)}{p\be}\int_{\T^\ds} |\nabla |\mu^t-\mua^t|^{\frac{p}{2}} |^2 dx -p\int_{\T^\ds}|\mu^t-\mua^t|^{p-2}(\mu^t-\mua^t)\nabla\mua^t\cdot\nabla\g\ast(\mu^t-\mua^t)dx \\
+ \cdd p\int_{\T^\ds}|\mu^t-\mua^t|^{p-2}(\mu^t-\mua^t)\mua^t\Dm^{2-\ds}(\mu^t-\mua^t)dx\\
  {-} p\int_{\T^\ds}|\mu^t-\mua^t|^{p-2}(\mu^t-\mua^t)\Ra^t dx.
\end{multline}

We choose $p=2$ and use H\"older's inequality together with $\nabla\g\in L^1$ ($\ds \geq 2$) or the $L^2$ boundedness of the Hilbert transform ($\ds=1$) to obtain
\begin{align}
\int_{\T^\ds}\left|(\mu^t-\mua^t)\nabla\mua^t\cdot\nabla\g\ast(\mu^t-\mua^t)\right| dx &\leq \|\nabla\mua^t\|_{L^\infty} \|\mu^t-\mua^t\|_{L^2}\|\nabla\g\ast(\mu^t-\mua^t)\|_{L^2} \nn\\
&\lesssim \ep e^{\la_{\be,1}t}\|\mu^t-\mua^t\|_{L^2}^2,
\end{align}
where the final line follows from estimate \eqref{eq:muaSob} from \cref{rem:muaRaSob}, assuming $\ep e^{\la_{\be,1}t}\leq\frac12$. Similarly, if $\ds\geq 2$, then we may use that the operator $\Dm^{2-\ds}$ is bounded on $L^2$ to estimate.
\begin{align}
\int_{\T^\ds}\left|(\mu^t-\mua^t)\mua^t\Dm^{2-\ds}(\mu^t-\mua^t)\right|dx &\leq \|\mua^t\|_{L^\infty}\|\Dm^{2-\ds}(\mua^t-\mu^t)\|_{L^2} \|\mua^t-\mu^t\|_{L^2} \nn\\
&\lesssim \|\mu^t-\mua^t\|_{L^2}^2. 
\end{align}
If $\ds=1$, then we instead argue
\begin{align}
\int_{\T^\ds}\left|(\mu^t-\mua^t)\mua^t\Dm(\mu^t-\mua^t)\right|dx &\leq \|\mua^t\|_{L^\infty}  \|\mu^t-\mua^t\|_{L^2} \|\Dm(\mu^t-\mua^t)\|_{L^2} \nn\\
&\lesssim \|\mu^t-\mua^t\|_{L^2}\|\mu^t-\mua^t\|_{\dot{H}^1}
\end{align}
Using estimate \eqref{eq:RaSob} from \cref{rem:muaRaSob},
\begin{align}
\int_{\T^\ds}\left|(\mu^t-\mua^t)\Ra^t\right| dx \leq \|\mu^t-\mua^t\|_{L^2}\|\Ra^t\|_{L^2} &\leq \|\mu^t-\mua^t\|_{L^2}\|\Ra^t\|_{L^\infty} \nn\\
&\lesssim \ep^{n+1} e^{(n+1)\la_{\be,1}t}\|\mu^t-\mua^t\|_{L^2}. \label{eq:mumuaRapp}
\end{align}
Lastly, if $\ds\geq 2$, then using H\"older's and triangle inequalities,
\begin{align}
\int_{\T^\ds}\left||\mu^t-\mua^t|^{2}\Dm^{2-\ds}(\mu^t)\right|dx &\leq \|\mu^t-\mua^t\|_{L^2}^2 \|\Dm^{2-\ds}\mua^t\|_{L^\infty} \nn\\
&\ph+ \int_{\T^\ds}|\mu^t-\mua^t|^{2} |\Dm^{2-\ds}(\mu^t-\mua^t)|dx \nn\\
&\lesssim \ep e^{\la_{\be,1}t}\|\mu^t-\mua^t\|_{L^2}^2 + \|\mu^t-\mua^t\|_{L^1} \|\nabla(\mu^t-\mua^t)\|_{L^2}^2,
\end{align}
where the final line follows from {$\Dm^{2-\ds}(\mua^t) = \Dm^{2-\ds}(\mua^t-1)$ and} estimate \eqref{eq:fLrHLS} from \cref{lem:FIests} and \cref{rem:muaRaSob}. %For the first term on the right-hand side,
%\begin{align}
%\|\mu^t-\mua^t\|_{L^2}^2 \|\Dm^{2-\ds}\mua^t\|_{L^\infty} \lesssim \ep e^{\la_{\be,1}t} \|\mu^t-\mua^t\|_{L^2}^2.
%\end{align}
%For the second term, we modify the proof of the estimate . By H\"older's inequality and Hardy-Littlewood-Sobolev, for $p\in (1,\frac{\ds}{\ds-2})$
%\begin{align}
%\int_{\T^\ds}|\mu^t-\mua^t|^{2} |\Dm^{2-\ds}(\mu^t-\mua^t)|dx \leq \|\mu^t-\mua^t\|_{L^{2p}}^2 \|\mu^t-\mua^t\|_{L^q},
%\end{align}
%where $\frac{\ds-2}{\ds} = \frac{1}{q} - \frac{p-1}{p}$. Using Gagliardo-Nirenberg interpolation,
%\begin{align}
%\|f-1\|_{L^{2p}} \leq \|f-1\|_{L^2}^{1-\theta_p} \|\nabla f\|_{L^2}^{\theta_p},\\
%\|f-1\|_{L^q} \leq \|f-1\|_{L^1}^{1-\theta_q} \|\nabla f\|_{L^2}^{\theta_q},
%\end{align}
%where
%\begin{align}
%\frac{1}{2p} = \theta_p(\frac12-\frac1\ds) + \frac{(1-\theta_p)}{2},  \qquad \theta_p = \frac{\ds(p-1)}{2p}\\
%\frac{1}{q} = \theta_q(\frac12-\frac1\ds) + (1-\theta_q), \qquad \theta_q = \frac{2\ds}{\ds+2}(1-\frac{1}{q}).
%\end{align}
If $\ds=1$, then instead,
\begin{align}
\int_{\T^\ds}\left||\mu^t-\mua^t|^{2}\Dm(\mu^t)\right|dx &\leq \|\mu^t-\mua^t\|_{L^4}^2 \|\Dm(\mu^t-\mua^t)\|_{L^2} + \|\Dm(\mua^t)\|_{L^\infty} \|\mu^t-\mua^t\|_{L^2}^2 \nn\\
&\lesssim \|\mua^t-\mu^t\|_{L^2}^{3/2} \|\mu^t-\mua^t\|_{\dot{H}^1}^{3/2} + \ep e^{\la_{\be,1}t} \|\mu^t-\mua^t\|_{L^2}^2.
\end{align}

Putting together our estimates and using $ab\leq \frac{a^2+b^2}{2}$ on \eqref{eq:mumuaRapp}, for $\ds\geq 2$, we then arrive at
\begin{multline}\label{eq:d2preGron}
\frac{d}{dt} \|\mu^t-\mua^t\|_{L^2}^2 \leq C\paren*{\|\mu^t-\mua^t\|_{L^2}^2 + \ep^{2(n+1)}e^{2(n+1)\la_{\be,1}t}} \\
+ \Big(C\|\mu^t-\mua^t\|_{L^1}- \frac2\be\Big)\|\nabla(\mu^t-\mua^t)\|_{L^2}^2,
\end{multline}
for some constant $C>0$ depending on $\ds,n$ and the eigenvalues $\la_{\be,1},\ldots,\la_{\be,n}$, provided $t$ is sufficiently small so that $\ep e^{\la_{\be,1}t} \leq \frac{1}{2}$. To obtain a closed estimate for $\|\mu^t-\mua^t\|_{L^2}^2$, we use a continuity argument.  Let $T_*$ be the maximal time in $(0,\frac{\log(1/2\ep)}{\la_{\be,1}}]$ such that
\begin{align}\label{eq:muT*}
\forall t\leq T_*, \qquad \|\mu^{t}-\mua^{t}\|_{L^2} \leq \frac{2}{\be C}.
\end{align}
We adopt the convention that $T_*\coloneqq \frac{\log(1/2\ep)}{\la_{\be,1}}$ if no such time exists. Assuming such a $T_*$ exists,  continuity in time implies $\|\mu^{T_*}-\mua^{T_*}\|_{L^2}= \frac{2}{\be C}$. Then since $\|\mu^t-\mua^t\|_{L^1}\leq \|\mu^t-\mua^t\|_{L^2}$, applying the Gr\"{o}nwall-Bellman lemma to \eqref{eq:d2preGron} (remembering that $\|\mu^0-\mua^0\|_{L^2}=0$) we obtain
\begin{align}
\forall t\in [0,T_*], \qquad \|\mu^t-\mua^t\|_{L^2}^2 \leq C\frac{\ep^{2(n+1)}}{2(n+1)\la_{\be,1}}e^{2(n+1)\la_{\be,1}t} e^{Ct}.
\end{align}
We claim that 
\begin{align}
T_*\geq \min\paren*{\frac{\log(1/2\ep)}{\la_{\be,1}}, \frac{1}{2(C+(n+1)\la_{\be,1})}\log(\frac{4\be(n+1)\la_{\be,1}}{C^2\ep^{2(n+1)}})}.
\end{align}
Otherwise, one checks that
\begin{align}
\frac{2}{\be C}=\|\mu^{T_*}-\mua^{T_*}\|_{L^2}^2 \leq C\frac{\ep^{2(n+1)}e^{2(n+1)\la_{\be,1}T_*}}{2(n+1)\la_{\be,1}} e^{2CT_*} < \frac{2}{\be C},
\end{align}
which is a contradiction.

%Applying the Gr\"onwall-Bellman lemma to \eqref{eq:d2preGron}, we obtain
%\begin{equation}
%\forall t \in [0,\frac{\log(1/2\ep)}{\la_{\be,1}}], \qquad \|\mu^t-\mua^t\|_{L^2}^2 \leq C\frac{\ep^{n+1}}{(n+1)\la_{\be,1}}e^{(n+1)\la_{\be,1}t} e^{Ct},
%\end{equation}
%as desired.
%If $d=1$, then we choose $p=2$ and using H\"older's inequality
%\begin{align}
%\int_{\T^\ds}\left| (\mu^t-\mua^t)\nabla\mua^t\cdot\nabla\g\ast(\mu^t-\mua^t)\right| dx &\leq \|\nabla\mua^t\|_{L^\infty} \|\mu^t-\mua^t%\|_{L^2} \|\nabla\g\ast(\mu^t-\mua^t)\|_{L^2} \nn\\
%&\lesssim \ep e^{\la_{\be,1}t}\|\mu^t-\mua^t\|_{L^2}^2 ,
%\end{align}
%where we also use Plancherel's theorem and the estimate \eqref{eq:muaSob}.
If $\ds=1$, then putting our estimates together and using Young's product inequality to absorb the $\dot{H}^1$ factors into the diffusion term, we arrive at
\begin{equation}\label{eq:d1preGron}
\frac{d}{dt} \|\mu^t-\mua^t\|_{L^2}^2 \leq C\Big( \|\mu^t-\mua^t\|_{L^2}^6 + \|\mu^t-\mua^t\|_{L^2}^2 + \ep^{2(n+1)}e^{2(n+1)\la_{\be,1}t}\Big),
\end{equation}
for some constant $C>0$ depending on $\ds,\be,\la_{\be,1},\ldots,\la_{\be,n}$, provided $\ep e^{\la_{\be,1}t} \leq \frac12$. To obtain a closed estimate for $\|\mu^t-\mua^t\|_{L^2}^2$, we again use a continuity argument. Let $T_*$ be the maximal time in $(0,\frac{\log(1/2\ep)}{\la_{\be,1}}]$ such that
\begin{align}\label{eq:muT*'}
\forall t\leq T_*, \qquad \|\mu^{t}-\mua^{t}\|_{L^2} \leq 1.
\end{align}
We adopt the convention that $T_*\coloneqq \frac{\log(1/2\ep)}{\la_{\be,1}}$ if no such time exists. Assuming such a $T_*$ exists,  continuity in time implies $\|\mu^{T_*}-\mua^{T_*}\|_{L^2}= 1$. Applying Gr\"onwall-Bellman to \eqref{eq:d1preGron}, it follows that
\begin{align}
\forall t\in [0,T_*], \qquad \|\mu^t-\mua^t\|_{L^2}^2  \leq C\frac{\ep^{2(n+1)}e^{2(n+1)\la_{\be,1}t}}{2(n+1)\la_{\be,1}} e^{2Ct}.
\end{align}
We claim that
\begin{align}
T_*\geq \min\paren*{\frac{\log(1/2\ep)}{\la_{\be,1}}, \frac{1}{2(C+(n+1)\la_{\be,1})}\log(\frac{2(n+1)\la_{\be,1}}{C\ep^{2(n+1)}})}.
\end{align}
Otherwise, one checks that
\begin{align}
1=\|\mu^{T_*}-\mua^{T_*}\|_{L^2}^2 \leq C\frac{\ep^{2(n+1)}e^{2(n+1)\la_{\be,1}T_*}}{2(n+1)\la_{\be,1}} e^{2CT_*} < 1,
\end{align}
which is a contradiction. This then completes completes the proof of the lemma.

%For the maximal lifespan lower bound \eqref{eq:TmaxT*}, we combine the estimates established above with the local well-posedness theory of \cref{sec:WP}. 
\end{proof}

We now have all the ingredients to conclude the proof of \cref{prop:NLinstab}.

\begin{proof}[Conclusion of proof of \cref{prop:NLinstab}]
From \cref{lem:mumua}, we know that there is a constant $C_0>0$ depending on $\ds,\be$, such that there is a solution $\mu$ to equation \eqref{eq:lima} with initial datum $1+\ep\nu_1^0$ on the interval $[0,C_0^{-1}\log(1/\ep)]$ and
\begin{align}\label{eq:mumuadiffL1}
\forall t\in [0,C_0^{-1}\log(1/\ep)], \qquad \|\mu^t-\mua^t\|_{L^2} \leq C_0\ep^2 e^{C_0 t},
\end{align}
where we have constructed $\mua$ for $n=2$ (recall the definition \eqref{eq:muadef}). Similarly, using \cref{lem:nujbnd}, we have the bound
\begin{align}\label{eq:mualb}
1+2\ep e^{\la_{\be,1}t}\cos(2\pi k\cdot x) + \mathsf{C}_2 \ep^2 e^{2\la_{\be,1}t} \geq \mua^t \geq 1+2\ep e^{\la_{\be,1}t}\cos(2\pi k\cdot x) - \mathsf{C}_2 \ep^2 e^{2\la_{\be,1}t}.
\end{align}
The inequalities \eqref{eq:mumuadiffL1}, \eqref{eq:mualb} together with the triangle inequality now imply
\begin{align}\label{eq:mu1L1lb}
\|\mu^t-1\|_{L^1} &\geq  \|\mua^t-1\|_{L^1} - \|\mua^t-\mu^t\|_{L^2}\nn\\
&\geq \frac{4\ep}{\pi} e^{\la_{\be,1}t} - \ep^2 \paren*{C_0e^{C_0 t} + \mathsf{C}_2 e^{2\la_{\be,1}t}}.
\end{align}
Taking $C_0$ larger if necessary, we may assume without loss of generality that $C_0\geq \max(\mathsf{C}_2, 2\la_{\be,1})$. Solving
\begin{align}
\frac{e^{\la_{\be,1}t}}{2} = 2\ep C_0e^{C_0 t} \Longrightarrow t = \frac{1}{C_0-\la_{\be,1}}\log\paren*{\frac{1}{4C_0\ep}}.
\end{align}
Substituting this choice for $t$ into the right-hand side of \eqref{eq:mu1L1lb}, we see the right-hand side $=\frac12$, completing the proof. 
\end{proof}
%Now suppose $\ds=1$. Provided $\ep>0$ is sufficiently small depending on $\ds,\be$, we have instead of \eqref{eq:mumuadiffL1},
%\begin{align}\label{eq:mumuadiffL2}
%\forall t\in [0,C_0^{-1}\log(1/\ep)], \qquad \|\mu^t-\mua^t\|_{L^2} \leq C_0\ep^2 e^{C_0 t}
%\end{align}
%and instead of \eqref{eq:mu1L1lb},
%\begin{align}\label{eq:mu1L2lb}
%\|\mu^t-1\|_{L^1} &\geq \|\mua^t-1\|_{L^1} - \|\mua^t-\mu^t\|_{L^2} \nn\\
%&\geq \frac{4\ep}{\pi} e^{\la_{\be,1}t} - \ep^2 \paren*{C_0e^{C_0 t} + \mathsf{C}_2 e^{2\la_{\be,1}t}}.
%\end{align}
%The remainder of the proof follows just as before.

\subsection{Nonlinear stability}
\cref{lem:aL2exp} shows that $\muu$ is \emph{nonlinearly stable} if $\be<\bels$, implying, by the result of \cref{prop:NLinstab}, that the threshold $\bels$ is sharp for nonlinear stability. Moreover, this shows that it is false that all nonstationary classical solutions blow up in finite time if $\be>\bec$. Indeed, for lower dimensions $\ds$ (including $\ds=2$) such that $\bec<\bels$, we can choose $\be \in (\bec, \bels)$ and take an $O(\ep)$ perturbation of $\muu$ as initial datum to obtain from \cref{lem:aL2exp} a global nonstationary solution $\mu$ which relaxes exponentially fast to $\muu$ as $t\rightarrow\infty$.

\begin{remark}
Since the linearized operator $L_{\be}$ has decaying modes (i.e., modes with negative eigenvalues) if $\be<\infty$, this might suggest that it is possible to construct global solutions near $\muu$ for any $\be<\infty$. However, it is not clear to us how to modify the proof of \cref{lem:aL2exp} to show such a result: there could be a migration in the spectrum of the solution from high to low frequencies. {More precisely, if the Fourier transform $\hat{\mu}^t(k)$ vanishes for $|k|\leq r$ (i.e., $\hat{\mu}^t$ is supported on high frequencies), then instead of \eqref{eq:diss}, we have the sharper estimate
\begin{align}
2\cdd\int_{\T^\ds}(\mu^t-1)\Dm^{2-\ds}(\mu^t-1)dx -\frac{2}{\be}\int_{\T^\ds}|\nabla\mu^t|^2dx \leq 2\Big(\frac{\cdd}{r^{\ds}} - \frac{1}{\be}\Big)\|\nabla\mu^t\|_{L^2}^2.
\end{align}
For any $\be$, the prefactor made be negative by taking $r$ sufficiently large. However, even if $\hat{\mu}^0$ is supported on high frequencies to ensure negativity of the prefactor, we do not have a mechanism to propagate this condition.}
\end{remark}

\section{Uniqueness}\label{sec:Unq}
We now turn to the proof of \cref{thm:mainstab} concerning uniqueness of solutions to equation \eqref{eq:introKM}.

\subsection{Nonuniqueness of stationary solutions}\label{ssec:Instabssunq}
We first show that if $\be\in (\bels,\bec)$, then the uniform distribution is not the unique solution to \eqref{eq:introKM}, a fact which does not seem to have been previously observed for the general $\ds$-dimensional case. For $\ds=2$, nonuniqueness when $\be>\bec$ is classical \cite{ST1998, RT1998}. In particular, this shows that $\bec$ is perhaps not the right temperature to speak of ``criticality'' of the system. The idea of the proof is simple. As $\Ec_\beta(\muu) = 0$, so we show that there is a probability density $\mu_0$ with $-\infty<\Ec_\beta(\mu_0)<0$. As minimizers of the free energy are stationary solutions, this then implies nonuniqueness.  After the completion of our work, we learned that a similar argument has been previously used for general McKean-Vlasov equations on the torus in \cite{CP2010mv, CGPS2020}.

\begin{prop}\label{prop:ssnounq}
If $\beta>\bels$, then
\begin{align}\label{eq:infFEub}
\Ec_\beta^*\coloneqq \inf_{\mu \in \mathcal{P}_{ac}(\T^\ds)} \Ec_\beta(\mu)\leq -\frac{\be}{432}\left|\frac1\bels-\frac1\be\right|^3 < 0 = \Ec_\beta(\muu).
\end{align}
Moreover, if $\Ec_\beta$ has a minimizer, then there exist infinitely many solutions to the equation
\begin{align}\label{eq:TEeqn}
\mu = \dfrac{e^{\beta\g\ast\mu}}{Z_\beta}, \qquad \text{where} \ Z_\be = \int_{\T^{\ds}}e^{\beta\g\ast\mu}.
\end{align}
\end{prop}
\begin{proof}
For $k\in\Z^\ds\setminus\{0\}$ to be chosen and $\ep\leq\frac14$, set
\begin{equation}
\mu_\ep \coloneqq 1 + \ep\paren*{e^{2\pi ik\cdot x} + e^{-2\pi i k\cdot x} } = 1+2\ep\cos(2\pi k\cdot x).
\end{equation} 
We compute
\begin{multline}
\Ec_\beta(\mu_\ep) = \frac1\be\int_{\T^\ds}\log\paren*{1+2\ep\cos(2\pi k\cdot x)}\paren*{1+2\ep\cos(2\pi k\cdot x)} dx \\
- \frac12\int_{(\T^\ds)^2} \g(x-y)\paren*{1+2\ep\cos(2\pi k\cdot x)}\paren*{1+2\ep\cos(2\pi k\cdot y)}dxdy.
\end{multline}

Consider first the interaction term. Using that $\g$ is even with $\int_{\T^\ds}\g=0$, we find it equals
\begin{equation}
-\ep^2\int_{(\T^\ds)^2}\g(x-y)e^{2\pi ik\cdot(x-y)}dxdy - \ep^2\int_{(\T^\ds)^2}\g(x-y)e^{2\pi ik\cdot(x+y)}dxdy \eqqcolon \Te_1+\Te_2.
\end{equation}
Making the change of variable $z\coloneqq x-y$ and using Fubini-Tonelli,
\begin{align}\label{eq:FEminT1}
\Te_1 = -\ep^2\int_{\T^\ds} \int_{\T^\ds}\g(z)e^{2\pi i k\cdot z}dz dy = -\ep^2 \hat{\g}(-k) = -\frac{\ep^2\cdd}{(2\pi |k|)^\ds}.
\end{align}
Similarly,
\begin{align}\label{eq:FEminT2}
\Te_2 = -\ep^2\int_{\T^\ds}e^{4\pi ik\cdot y}\paren*{\int_{\T^\ds}\g(z) e^{2\pi i k\cdot z}dz}dy = -\ep^2\hat{\g}(-k)\int_{\T^\ds}e^{4\pi ik\cdot y}dy = 0.
\end{align}

Consider now the entropy term. Using the analyticity of $z\mapsto \log(1+z)$ for $|z|<1$, our assumption $\ep<\frac14$ allows us to write it as
\begin{multline}
\frac{2\ep}{\be}\int_{\T^\ds}\cos(2\pi k\cdot x)\paren*{1+2\ep\cos(2\pi k\cdot x)}dx - \frac{\ep^2}{\be}\int_{\T^\ds}\cos^2(2\pi k\cdot x)\paren*{1+2\ep\cos(2\pi k\cdot x)}dx \\
+\frac1\be\sum_{n=3}^\infty \frac{(-1)^{n+1}(2\ep)^{n}}{n} \int_{\T^\ds} \cos^n(2\pi k\cdot x)\paren*{1+2\ep\cos(2\pi k\cdot x)}dx.
\end{multline}
By the fundamental theorem of calculus and periodic boundary conditions, the first line above simplifies to
\begin{align}\label{eq:FEminT3}
\frac{2\ep^2}{\be} \int_{\T^\ds}\cos^2(2\pi k\cdot x)dx = \frac{\ep^2}{\be}.
\end{align}
This gives the leading contribution to the entropy term. For the subleading contribution, we crudely estimate
\begin{align}
\frac{1}{\be}\sum_{n=3}^\infty \left|\frac{(-1)^{n+1}(2\ep)^{n}}{n} \int_{\T^\ds} \cos^n(2\pi k\cdot x)\paren*{1+2\ep\cos(2\pi k\cdot x)}dx\right| \leq \frac{8\ep^3(1+2\ep)}{3(1-2\ep)\be} \leq \frac{8\ep^3}{\be}, \label{eq:FEminT4}
\end{align}
where we use our assumption that $\ep\leq\frac14$. Combining \eqref{eq:FEminT1}-\eqref{eq:FEminT4}, we obtain
\begin{equation}\label{eq:Fcmuepub}
\Ec_\beta(\mu_\ep) \leq \ep^2\paren*{\frac{(1+8\ep)}{\be} - \frac{\cdd}{(2\pi|k|)^\ds}}.
\end{equation}
We minimize the right-hand side by choosing $|k|=1$. Choosing $\ep>0$ sufficiently small so that $\ep < \frac{\be}{8}(\frac{\cdd}{(2\pi)^\ds} -\frac1\be)$, we see that the right-hand side of \eqref{eq:Fcmuepub} is $<0$, establishing assertion \eqref{eq:infFEub}. Viewing the right-hand side of \eqref{eq:Fcmuepub} as a function $\varphi(\ep)$, we see that
\begin{align}
\varphi'(\ep) = 2\ep\left(\frac{(1+8\ep)}{\be}-\frac{1}{\bels}\right)+\frac{8\ep^2}{\be}.
\end{align} 
Setting the left-hand side equal to zero, we find the unique critical point $\ep=\frac{\be}{12}\left(\frac1\bels-\frac1\be\right)$, which is a global minimum. We then compute
\begin{align}
\varphi\left(\frac{\be}{12}\left(\frac1\bels-\frac1\be\right)\right) = -\frac{\be}{432}\left|\frac1\bels-\frac1\be\right|^3.
\end{align}
It then follows $\inf_{\mu\in\P_{ac}}\Ec_\beta(\mu) \leq -\frac{\be}{432}\left|\frac1\bels-\frac1\be\right|^3$.

To show that there are infinitely many solutions to \eqref{eq:TEeqn}, we use the translation invariance of $\T^\ds$. This completes the proof. %Let $\mum$ be a minimizer of the free energy, which we know is a solution to \eqref{eq:TEeqn} and also must be smooth. Since $\mum\neq \muu$, there exist two distinct points $a,b\in\T^\ds$ such that $\mum(a) < \mum(b)$. since $\mum$ is continuous, there exists a sequence of (necessarily distinct) points $\{x_n\}_{n=1}^\infty \subset \T^\ds$ such that
%\begin{align}
%\mu(x_n) < \mu(x_{n+1}) \quad \text{and} \quad |x_n-x_{n+1}|> |x_{n+1}-x_{n+2}|.
%\end{align}
%Thus, if set $e_n \coloneqq x_{n+1}-x_n$, for $n\geq 1$, and $e_0 \coloneqq 0$, we see that the family of translates $\{\mum(\cdot+e_n)\}_{n=0}^\infty$ are pairwise distinct solutions.
\end{proof}

%\begin{remark}
%\cref{prop:ssnounq} shows that the unstable solutions $\mu_\ep$ given by \cref{prop:NLinstab} cannot converge weakly to $\muu$ as $t%\rightarrow\infty$. Indeed, $\Ec_\beta(\mu_\ep^0) < \Ec_\beta(\muu)$ and the free energy is nonincreasing along the flow.  {\cb [for small $\ep$?]}
%\end{remark}

{\cref{prop:ssnounq} tells us that if $\bels<\bec$, then for any $\be\in (\bels,\bec)$, the minimizer $\mu_\be$ of $\Ec_\be$ exists and is nonuniform. This gives us nonuniqueness of solutions to \eqref{eq:TEeqn} in this case, which corresponds to dimensions $\ds\ge 11$. If $\bec<\bels$, then for $\be>\bels$, there does not exist a minimizer for $\Ec_\be$, so we cannot show nonuniqueness through this route. However, using a mountain pass argument (cf. \cite[Section 3]{ST1998}), we can show that if $\bels>\bec$, then for $\be\in (\bec, \bels)$, there exists a nonuniform solution $\mu_\beta$ to equation \eqref{eq:TEeqn}.

\begin{prop}\label{prop:ssnounq2}
Let $\ds\geq 1$ and suppose that $\bels>\bec$. Then for $\be\in (\bec,\bels)$,  there exists a $\mu_\be\in \P_{ac}(\T^\ds)$ with $\Ec_\beta(\mu_\beta) \in (0,\infty)$ and $\mu_\beta$ is a nonuniform solution of equation \eqref{eq:TEeqn}. 
\end{prop}

To prove \cref{prop:ssnounq2}, rather than work with $\Ec_{\be}$, it is more convenient to work with its dual formulation
\begin{align}
\mathsf{E}_\be(\mathsf{h}) \coloneqq \frac{1}{2\cdd}\|\mathsf{h}\|_{\dot{H}^{\frac{\ds}{2}}}^2 - \be\log\Big[\int_{\T^\ds}e^{\hs}dx\Big]
\end{align}
defined on $\dot{H}^{\frac\ds2}$ identified as the subspace of $H^{\frac\ds2}$ with zero mean. By Jensen's inequality, $\int_{\T^\ds}e^{\hs}\geq e^{\int_{\T^\ds}\hs} = 1$, since $\hs$ has zero mean. Using the dual formulation of entropy and \cref{lem:logHLS}, we see that $\Es_\be$ is bounded from below if and only if $\be\leq \bec$. A critical point of $\Es_\be$ is a solution to the equation
\begin{align}\label{eq:hseqn1}
\frac{1}{\cdd}\Dm^{\ds}\hs =\be \left(\frac{e^{ \hs}}{\int_{\T^\ds}e^{\hs}} - 1\right),
\end{align}
which yields a solution $\mu = \frac{1}{\cdd\be}\Dm^{\ds}\hs$ to equation \eqref{eq:TEeqn}. Note that under this duality, $\hs=0$ if and only if $\mu=\muu$.  {An argument similar to the proof of \cref{lem:FEsc} shows that for $\be>\bec$, $\Es_\be$ is unbounded below.}
%, there exists a family $\{\hs_\lambda\}_{\la>0}$ in $\dot{H}^{\ds/2}$ such that
%\begin{align}
%\Es_\be(\hs_\la) = (\bec-\be)\log(1/\la) + O(1).
%\end{align}
%{\cb [Isn't it enough to say that $\Es_\be$ is unbounded below or is the point to prove it here?]}

We deduce \cref{prop:ssnounq2} from the following result for the functional $\Es_\be$.

\begin{prop}\label{lem:Esbe}
Let $\ds\geq 1$ and suppose that $\bels>\bec$. Then for $\be\in (\bec,\bels)$, there exists a critical point $\hs_\be \in \dot{H}^{\frac\ds2}(\T^d)$ of $\Es_\be$, such that $\Es_\be(\hs_\be)>0$.
\end{prop}

To prove \cref{lem:Esbe}, we make use of the fact that if $\be<\bels$, then $\hs=0$ is a strict local minimizer of $\Es_\be$. This is the content of the next lemma.

\begin{lemma}\label{lem:locmin}
There exists a constant $C>0$, depending only on $\ds$, such that for any direction $\phi \in \dot{H}^{\frac{\ds}{2}}$, we have
\begin{align}\label{eq:d2dt2Eslb}
\frac{d^2}{dt^2}\Es_\be(t\phi) \geq (2\pi)^{-\ds}\Big(\bels-\be - tC\be \Big(e^{2t\|\phi\|_{\dot{H}^{\frac12}} + 2t^2\|\phi\|_{\dot{H}^{\frac12}}^2}\indic_{\ds=1} + e^{\frac{1}{C}\|\phi\|_{\dot{H}^{\frac\ds2}}^\ds}\indic_{\ds\geq 2}\Big)^{1/2}\Big)\|\phi\|_{\dot{H}^{\frac\ds2}}^2 .
\end{align}
In particular, if $\be<\bels$, then $\hs=0$ is a strict local minimizer of $\Es_\be$. 
\end{lemma}
\begin{proof}
Given any direction $\phi \in \dot{H}^{\frac{\ds}{2}}$, we compute
\begin{align}
\frac{d^2}{dt^2}\Es_\be(t\phi) &= \frac{1}{\cdd}\|\phi\|_{\dot{H}^{\frac\ds2}}^2 - \be\frac{\int_{\T^\ds}(\phi)^2 e^{t\phi}dx}{\int_{\T^\ds}e^{t\phi}dx} + 2\be\frac{\left(\int_{\T^\ds}\phi e^{t\phi}dx\right)^2}{\left(\int_{\T^\ds} e^{t\phi}dx\right)^2} \nn\\
&\geq \frac{1}{\cdd}\|\phi\|_{\dot{H}^{\frac\ds2}}^2 - \be\frac{\int_{\T^\ds}(\phi)^2 e^{t\phi}dx}{\int_{\T^\ds}e^{t\phi}dx}. \label{eq:EsbeSVlb}
\end{align}
By mean value theorem for the numerator and Jensen's inequality for the denominator,
\begin{align}
\frac{\int_{\T^\ds}(\phi)^2 e^{t\phi}dx}{\int_{\T^\ds}e^{t\phi}dx} \leq \|\phi\|_{L^2}^2 + t\int_{\T^\ds}|\phi|^3 e^{t|\phi|}dx 
&\leq\|\phi\|_{L^2}^2 + t\|\phi\|_{L^6}^2\Big(\int_{\T^\ds}e^{2t|\phi|}dx\Big)^{1/2}, \label{eq:e2tphipre}
\end{align}
where the final inequality is by Cauchy-Schwarz. If $\ds=1$, then we use Onofri's inequality for $\mathbb{T}$ (e.g., see \cite[Theorem 5]{CL1992}) to bound
\begin{align}\label{eq:e2tphi1}
\int_{\T^\ds}e^{2t|\phi|}dx \leq e^{2t\int_{\T^\ds}|\phi|dx} e^{2t^2\|\phi\|_{\dot{H}^{\frac12}}^2} \leq e^{2t\|\phi\|_{\dot{H}^{\frac12}} + 2t^2\|\phi\|_{\dot{H}^{\frac12}}^2}.
\end{align}
If $\ds\geq 2$, then we use the Trudinger-Moser inequality \cite{Trudinger1967,Moser1970} and the embedding $\dot{H}^{\frac{\ds-2}{2}} \subset L^{\ds}$ to estimate
\begin{align}\label{eq:e2tphi2}
\int_{\T^\ds}e^{2t|\phi|}dx  \leq \int_{\T^\ds}e^{\frac{(2\|\nabla\phi\|_{L^{\ds}}/c)^\ds}{\ds} + \frac{\ds-1}{\ds}(c|\phi|/\|\nabla\phi\|_{L^{\ds}})^{\frac{\ds}{\ds-1}}}dx \lesssim e^{c'\|\phi\|_{\dot{H}^{\frac\ds2}}^\ds},
\end{align}
for $c>0$ chosen sufficiently small. Inserting the estimates \eqref{eq:e2tphi1}, \eqref{eq:e2tphi2} into the right-hand side of \eqref{eq:e2tphipre}, we find that
\begin{align}
\frac{\int_{\T^\ds}(\phi)^2 e^{t\phi}dx}{\int_{\T^\ds}e^{t\phi}dx} &\leq \|\phi\|_{L^2}^2  +  t\|\phi\|_{L^6}^2\Big(e^{2t\|\phi\|_{\dot{H}^{\frac12}} + 2t^2\|\phi\|_{\dot{H}^{\frac12}}^2}\indic_{\ds=1} + Ce^{c'\|\phi\|_{\dot{H}^{\frac\ds2}}^\ds}\indic_{\ds\geq 2}\Big)^{1/2} \nn\\
&\leq \|\phi\|_{L^2}^2  +  tC'\|\phi\|_{\dot{H}^{\frac\ds2}}^2\Big(e^{2t\|\phi\|_{\dot{H}^{\frac12}} + 2t^2\|\phi\|_{\dot{H}^{\frac12}}^2}\indic_{\ds=1} + e^{c'\|\phi\|_{\dot{H}^{\frac\ds2}}^\ds}\indic_{\ds\geq 2}\Big)^{1/2},\label{eq:e2phi3}
\end{align}
where the final line is by the Sobolev embedding, $\|\phi\|_{L^6} \lesssim \|\phi\|_{\dot{H}^{\frac\ds3}} \leq (2\pi)^{-\frac{\ds}{6}} \|\phi\|_{\dot{H}^{\frac\ds2}}$. Inserting the estimate \eqref{eq:e2phi3} into \eqref{eq:EsbeSVlb} and using $\|\phi\|_{L^2} \leq (2\pi)^{-\frac\ds2} \|\phi\|_{\dot{H}^{\frac\ds2}}$, we arrive at
\begin{align}
\frac{d^2}{dt^2}\Es_\be(t\phi) \geq \Big(\frac1\cdd - \frac{\be}{(2\pi)^{\ds}} - tC'\be \Big(e^{2t\|\phi\|_{\dot{H}^{\frac12}} + 2t^2\|\phi\|_{\dot{H}^{\frac12}}^2}\indic_{\ds=1} + e^{c'\|\phi\|_{\dot{H}^{\frac\ds2}}^\ds}\indic_{\ds\geq 2}\Big)^{1/2}\Big)\|\phi\|_{\dot{H}^{\frac\ds2}}^2 .
\end{align}
Since $\be<\bels$ and therefore $\frac{1}{\cdd} - \frac{\be}{(2\pi)^{\ds}} = \frac{\bels-\be}{(2\pi)^{\ds}} >0$, we may take $t$ sufficiently small depending only on $\|\phi\|_{\dot{H}^{\frac\ds2}}$ and $\be$, so that the expression inside the parentheses is $>0$. This completes the proof of the lemma.
\end{proof}

Fix $\be_0\in (\bec,\bels)$. {Since $\Es_\be$ is unbounded below, there exists $\hs_0 \in\dot{H}^{\frac\ds2}$} such that $\Es_{\be_0}(\hs_0) <0$ and $\|\hs_0\|_{\dot{H}^{\frac\ds2}}\geq 1$. Evidently, for any $\be\geq \be_0$, we have
\begin{align}
\Es_{\be}(\hs_0) \leq \Es_{\be_0}(\hs_0) < 0.
\end{align}
Let
\begin{align}
P \coloneqq \{\zeta \in C([0,1], \dot{H}^{\frac\ds2}) : \zeta(0) =0, \ \zeta(1) = \hs_0\},
\end{align}
and define the function
\begin{align}
\forall \be\geq\be_0, \qquad \Fs_{\be} \coloneqq \inf_{\zeta \in P} \max_{t\in [0,1]} \Es_{\be}(\zeta(t)).
\end{align}
%The function $\Fs_\be$ is monotone decreasing, hence differentiable for a.e. $\be\in (\be_0,\bels)$.
From our assumption that $\|\hs_0\|_{\dot{H}^{\frac\ds2}}\geq 1$ and the intermediate value theorem, it follows that for any $\delta\in (0,1)$ and path $\zeta \in P$, there exists $t_\delta \in (0,1)$ such that $\|\zeta(t_\delta)\|_{\dot{H}^{\frac\ds2}}=\delta$. Combining this observation with \eqref{eq:d2dt2Eslb}, we see that there exists $c_0>0$ such that
\begin{align}
\forall \be\in [\be_0,\bels), \qquad \Fs_\be \geq c_0\left(\be_s-\be\right).
\end{align}

To prove \cref{lem:Esbe}, we want to apply the mountain pass theorem \cite[Section 8.5, Theorem 2]{Evans2010book}, which would yield that $\Fs_\be$ is a critical value for $\Es_\be$. That is, there is a critical point $\hs \in \dot{H}^{\frac\ds2}$ of $\Es_\be$ such that $\Es_\be(\hs) = \Fs_\be>0$. To do so, we only need to check that $\Es_\be$ satisfies the Palais-Smale compactness condition, which is a consequence of the next lemma.

\begin{lemma}\label{lem:PS}
Let $\{\hs_n\}$ be a sequence in $\dot{H}^{\frac\ds2}$ such that $\|\hs_n\|_{\dot{H}^{\frac\ds2}}^2\leq C_1$ for all $n$, for some constant $C_1>0$ and $D\Es_\be(\hs_n)\rightarrow 0$ as $n\rightarrow\infty$. Then $\{h_n\}$ has a convergent subsequence in $\dot{H}^{\frac\ds2}$.
\end{lemma}
\begin{proof}
Let $\{\hs_n\}$ be a sequence as in the statement of the lemma. By weak compactness, passing to a subsequence if necessary, there exists an $\hs\in \dot{H}^{\frac\ds2}$ such that $\hs_n\xrightharpoonup[]{} \hs$ in $\dot{H}^{\frac\ds2}$. By Rellich-Kondrachov, $\hs_n\rightarrow \hs$ in $\dot{H}^{\frac{\ds}{2}-\epsilon}$ as $n\rightarrow\infty$, for any $\epsilon>0$. Then
\begin{align}
\ipp*{D\Es_\be(\hs_n), \hs_n-\hs}  &=  \frac{1}{\cdd}\int_{\T^\ds}\Dm^{\frac{\ds}{2}}(\hs_n)\Dm^{\frac{\ds}{2}}(\hs_n-\hs) - \be\frac{\int_{\T^\ds}(\hs_n-\hs)e^{\hs_n}dx}{\int_{\T^\ds}e^{\hs_n}dx} \nn\\
&=\frac{1}{\cdd}\|\hs_n-\hs\|_{\dot{H}^{\frac\ds2}}^2  + \frac1\cdd\int_{\T^\ds}\Dm^{\frac{\ds}{2}}\hs\Dm^{\frac{\ds}{2}}(\hs_n-\hs)dx  - \be\frac{\int_{\T^\ds}(\hs_n-\hs)e^{\hs_n}dx}{\int_{\T^\ds}e^{\hs_n}dx}.
\end{align}
Since $\hs_n\xrightharpoonup[]{} \hs$ in $\dot{H}^{\frac\ds2}$, we have
\begin{align}
\lim_{n\rightarrow\infty} \left|\int_{\T^\ds}\Dm^{\frac{\ds}{2}}\hs\Dm^{\frac{\ds}{2}}(\hs_n-\hs)dx\right| = 0.
\end{align}
By Cauchy-Schwarz and Jensen,
\begin{align}
\frac{\int_{\T^\ds}\left|\hs_n-\hs\right|e^{\hs_n}dx}{\int_{\T^\ds}e^{\hs_n}dx} \leq \|\hs_n-\hs\|_{L^2} \frac{\Big(\int_{\T^\ds}e^{2\hs_n}dx\Big)^{1/2}}{\int_{\T^\ds}e^{\hs_n}dx} \leq \|\hs_n-\hs\|_{L^2}\Big(\int_{\T^\ds}e^{2\hs_n}dx\Big)^{1/2}.
\end{align}
The second factor on the right-hand side is uniformly bounded in $n$ by Onofri/Trudinger-Moser, similarly to the proof of \cref{lem:locmin}, and so
\begin{align}
\lim_{n\rightarrow\infty} \be\frac{\int_{\T^\ds}(\hs_n-\hs)e^{\hs_n}dx}{\int_{\T^\ds}e^{\hs_n}dx} = 0.
\end{align}
Since
\begin{align}
\lim_{n\rightarrow\infty}\left|\ipp*{D\Es_\be(\hs_n), \hs_n-\hs}\right| \leq 2C_1\lim_{n\rightarrow\infty}\|D\Es_\be(\hs_n)\|_{\dot{H}^{\frac\ds2}} =0,
\end{align}
we conclude, after a little bookkeeping, that $\hs_n\rightarrow \hs$ in $\dot{H}^{\frac\ds2}$.
\end{proof}
}

\subsection{Uniqueness of stationary solutions}
We now show uniqueness of solutions to \eqref{eq:TEeqn} for high temperatures, i.e. $\muu$ is the unique solution. Equivalently, the uniqueness threshold $\beu>0$. The idea is essentially to show that the map $\mu \mapsto \frac{e^{\beta\g\ast\mu}}{Z_\beta}$ is a contraction (cf. \cite[Section 5]{ST1998}).

\begin{prop}\label{prop:ssunq}
Let $\ds\geq 1$. There exists {$\be_0>0$}, depending only on $\ds$, such that if $\be<\be_0$ and  $\mu$ is a solution to equation \eqref{eq:TEeqn}, then $\mu=\muu$. 
\end{prop}
\begin{proof}
Let $\mu \in \P_{ac}(\T^\ds)$ be a solution to equation \eqref{eq:TEeqn}. Set $\mathsf{h}\coloneqq \g\ast\mu$, so that $\hs$ solves
\begin{align}\label{eq:hseqn}
\frac{1}{\cdd}\Dm^{\ds}\hs = \frac{e^{\be \hs}}{\int_{\T^\ds}e^{\be \hs}} - 1.
\end{align}
Also, note that by $\int_{\T^\ds}\hs=0$ and Jensen's inequality, $\int_{\T^\ds}e^{\be\hs} \geq 1$. Observe that 
\begin{align}
\mathsf{h}(x) =\frac{1}{\cdd}\left(\g\ast\Dm^{\ds}\hs\right)(x) &= \frac{1}{\cdd}\int_{\T^\ds}\g(x-y)\Dm^{\ds}\hs(y) dy \nn\\
&= \int_{\T^\ds}\g(x-y)\frac{e^{\be \hs(y)}}{\int_{\T^\ds}e^{\be \hs}} \nn\\
&\leq-\frac{1}{\int_{\T^\ds}e^{\be \hs}} \int_{[-\frac12,\frac12]^{\ds}}\log|x-y|e^{\be \hs(y)}dy + \underbrace{\sup_{z\in [-\frac12,\frac12]^{\ds}} \left|\g(z)+\log|z|\right|}_{C_{\ds}'}, \label{eq:hunq1}
\end{align}
where the final line follows from adding and subtracting $-\log x$ and using \eqref{eq:ggE}. Using the elementary inequality $ab \leq b(\log b-1) + e^{a}$, for $a>0$ and $b\in\R$, we see that for any $\ep \in (0,\ds)$ (note that $\ds=\frac{\bec}{2}$),
\begin{align}
-\int_{[-\frac12,\frac12]^{\ds}}\log|x-y|e^{\be \hs(y)}dy \leq \int_{[-\frac12,\frac12]^{\ds}} |x-y|^{-\ep}dy + \int_{[-\frac12,\frac12]^{\ds}}\frac{e^{\be \hs(y)}}{\ep}\left(\log\Big(\frac{e^{\be \hs(y)}}{\ep}\Big) - 1\right)dy. \label{eq:hunq2'}
\end{align}
Evidently,
\begin{align}
\int_{[-\frac12,\frac12]^{\ds}} |x-y|^{-\ep}dy \leq \frac{C_{\ds}}{\ds-\ep}, \label{eq:hunq2}
\end{align}
for some $C_{\ds}>0$. Next, observe that
\begin{align}
\int_{\T^{\ds}}\frac{e^{\be \hs}}{\ep}\left(\log\Big(\frac{e^{\be \hs}}{\ep}\Big) - 1\right)dy = \frac1\ep\int_{\T^\ds}e^{\be\hs}\be\hs dy - \frac{\log\ep}{\ep}\int_{\T^\ds}e^{\be\hs }dy  - \frac1\ep\int_{\T^\ds}e^{\be\hs}dy. \label{eq:hunq3}
\end{align}
Since, using \eqref{eq:hseqn}, we have
\begin{align}
\frac1\ep\int_{\T^\ds}e^{\be\hs}\be\hs dy = \frac{\be\int_{\T^\ds}e^{\be\hs}}{\cdd\ep} \int_{\T^\ds}\Dm^\ds\hs \hs = \frac{\be\int_{\T^\ds}e^{\be\hs}}{\cdd\ep}\|\hs\|_{\dot{H}^{\frac{\ds}{2}}}^2,
\end{align}
combining \eqref{eq:hunq1}, \eqref{eq:hunq2}, \eqref{eq:hunq3}, we obtain
\begin{align}\label{eq:hunqLinf}
\|\hs\|_{L^\infty} \leq \frac{C_{\ds}}{\ds-\ep} + \frac{\be}{\cdd\ep}\|\hs\|_{\dot{H}^{\frac{\ds}{2}}}^2  -\frac{(\log\ep + 1)}{\ep} + {\sup_{z\in [-\frac12,\frac12]^{\ds}} \left|\g(z)+\log|z|\right|}.
\end{align}
On the other hand,
\begin{align}
\int_{\T^\ds}e^{\be\hs}\be\hs dy \leq \be \|\hs\|_{L^\infty} \int_{\T^\ds}e^{\be\hs}dy,
\end{align}
which implies that
\begin{align}
\frac{1}{\cdd}\|\hs\|_{\dot{H}^{\frac{\ds}{2}}}^2 \leq  \|\hs\|_{L^\infty}.
\end{align}
Inserting this bound into \eqref{eq:hunqLinf} and rearranging yields
\begin{align}
\|\hs\|_{L^\infty} \leq \left(1-\frac{\be}{\ep}\right)^{-1}\Big( \frac{C_{\ds}}{\ds-\ep} -\frac{(\log\ep + 1)}{\ep}  + C_{\ds}'\Big),\\
\|\hs\|_{\dot{H}^{\frac{\ds}{2}}}^2 \leq \cdd\left(1-\frac{\be}{\ep}\right)^{-1}\Big( \frac{C_{\ds}}{\ds-\ep}-\frac{(\log\ep + 1)}{\ep} +C_{\ds}'\Big),
\end{align}
provided $\frac{\be}{\ep}<1$. If $\be<\ds$, we may satisfy both constraints on $\ep$ by choosing any $\ep \in (\be,\ds)$, such that $\ep\geq e^{-1}$. The choice $\ep=\frac{\be+\ds}{2}$ satisfies these conditions, leading to the estimate
\begin{align}
\|\hs\|_{L^\infty} \leq \frac{2C_{\ds}(\ds+\be)}{(\ds-\be)^2} + \frac{C_{\ds}'(\ds+\be)}{(\ds-\be)}, \label{eq:hLinffin}\\
\|\hs\|_{\dot{H}^{\frac{\ds}{2}}}^2 \leq   \frac{2\cdd C_{\ds}(\ds+\be)}{(\ds-\be)^2} +  \frac{\cdd C_{\ds}'(\ds+\be)}{(\ds-\be)}. \label{eq:hSobfin}
\end{align}

Since $\int_{\T^\ds}\hs=0$, we may write
\begin{align}
\|\hs\|_{\dot{H}^{\frac\ds2}}^2 =  \cdd\frac{\int_{\T^\ds}e^{\be\hs}\hs dy}{\int_{\T^\ds}e^{\be\hs}dy}=\cdd\frac{\int_{\T^\ds}[e^{\be\hs}-1]\hs dy}{\int_{\T^\ds}e^{\be\hs}dy}. \label{eq:hSobcon1}
\end{align}
By the mean value theorem, $|e^{\be\hs}-1| \leq \be|\hs|e^{\be\|\hs\|_{L^\infty}}$. Combining this bound with \eqref{eq:hLinffin} and inserting into the right-hand side of \eqref{eq:hSobcon1} yields
\begin{align}
\|\hs\|_{\dot{H}^{\frac\ds2}}^2  &\leq \frac{\cdd \be e^{\be\|\hs\|_{L^\infty}} \int_{\T^\ds}|\hs|^2 dy}{\int_{\T^\ds}e^{\be \hs}dy} \nn\\
&\leq {\frac{\cdd \be}{(2\pi)^{\ds/2}}\exp(\frac{2C_{\ds}(\ds+\be)}{(\ds-\be)^2} + \frac{C_{\ds}'(\ds+\be)}{(\ds-\be)})} \|\hs\|_{\dot{H}^{\frac\ds2}}^2,
\end{align}
where we have used that $\|\hs\|_{\dot{H}^{\frac\ds2}}\geq (2\pi)^{\frac{\ds}{2}}\|\hs\|_{L^2}$ by Plancherel and $\int_{\T^\ds}e^{\be\hs}\geq 1$, as commented above. Let {$\be_0$} be such that
\begin{align}\label{eq:uniqtcon}
 \frac{\cdd {\be_0} e^{\frac{2C_{\ds}(\ds+{\be_0})}{(\ds-{\be_0})^2} + \frac{C_{\ds}'(\ds+{\be_0})}{(\ds-{\be_0})}}}{(2\pi)^{\ds/2}} = 1.
\end{align}
Then for $\be<{\be_0}$, we conclude that $\|\hs\|_{\dot{H}^{\frac\ds2}}=0$. This implies that $\mu=\muu$, which completes the proof.
\end{proof}

%\begin{remark}
%Unfortunately, we are unable to give a more explicit condition for $\be$ than \eqref{eq:uniqtcon}. The reason is that an explicit formula for the quantity $C_{\ds}'$ defined in \eqref{eq:hunqLinf} does not seem to be known.
%\end{remark}

\section{A modulated log HLS inequality}\label{sec:MLHLS}
In this section, we prove \cref{thm:mainMLHLS}. We divide the section into two subsections: \cref{ssec:MLHLSdprf} shows a counterexample to the modulated log HLS inequality at low temperature, while \cref{ssec:MLHLSprf} proves modulated log HLS inequality at high temperature.

\subsection{Disproof at low temperature}\label{ssec:MLHLSdprf}
We start with the negative portion of \cref{thm:mainMLHLS}. 

%We know that any probability density minimizer $\mu_{\min}$ of the free energy $\Ec_\beta$ must satisfy $\Ec_\beta(\mu_{\min}) < 0$ and be a stationary solution. for $\theta>\frac{1}{2d}$, probability density $f_N$ on $(\T^\ds)^N$, and probability density $\mu$ on $(\T^\ds)$,
%\begin{equation}\label{eq:MlHLSdis}
%\frac{1}{2\theta}\E_{f_N}\paren*{\Fr_N(\ux_N,\mu)} \leq H_N(f_N\vert \mu^{\otimes N}) + CN^{-\al},
%\end{equation}
%where $\al>0$ depends on $d,\theta$ and $C>0$ depends on $d,\theta, \inf_{\T^\ds}\mu$, and $\|\mu\|_{W^{k,\infty}}$, for some $k$. 
\begin{proof}
{Suppose $\be>\min(\bels,\bec)$. Then by \cref{lem:FEsc} or the upper bound for the minimal free energy $\Ec_\beta^*$ in \eqref{eq:infFEub}, there exists an $\eta_\be>0$ and a smooth probability density $\mum$, such that 
\begin{align}\label{eq:mumeta}
\frac1\be\int_{\T^\ds}\log\paren*{\mum}d\mum -\frac12\int_{(\T^\ds)^2}\g(x-y)d\mum^{\otimes 2}(x,y) \leq -\eta_\be.
\end{align}}
Set $f_N\coloneqq \mum^{\otimes N}$. Since $\int_{\T^\ds}\g=0$, we have
\begin{align}\label{eq:MEunif}
\Fr_N(\ux_N,\muu) = \frac{1}{2N^2}\sum_{1\leq i\neq j\leq N}\g(x_i-x_j).
\end{align}
Now since $f_N$ is an exchangeable law, the linearity of expectation implies
\begin{align}\label{eq:MEunif'}
\E_{f_N}\paren*{\Fr_N(\ux_N,\muu)} = \frac{(N-1)}{2N}\int_{(\T^\ds)^2}\g(x-y)d\mum^{\otimes 2}(x,y).
\end{align}
Now observe
\begin{align}
H_N(f_N\vert \muu^{\otimes N}) = \frac1N\int_{(\T^\ds)^N}\log\paren*{\mum^{\otimes N}}d\mum^{\otimes N} = \int_{\T^\ds}\log\paren*{\mum}d\mum.
\end{align}
Hence, by \eqref{eq:mumeta}, 
\begin{align}
\frac1\be H_N(f_N\vert \muu^{\otimes N}) = \frac1\be\int_{\T^\ds}\log\paren*{\mum}d\mum &\leq \frac12\int_{(\T^\ds)^2}\g(x-y)d\mum^{\otimes 2}(x,y) - \eta_\be \nn\\
&=\frac{N}{(N-1)}\E_{f_N}\paren*{\Fr_N(\ux_N,\muu)} - \eta_\be \nn\\
&= \E_{f_N}\paren*{\Fr_N(\ux_N,\muu)} + \frac{\|\mum\|_{\dot{H}^{-\frac{\ds}{2}}}^2}{2N} - \eta_\be, \label{eq:MElHLSdisID}
\end{align}
where the last line follows from \eqref{eq:MEunif'}. Choosing $N=N(\be)$ sufficiently large so that $\frac{\|\mum\|_{\dot{H}^{-\frac{d}{2}}}^2}{2N}\leq \frac{\eta_\be}{2}$, \eqref{eq:MElHLSdisID} implies
\begin{align}
\frac1\beta H_N(f_N\vert \muu^{\otimes N})  \leq  \E_{f_N}\paren*{\Fr_N(\ux_N,\muu)}  - \frac{\eta_\be}{2}.
\end{align}
\end{proof}

\begin{comment}
If \eqref{eq:MlHLSdis} did hold, then multiplying both sides by $\frac{N}{N-1}$, letting $N\rightarrow\infty$, and using \eqref{eq:MElHLSdisID}, we would obtain
\begin{align}
\frac12\int_{(\T^\ds)^2}\g(x-y)d\mum^{\otimes 2}(x,y) &= \lim_{N\rightarrow\infty} \frac{N}{2\theta (N-1)}\E_{f_N}\paren*{\Fr_N(\ux_N,\muu)} \nn\\
&\leq \lim_{N\rightarrow\infty}\paren*{H_N(f_N \vert \mum^{\otimes N}) + \frac{CN^{1-\al}}{(N-1)}} \nn\\
&=\int_{\T^\ds}\log\paren*{\mum}d\mum,
\end{align}
which is a contradiction.
\end{comment}

\subsection{Proof at high temperature}\label{ssec:MLHLSprf}
We turn to the proof of the positive portion of \cref{thm:mainMLHLS}, which proceeds through several lemmas. {We outline the main proof first and then will give the proofs of the individual lemmas in the ensuing subsections.}

As our starting point, we apply the Donsker-Varadhan lemma, for $\be>0$, to obtain
\begin{equation}\label{eq:gr1F}
\E_{f_N}\left[\Fr_{N}(\ux_N,\mu)\right]  \leq \frac1\be\paren*{H_N(f_N\vert \mu^{\otimes N}) + \frac{\log K_{N,\be}(\mu)}{N}},
\end{equation}
where the reader will recall from \eqref{defKNbe} the modulated partition function notation $K_{N,\be}(\mu)$. We want to regularize the potential $\g$ by replacing it with a {truncated version}. Given $\eta>0$, define the truncated potential
\begin{align}\label{eq:getadef}
\g_{(\eta)} \coloneqq \g+ \log\frac{|x|}{\max(|x|,\eta)}.
\end{align}
The function $\g_{(\eta)}=\g$ for $|x|\geq \eta$ and is equal to a $C^\infty$ perturbation of $-\log\eta$ for $|x|<\eta$. In particular, $\g_{(0)}=\g$. Using that $\supp\left(\log\frac{|x|}{\max(|x|,\eta)}\right)\subset \ol{B}(0,\eta)$ with $\|\log\frac{|x|}{\max(|x|,\eta)}\|_{L^1} \lesssim \eta^\ds$, it follows that
\begin{equation}
\Fr_{N}({\ux_N},\mu) \leq \Fr_{N,\eta}({\ux_N},\mu) + \frac{1}{2N^2}\sum_{\substack{1\leq i\neq j\leq N \\ |x_i-x_j|\leq \eta }}\log\paren*{\frac{\eta}{|x_i-x_j|}} + C\eta^\ds\|\mu\|_{L^\infty},
\end{equation}
for some constant $C>0$ depending only on $\ds$, where $\Fr_{N,\eta}$ is the truncated modulated energy defined by
\begin{equation}
\Fr_{N,\eta}({\ux_N},\mu) \coloneqq \int_{(\T^\ds)^2\setminus\triangle} \g_{(\eta)}(x-y)d\paren*{\frac1N\sum_{i=1}^N\d_{x_i}-\mu}^{\otimes 2}(x,y).
\end{equation}
Thus, we have shown that
\begin{equation}\label{eq:ZNretabetapre}
K_{N,\be}(\mu) \leq e^{CN\be\|\mu\|_{L^\infty}\eta^\ds}\mathbb{E}_{\mu^{\otimes N}}\left[\exp\Bigg(N\be\Fr_{N,\eta}({\ux_N},\mu) + \frac{\be}{2N}\sum_{\substack{1\leq i\neq j\leq N \\ |x_i-x_j|\leq \eta }}\log\frac{\eta}{|x_i-x_j|}\Bigg)\right].
\end{equation}
It will be convenient to introduce the following partition function notation for the truncated modulated energy:
\begin{align}\label{eq:KNgatrun}
K_{N,\be,\eta}(\mu) \coloneqq \E_{\mu^{\otimes N}}\left[e^{N\be\Fr_{N,\eta}({\ux_N},\mu)}\right].
\end{align}
{Our first lemma shows that we can estimate the second factor in \eqref{eq:ZNretabetapre} in terms of $K_{N,\be',\eta}(\mu)$, for larger inverse temperature $\be'<\bec$, up to an error factor whose log will be negligible as $N\rightarrow\infty$.}

\begin{lemma}\label{lem:expfacbnd}
Let $0<4\eta\leq \ep\leq\frac{1}{4}$, $0<\be<\be'<\bec$. Suppose that $\mu\in \P(\T^\ds)$ and $\log\mu\in W^{1,\infty}(\T^\ds)$. Then there exist constants $C_0,C_1,C_2>0$ depending only on $\ds$, such that
\begin{multline}\label{eq:expfacbnd}
\mathbb{E}_{\mu^{\otimes N}}\left[\exp\Bigg(N\be\Fr_{N,\eta}({\ux_N},\mu) + \frac{\be}{2N}\sum_{\substack{1\leq i\neq j\leq N \\ |x_i-x_j|\leq \eta }}\log\frac{\eta}{|x_i-x_j|}\Bigg)\right]\\
\leq K_{N,\be',\eta}(\mu)^{\frac{\be}{\be'}} e^{C_0N\eta\be\paren*{\|\nabla\mu\|_{L^\infty} + \frac{1}{\ep} }} \\
\times \paren*{1+C_24^{\frac{\be\be'}{\be'-\be}} \|\mu\|_{L^\infty}\ep^\ds +  C_2\|\mu\|_{L^\infty}\eta^\ds\left(\frac{4^\be\|\mu\|_{L^\infty}}{(\inf\mu)C_1(\bec-\be)}\right)^{\frac{\be'}{\be'-\be}}}^{(N-1)\left(1-\frac{\be}{\be'}\right)}.
\end{multline}
\end{lemma}

The truncated potential $\g_{(\eta)}$ introduced in \eqref{eq:getadef} is continuous. Hence, $\g_{(\eta)}\ast\nu$ is a continuous function for any $\nu\in\P(\T^\ds)$. So we are justified in introducing the notation
\begin{equation}
\bar{F}_{\eta}(\nu,\mu)\coloneqq \frac12\int_{(\T^\ds)^2}\g_{(\eta)}(x-y)d\paren*{\nu-\mu}^{\otimes 2}(x,y),
\end{equation}
with which we then obtain (using $\g_{(\eta)}(0)= |\log\eta|$)
\begin{equation}
\Fr_{N,\eta}(\ux_N,\mu) = \bar{F}_{\eta}\left(\frac1N\sum_{i=1}^N\d_{x_i},\mu\right) - \frac{|\log\eta|}{2N}.
\end{equation}

Given $\ux_N$, abbreviate $\mu_{\ux_N} \coloneqq \frac1N\sum_{i=1}^N \d_{x_i}$. We want to be able to speak of the relative entropy between $\mu_{\ux_N}$ and $\mu$, but of course this does not make sense since $\mu_{\ux_N}$ is not absolutely continuous with respect to $\mu$. Accordingly, we coarse grain $\mu_{\ux_N}$ to obtain an absolutely continuous measure. Given an integer $M\in\N$, define the \emph{coarse-graining operator at scale $M^{-1/\ds}$},
\begin{equation}\label{eq:CMdef}
\Cc_M: L^1(\T^\ds)\rightarrow L^\infty(\T^\ds), \qquad \Cc_M(f) \coloneqq \sum_{k=1}^M f_{Q_k}\indic_{Q_k},
\end{equation}
where $\{Q_k\}$ is a partition of $\T^\ds$ into $M$ pairwise disjoint half-open cubes of side length $M^{-1/\ds}$ (i.e., translates of $[0, M^{-1/\ds})^\ds$) and $f_{Q_k} \coloneqq \dashint_{Q_k}df$ is the average of $f$ over $Q_k$. If $f\in \dot{C}^{0,\la}(\T^\ds)$, for some $0<\la\leq 1$, then if $x\in Q_k$, 
\begin{equation}
|f(x) - \Cc_M(f)(x)| = |f(x)-f_{Q_k}| = \left|\dashint_{Q_k}\paren*{f(x)-f(y)}dy\right| \lesssim M^{-\la /\ds}\|f\|_{\dot{C}^{0,\la}}.
\end{equation}

Since $\|\g_{(\eta)}\|_{\dot{C}^{0,1}} \leq \frac{1}{\eta}$, we have that
\begin{equation}
\forall \nu\in\P(\T^\ds), \qquad \left|\g_{(\eta)}\ast\nu - \g_{(\eta)}\ast\Cc_M(\nu)\right| \lesssim \frac{M^{-1/\ds}}{\eta}.
\end{equation}
Recalling from \eqref{eq:dcompundo} the definition of the truncated partition function $K_{N,\be,\eta}(\mu)$, the preceding implies
\begin{align}\label{eq:logZNetapre}
\frac{\log K_{N,\be,\eta}(\mu)}{N} \leq \frac{CM^{-1/\ds}}{\eta} + \frac1N\log\E_{\mu^{\otimes N}}\Bigg[e^{N\be \bar{F}_{\eta}(\Cc_M(\mu_{\ux_N}),\mu)}\Bigg].
\end{align}
Introduce the \emph{rate functional}
\begin{equation}\label{eq:Ibeetadef}
I_{\be,\eta}(\mu) \coloneqq \sup_{\nu\in\P(\T^\ds)} \paren*{\be\bar{F}_{\eta}(\nu,\mu) - H(\nu\vert \mu)},
\end{equation}
where $H(\cdot\vert\cdot)$ is the relative entropy for measures on $\T^\ds$. Equivalently, $I_{\be,\eta}(\mu)$ is the opposite of of the infimum of the single-particle modulated free energy with background $\mu$.  Obviously, $I_{\be,\eta}(\mu)=-\infty$ unless the Radon-Nikodym derivative $\frac{d\nu}{d\mu}\in L\log L(\mu)$. Furthermore, since $\be\bar{F}_\eta(\mu,\mu) - H(\mu\vert\mu)=0$, we have that $I_{\be,\eta}(\mu)\geq 0$. Since $\Cc_M(\mu_{\ux_N})\in\P(\T^\ds)$, we have by definition of supremum that
\begin{align}\label{eq:Ivanapp}
&\frac1N\log\E_{\mu^{\otimes N}}\Bigg[e^{N\be \bar{F}_{\eta}(\Cc_M(\mu_{\ux_N}),\mu)}\Bigg] \leq I_{\be,\eta}(\mu) + \frac1N\log\E_{\mu^{\otimes N}}\Bigg[e^{N H(\Cc_M(\mu_{\ux_N})\vert\mu)}\Bigg].
\end{align}
We want to show that the first term on the right-hand side vanishes, which we can show if $\be$ is sufficiently small depending on $\ds$. 

%The first term in the right-hand side \eqref{eq:Ivanapp} vanishes provided that the localization parameter $r>0$ is chosen so as to make $\|V_{r,\eta}\|_{L^1}$ small enough. This is the content of the following lemma adapted from \cite[Lemma 3.2]{BJW2020}.

\begin{lemma}\label{lem:Ivan}
There exists $\beta_0>$, depending on $\ds$, such that for any $\be<\be_0$, there exists $\delta_{\be},\eta_{\be}>0$, depending on $\ds,\be$, such that if $\|\log\mu\|_{L^\infty}\leq\delta_{\be}$ and $0\leq\eta\leq\eta_\be$, then $I_{\be,\eta}(\mu)=0$.
\end{lemma}

{\begin{remark}
If $\mu=\muu$, then since
\begin{align}
F_\eta(\nu,\muu) \leq \int_{(\T^\ds)^2}\g(x-y)d\nu^{\otimes 2}(x,y) + O(\eta^\ds),
\end{align}
it follows that
\begin{align}
I_{\be,\eta}(\muu) \leq \beta O(\eta^\ds) + \sup_{\nu\in\P(\T^\ds)} \left(\be\bar{F}(\nu,\muu) - H(\nu\vert \muu)\right) = \beta O(\eta^\ds) -\be\Ec_\be^*.
\end{align}
If $\ds=2$, then by the work \cite{LL2006}, it is known that $\Ec_\be^* = 0$ for $\be\leq \bec=4$. Although this does not show that $I_{\be,\eta}(\muu)=0$, except for $\eta=0$, the reader will see below that the $O(\eta^\ds)$ bound is sufficient to prove the mLHLS inequality in this case.
\end{remark}}

Returning to \eqref{eq:Ivanapp}, we see that then
\begin{equation}
\frac1N\E_{\mu^{\otimes N}}\Bigg[e^{N\be \bar{F}_{\eta}(\Cc_M(\mu_{\ux_N}),\mu)}\Bigg]  \leq \frac1N\log\E_{\mu^{\otimes N}}\Bigg[e^{N H(\Cc_M(\mu_{\ux_N})\vert\mu)}\Bigg].
\end{equation}

Next, by definition \eqref{eq:CMdef} of $\Cc_M$, we have
\begin{align}
H(\Cc_M(\mu_{\ux_N})\vert\mu) &= \sum_{Q_k} \frac{M^{-1}|\{x_i \in Q_k\}|}{N}\int_{Q_k}\log\paren*{\frac{\Cc_M(\mu_{\ux_N})}{\mu}}dx\nn\\
&= \sum_{Q_k}\frac{|\{x_i \in Q_k\}|}{N}\paren*{\log\paren*{\frac{M^{-1}|\{x_i \in Q_k\}|}{N}}-\dashint_{Q_k}\log\mu dx} \nn\\
&=\int_{\T^\ds}\log\paren*{\Cc_M(\mu_{\ux_N)}} d\mu_{\ux_N} - \sum_{Q_k}\frac{|\{x_i \in Q_k\}|}{N}\dashint_{Q_k}\log\mu dx.
\end{align}
Since $\log\mu$ is Lipschitz by assumption, we have by mean-value theorem
\begin{align}
\forall x_i \in Q_k, \qquad \left|\dashint_{Q_k}\log\mu dx - \log\mu(x_i)\right| \lesssim \|\nabla\log\mu\|_{L^\infty} M^{-1/\ds},
\end{align}
which implies that
\begin{equation}
\left|\sum_{Q_k}\frac{|\{x_i \in Q_k\}|}{N}\dashint_{Q_k}\log\mu dx - \int_{\T^\ds}\log\mu d\mu_{\ux_N}\right| \lesssim \|\nabla\log\mu\|_{L^\infty}  M^{-1/\ds}.
\end{equation}
Therefore,
\begin{equation}
\frac1N\E_{\mu^{\otimes N}}\Bigg[e^{N H(\Cc_M(\mu_{\ux_N})\vert\mu)}\Bigg] \leq {CM^{-\frac{1}{\ds}}\|\nabla\log\mu\|_{L^\infty}} + \frac1N\log \E_{\mu^{\otimes N}}\Bigg[e^{N\int_{\T^\ds}\log\left(\frac{\Cc_M(\mu_{\ux_N})}{\mu}\right)d\mu_{\ux_N} }\Bigg]. \label{eq:LDcgapp}
\end{equation}

The advantage of coarse graining is the we have essentially reduced to computing large deviations with respect to {the} uniform measure. The following lemma is a reformulation of \cite[Proposition 3.1]{BJW2020}. The proof uses classical large deviation arguments (cf. \cite{BG1999}) and is essentially a combinatorial exercise.

\begin{lemma}\label{lem:LDcg}
There exists a constant $C>0$ depending only on $\ds$, such that for any $\mu\in \P_{ac}(\T^\ds)$ and $M\leq \frac{N}{2}$, it holds that
\begin{equation}
C^{-M}\frac{N^{\frac{M-1}{2}}}{M^{M-\frac{1}{2}}}\leq \E_{\mu^{\otimes N}}\Bigg[e^{N\int_{\T^\ds}\log\paren*{\frac{\Cc_M(\mu_{\ux_N})}{\mu} }d\mu_{\ux_N}}\Bigg]\leq C N^{M+\frac12}.
\end{equation}
\end{lemma}

Applying \cref{lem:LDcg} to the right-hand side of \eqref{eq:LDcgapp} and recalling our starting identity \eqref{eq:Ivanapp}, we find
\begin{equation}
\frac1N\log \E_{\mu^{\otimes N}}\Bigg[e^{N\be \bar{F}_{\eta}(\Cc_M(\mu_{\ux_N}),\mu)}\Bigg] \leq \frac{C\|\nabla\log\mu\|_{L^\infty}}{M^{1/\ds}} + \frac{C(M+\frac{1}{2})\log N}{N},
\end{equation}
for a constant $C>0$ depending only on $\ds$. Applying this bound to the right-hand side of \eqref{eq:logZNetapre}, we obtain
\begin{equation}\label{eq:logZNrbeta'bnd}
\frac{\log K_{N,\be,\eta}(\mu)}{N} \leq \frac{CM^{-1/\ds}}{\eta}  + \frac{C\|\nabla\log\mu\|_{L^\infty}}{M^{1/\ds}} + \frac{C(M+\frac{1}{2})\log N}{N}.
\end{equation}

The conclusion of the proof of \cref{thm:mainMLHLS} is now just a matter of bookkeeping and optimization of parameters, so that all terms involving $\eta,\ep,M,N$ are $O(N^{-\ga})$ (see \cref{sssec:mLHLSconc} below for the details).

\subsubsection{Proof of \cref{lem:expfacbnd}}\label{sssec:mLHLSlem1}
Here, we prove  \cref{lem:expfacbnd}.

By H\"older's inequality in the $x_i$ variable,
\begin{align}
&\mathbb{E}_{\mu^{\otimes N}}\left[\exp\Bigg(\be\sum_{i=1}^N\Bigg(\Fr_{N,\eta}({\ux_N},\mu) + \frac{1}{2N}\sum_{\substack{1\leq  j\neq i\leq N \\ |x_i-x_j|\leq \eta }}\log\frac{\eta}{|x_i-x_j|}\Bigg)\Bigg)\right]\nn\\
&\leq \prod_{i=1}^N  \mathbb{E}_{\mu^{\otimes N}}^{\frac{1}{N}}\left[\exp\Bigg(\be N\Fr_{N,\eta}({\ux_N},\mu) + \frac{\be}{2}\sum_{\substack{1\leq  j\neq i\leq N \\ |x_i-x_j|\leq \eta }}\log\frac{\eta}{|x_i-x_j|}\Bigg)\right]\nn\\
&= \mathbb{E}_{\mu^{\otimes N}}\left[\exp\Bigg(\be N\Fr_{N,\eta}({\ux_N},\mu) + \frac{\be}{2}\sum_{\substack{1<j\leq N \\ |x_1-x_j|\leq \eta }}\log\frac{\eta}{|x_1-x_j|}\Bigg)\right], \label{eq:HolNapp}
\end{align}
where the final line follows from symmetry with respect to exchange of particle labels.

For each $2\leq j\leq N$, write $1 = \indic_{|x_1-x_j|\leq \eta} + \indic_{|x_1-x_j|> \eta} $. Inserting the product of these factors over the range $2\leq j \leq N$ into the expectation, expanding, and using the particle label symmetry plus linearity of expectation,
%Let
%\begin{equation}
%n \coloneqq |\{1\leq j\leq N : j\neq 1 \ \text{and} \ |x_1-x_j|\leq \eta\}|.
%\end{equation}
%In other words, $n$ is the number of points distinct from $x_1$ but which are within $\eta$ distance of $x_1$. Again using the particle label symmetry and linearity of expectation, 
we have that \eqref{eq:HolNapp} equals
\begin{equation}\label{eq:premu1n}
\sum_{n=0}^{N-1} {N-1\choose n}  \mathbb{E}_{\mu^{\otimes N}}\Bigg[ e^{N\be\Fr_{N,\eta}({\ux_N},\mu)} e^{\frac{\be}{2}\sum_{j=2}^{n+1}\log\paren*{\frac{\eta}{|x_1-x_j|}}}\prod_{j=2}^{n+1}\indic_{|x_1-x_j|\leq \eta}\prod_{j=n+2}^{N}\indic_{|x_1-x_j|>\eta}\Bigg].
\end{equation}
We would like to use Fubini-Tonelli to reduce to estimating an expectation of the form
\begin{equation}
\E_{\mu^{\otimes n}}\Bigg[e^{\frac{\be}{2}\sum_{j=2}^{n+1}\log\paren*{\frac{\eta}{|x_1-x_j|}}}\prod_{j=2}^{n+1}\indic_{|x_1-x_j|\leq \eta}\Bigg]  = \E_{\mu^{\otimes n}}\Bigg[\prod_{j=2}^{n+1} \paren*{\frac{\eta}{|x_1-x_j|}}^{\frac{\be}{2}}\indic_{|x_1-x_j|\leq \eta}\Bigg],
\end{equation}
but the issue is that $\Fr_{N,\eta}({\ux_N},\mu)$ in \eqref{eq:premu1n} depends on $x_2,\ldots,x_{n+1}$. However, since $x_1$ and $x_j$ are close, for $2\leq j\leq n+1$, we may hope to replace any such $x_j$ by $x_1$ up to negligible error.

We make this precise as follows. Let $\ux_N^{(1,n)}\coloneqq (x_1^{(1,n)},\ldots,x_N^{(1,n)})\in (\T^\ds)^N$ be defined by
\begin{equation}
x_{j}^{(1,n)} \coloneqq  \begin{cases} x_1, & {1\leq j\leq n+1} \\ x_j, & {n+2\leq j\leq N}. \end{cases}
\end{equation}
The transformation $\ux_N \mapsto \ux_{N}^{(1,n)}$ is evidently measurable. We also introduce an additional scale $\ep\geq 4\eta$. For $2\leq j\leq n+1$ and any $1\leq i\leq N$, observe from \eqref{eq:getadef} that
\begin{equation}
\left|\g_{(\eta)}(x_j-x_i) - \g_{(\eta)}(x_1-x_i)\right|= \left|\log\frac{\max(|x_1-x_i|,\eta)}{\max(|x_j-x_i|,\eta)} \right|.
\end{equation}
Evidently, if $\max\paren*{|x_1-x_i|,|x_j-x_i|} \leq \eta$, then
\begin{equation}
|\g_{(\eta)}(x_j-x_i) - \g_{(\eta)}(x_1-x_i)| = 0.
\end{equation}
Also, without loss of generality, if $|x_1-x_i| > \eta \geq |x_j-x_i|$, then
\begin{equation}
|x_1-x_i| \leq |x_1-x_j| + |x_j-x_i| \leq 2\eta,
\end{equation}
hence
\begin{equation}\label{eq:gdiffji1i}
|\g_{(\eta)}(x_j-x_i) -\g_{(\eta)}(x_1-x_i)| \leq \log 2.
\end{equation}
Finally, if $\min\paren*{|x_1-x_i|,|x_j-x_i|} > \eta$, then since $|x_1-x_j|\leq \eta$ (by assumption $2\leq j\leq n+1$), 
\begin{equation}
|x_1-x_i| \leq |x_1-x_j| + |x_j-x_i| \leq \eta + |x_j-x_i| < 2|x_j-x_i|.
\end{equation}
By the same reasoning, $|x_j-x_i| <2|x_1-x_i|$. So, \eqref{eq:gdiffji1i} also holds. In all cases we have,
\begin{align}\label{eq:getadiff}
|\g_{(\eta)}(x_j-x_i) - \g_{(\eta)}(x_1-x_i)| \leq \log 2.
\end{align}
Suppose $|x_1-x_i| \geq \ep\geq 4|x_1-x_j|$. Then since $\ep\geq 4\eta$, it follows that $|x_j-x_i| \geq \frac{3\ep}{4}\geq 3\eta$. So, by the mean-value theorem,
\begin{align}\label{eq:getaMVT}
|\g_{(\eta)}(x_j-x_i) - \g_{(\eta)}(x_1-x_i)| &= \log\left(\frac{|x_j-x_i|}{|x_1-x_i|}\right)\leq  \frac{2\eta}{\ep}.
\end{align}

Inserting $\prod_{j=n+2}^{N}\left(\indic_{\eta<|x_1-x_j|\leq \ep} + \indic_{|x_1-x_j|>\ep}\right)$ into \eqref{eq:premu1n} and expanding the product \eqref{eq:FFmu1n}, it follows that \eqref{eq:premu1n} is $\leq$
\begin{multline}\label{eq:preinsert}
\sum_{n=0}^{N-1}  \sum_{m=0}^{N-1-n} {N-1\choose n} {N-n-1\choose m} \mathbb{E}_{\mu^{\otimes N}}\Bigg[ e^{N\be\Fr_{N,\eta}({\ux_N},\mu)} e^{\frac{\be}{2}\sum_{j=2}^{n+1}\log\paren*{\frac{\eta}{|x_1-x_j|}}}\\
\prod_{j=2}^{n+1}\indic_{|x_1-x_j|\leq \eta}\prod_{j=n+2}^{n+m+1}\indic_{\eta<|x_1-x_j|\leq \ep}  \prod_{j=n+m+2}^N \indic_{\ep< |x_1-x_j|}\Bigg].
\end{multline}
Using \eqref{eq:getadiff}, \eqref{eq:getaMVT}, we see that
\begin{equation}\label{eq:FFmu1n}
\Fr_{N,\eta}(\ux_N,\mu) \leq \Fr_{N,\eta}(\ux_N^{(1,n)},\mu) +  \frac{Cn\eta\|\nabla\mu\|_{L^\infty}}{N} + \frac{C\eta}{\ep} + \frac{(n^2+nm)\log 2}{N^2},
\end{equation}
and by reversing our steps, also that
\begin{equation}\label{eq:Fmu1nF}
\Fr_{N,\eta}(\ux_N^{(1,n)},\mu) \leq \Fr_{N,\eta}(\ux_N,\mu) +  \frac{Cn\eta\|\nabla\mu\|_{L^\infty}}{N} +\frac{C\eta}{\ep} + \frac{(n^2+nm)\log 2}{N^2}.
\end{equation}
Inserting the bound \eqref{eq:FFmu1n} into \eqref{eq:preinsert}, we see that \eqref{eq:preinsert} is $\leq$
\begin{multline}\label{eq:FNretax1n}
\sum_{n=0}^{N-1} \sum_{m=0}^{N-1-n} {N-1\choose n} {N-n-1\choose m} e^{CN\eta\be\paren*{\|\nabla\mu\|_{L^\infty} + \frac{1}{\ep} }} 2^{\be(n+m)}  \\
\E_{\mu^{\otimes N}}\Bigg[e^{N\be\Fr_{N,\eta}(\ux_N^{(1,n)},\mu)} \prod_{j=2}^{n+1} \paren*{\frac{\eta}{|x_1-x_j|}}^{\frac{\be}{2}}\indic_{|x_1-x_j|\leq \eta} \prod_{j=n+2}^{n+m+1}\indic_{\eta<|x_1-x_j|\leq \ep}  \prod_{j=n+m+2}^N \indic_{\ep< |x_1-x_j|}\Bigg].
\end{multline}
%Let $m$ denote the number of points $x_i$ whose distance to $x_1$ satisfies $\eta< |x_1-x_i|\leq \ep$:
%\begin{equation}
%m \coloneqq |\left\{n+2\leq i\leq N : \eta< |x_1-x_i| \leq \ep \right\} |,
%\end{equation}
%where we know that $i$ cannot be $\leq n+1$ since otherwise that would imply $|x_1-x_i|\leq \eta$. Without loss of generality by %particle label symmetry, suppose that $\eta< |x_1-x_i|\leq \ep$ for $i=n+2,\ldots,n+m+1$.

We use Fubini-Tonelli to write the expectation as an iterated integral:
\begin{multline}
\int_{\T^\ds}d\mu(x_1)\int_{(\T^\ds)^{N-n-m-1}}d\mu^{\otimes N-n-m-1}(\ux_{n+2+m;N}) \prod_{j=n+m+2}^N \indic_{\ep< |x_1-x_j|}\\
\int_{(\T^\ds)^{m}}d\mu^{\otimes m}(\ux_{n+2;n+1+m})e^{N\be\Fr_{N,\eta}(\ux_N^{(1,n)},\mu)}\prod_{j=n+2}^{n+m+1}\indic_{\eta<|x_1-x_j|\leq \ep}  \\
\int_{(\T^\ds)^{n}}d\mu^{\otimes n}(\ux_{2;n+1})\prod_{j=2}^{n+1} \paren*{\frac{\eta}{|x_1-x_j|}}^{\frac{\be}{2}}\indic_{|x_1-x_j|\leq \eta},
\end{multline}
where {we use the notation $X_{l;k} \coloneqq (x_{l},x_{l+1},\ldots,x_{k})$, for $l\leq k$}. Provided $\be<\bec {= 2\ds}$, we have by direct computation,
\begin{align}
&\int_{(\T^\ds)^{n}}d\mu^{\otimes n}(\ux_{2;n+1})\prod_{j=2}^{n+1} \paren*{\frac{\eta}{|x_1-x_j|}}^{\frac{\be}{2}}\indic_{|x_1-x_j|\leq \eta} \nn\\
&\leq \left(\frac{\|\mu\|_{L^\infty}}{C_1(\bec-\be)}\eta^{\ds}\right)^n\nn\\
&=\left(\frac{\|\mu\|_{L^\infty}}{C_1(\bec-\be)}\right)^n\int_{(\T^\ds)^n}\prod_{j=2}^{n+1}\indic_{|x_1-x_j|\leq \eta} d\ux_N\nn\\
&=\left(\frac{\|\mu\|_{L^\infty}}{(\inf\mu)C_1(\bec-\be)}\right)^n \int_{(\T^\ds)^n}\prod_{j=2}^{n+1}\indic_{|x_1-x_j|\leq \eta} d\mu^{\otimes n}(\ux_{2;n+1}), \label{eq:C1cons}
\end{align}
where the ultimate line follows from using $1\leq \frac{\mu}{\inf\mu}$. Using the inequality \eqref{eq:Fmu1nF} to replace $\Fr_{N,\eta}(\ux_N^{(1,n)},\mu)$  by $\Fr_{N,\eta}(\ux_N,\mu)$, we see that \eqref{eq:FNretax1n} is $\leq$
\begin{multline}\label{eq:recallbkp}
\sum_{n=0}^{N-1} \sum_{m=0}^{N-1-n} {N-1\choose n} {N-n-1\choose m} e^{2CN\eta\be\paren*{\|\nabla\mu\|_{L^\infty} + \frac{1}{\ep} }} 2^{2\be(n+m)} \left(\frac{\|\mu\|_{L^\infty}}{(\inf\mu)C_1(\bec-\be)}\right)^n\\
\E_{\mu^{\otimes N}}\Bigg[e^{N\be \Fr_{N,\eta}(\ux_N,\mu)}\prod_{j=2}^{n+1}\indic_{|x_1-x_j|\leq \eta}\prod_{j=n+m+2}^N \indic_{|x_1-x_j|>\ep}\prod_{j=n+2}^{n+m+1}\indic_{\eta<|x_1-x_j|\leq \ep}\Bigg].
\end{multline}

Let $\bec>\be'>\be$. Applying H\"older's inequality with exponents $\frac{\be'}{\be}$ and $\frac{\be'}{\be'-\be}$, which are obviously conjugate to $1$, we find that
\begin{align}
&\E_{\mu^{\otimes N}}\Bigg[e^{N\be \Fr_{N,\eta}(\ux_N,\mu)}\prod_{j=2}^{n+1}\indic_{|x_1-x_j|\leq \eta}\prod_{j=n+m+2}^N \indic_{|x_1-x_j|>\ep}\prod_{j=n+2}^{n+m+1}\indic_{\eta<|x_1-x_j|\leq \ep}\Bigg]\nn\\
&\leq \E_{\mu^{\otimes N}}^{\frac{\be}{\be'}}\Bigg[e^{N\be' \Fr_{N,\eta}(\ux_N,\mu)}\prod_{j=2}^{n+1}\indic_{|x_1-x_j|\leq \eta}\prod_{j=n+m+2}^N \indic_{|x_1-x_j|>\ep}\prod_{j=n+2}^{n+m+1}\indic_{\eta<|x_1-x_j|\leq \ep}\Bigg] \nn\\
&\ph\qquad \times\E_{\mu^{\otimes N}}^{1-\frac{\be}{\be'}}\Bigg[\prod_{j=2}^{n+1}\indic_{|x_1-x_j|\leq \eta}\prod_{j=n+m+2}^N \indic_{|x_1-x_j|>\ep}\prod_{j=n+2}^{n+m+1}\indic_{\eta<|x_1-x_j|\leq \ep}\Bigg] \nn\\
&\leq \left(C_2^{n+m}\|\mu\|_{L^\infty}^{n+m}  \eta^{\ds n}\ep^{\ds m}\right)^{1-\frac{\be}{\be'}}\E_{\mu^{\otimes N}}^{\frac{\be}{\be'}}\Bigg[e^{N\be'\Fr_{N,\eta}(\ux_N,\mu)}\prod_{j=2}^{n+1}\indic_{|x_1-x_j|\leq \eta} \nn\\
&\ph\qquad \prod_{j=n+m+2}^N\indic_{|x_1-x_j|>\ep} \prod_{j=n+2}^{n+m+1}\indic_{\eta<|x_1-x_j|\leq \ep}\Bigg].
\end{align}
To obtain the final inequality, we have implicitly bounded $\mu(x_j) \leq \|\mu\|_{L^\infty}$ for $j\leq n+m+1$ and used that $\int_{|x_1-x_j|>\ep}d\mu(x_j) \leq 1$, since $\mu$ is a probability density, for $n+m+2\leq j\leq N$. Recalling the constant $C_1$ from \eqref{eq:C1cons}, we use H\"older's inequality again with respect to the summations over $n,m$ to obtain
\begin{multline}
\sum_{n=0}^{N-1} \sum_{m=0}^{N-1-n} {N-1\choose n} {N-n-1\choose m} 4^{\be(n+m)}\left(\frac{\|\mu\|_{L^\infty}}{(\inf\mu)C_1(\bec-\be)}\right)^n  \paren*{C_2^{n+m}\|\mu\|_{L^\infty}^{n+m} \eta^{\ds n}\ep^{\ds m}}^{1-\frac{\be}{\be'}}\\
 \E_{\mu^{\otimes N}}^{\frac{\be}{\be'}}\Bigg[e^{N\be' \Fr_{N,\eta}(\ux_N,\mu)}\prod_{j=2}^{n+1}\indic_{|x_1-x_j|\leq \eta}\prod_{j=n+m+2}^N \indic_{|x_1-x_j|>\ep}\prod_{j=n+2}^{n+m+1}\indic_{\eta<|x_1-x_j|\leq \ep}\Bigg] \\
\leq \Bigg[\sum_{n=0}^{N-1} \sum_{m=0}^{N-1-n} {N-1\choose n} {N-n-1\choose m} \paren*{4^{\be(n+m)}\left(\frac{\|\mu\|_{L^\infty}}{(\inf\mu)C_1(\bec-\be)}\right)^n}^{\frac{\be'}{\be'-\be}}  \\
 \paren*{C_2^{n+m}\|\mu\|_{L^\infty}^{n+m} \eta^{\ds n}\ep^{\ds m}}\Bigg]^{1-\frac{\be}{\be'}} \Bigg[\sum_{n=0}^{N-1} \sum_{m=0}^{N-1-n} {N-1\choose n} {N-n-1\choose m}\E_{\mu^{\otimes N}}\Bigg[e^{N\be'\Fr_{N,\eta}(\ux_N,\mu)}\prod_{j=2}^{n+1}\indic_{|x_1-x_j|\leq \eta} \\
 \prod_{j=n+m+2}^N\indic_{|x_1-x_j|>\ep} \prod_{j=n+2}^{n+m+1}\indic_{\eta<|x_1-x_j|\leq \ep}\Bigg]\Bigg]^{\frac{\be}{\be'}}. \label{eq:posHoldcomp}
\end{multline}

Using the binomial formula $(a+b)^M = \sum_{\ell=0}^M {M\choose \ell}a^{\ell}b^{M-\ell}$, we find that
\begin{align}
&\sum_{n=0}^{N-1} \sum_{m=0}^{N-1-n} {N-1\choose n} {N-n-1\choose m} \paren*{4^{\be(n+m)}\left(\frac{\|\mu\|_{L^\infty}}{(\inf\mu)C_1(\bec-\be)}\right)^n}^{\frac{\be'}{\be'-\be}}   \paren*{C_2^{n+m}\|\mu\|_{L^\infty}^{n+m} \eta^{\ds n}\ep^{\ds m}} \nn\\
&=\sum_{n=0}^{N-1} \sum_{m=0}^{N-1-n} {N-1\choose n} {N-n-1\choose m}\left( C_2\|\mu\|_{L^\infty}\eta^\ds\left(\frac{4^\be\|\mu\|_{L^\infty}}{(\inf\mu)C_1(\bec-\be)}\right)^{\frac{\be'}{\be'-\be}} \right)^n \left(C_24^{\frac{\be\be'}{\be'-\be}} \|\mu\|_{L^\infty}\ep^\ds\right)^m \nn\\
&= \sum_{n=0}^{N-1} {N-1\choose n}\paren*{1+C_24^{\frac{\be\be'}{\be'-\be}} \|\mu\|_{L^\infty}\ep^\ds}^{N-n-1}\left( C_2\|\mu\|_{L^\infty}\eta^\ds\left(\frac{4^\be\|\mu\|_{L^\infty}}{(\inf\mu)C_1(\bec-\be)}\right)^{\frac{\be'}{\be'-\be}} \right)^n \nn\\
&= \paren*{1+C_24^{\frac{\be\be'}{\be'-\be}} \|\mu\|_{L^\infty}\ep^\ds +  C_2\|\mu\|_{L^\infty}\eta^\ds\left(\frac{4^\be\|\mu\|_{L^\infty}}{(\inf\mu)C_1(\bec-\be)}\right)^{\frac{\be'}{\be'-\be}}}^{N-1}. \label{eq:Holdcomp}
\end{align}

Reversing the decompositions $1=\indic_{|x_1-x_j|\leq \eta}+\indic_{|x_1-x_j|>\eta}$ and $\indic_{|x_1-x_j|>\eta} = \indic_{|x_1-x_j|>\ep} + \indic_{\eta<|x_1-x_j|\leq \ep}$,
\begin{multline}\label{eq:dcompundo}
\sum_{n=0}^{N-1} \sum_{m=0}^{N-1-n} {N-1\choose n} {N-n-1\choose m}\E_{\mu^{\otimes N}}\Bigg[e^{N\be'\Fr_{N,\eta}(\ux_N,\mu)}\prod_{j=2}^{n+1}\indic_{|x_1-x_j|\leq \eta} \prod_{j=n+m+2}^N\indic_{|x_1-x_j|>\ep} \\
\prod_{j=n+2}^{n+m+1}\indic_{\eta<|x_1-x_j|\leq \ep}\Bigg]  =\E_{\mu^{\otimes N}}\Bigg[e^{N\be'\Fr_{N,\eta}(\ux_N,\mu)}\Bigg]= K_{N,\be',\eta}(\mu).
\end{multline}

Applying \eqref{eq:Holdcomp}, \eqref{eq:dcompundo} to the right-hand side of \eqref{eq:posHoldcomp} yields
\begin{equation}
K_{N,\be',\eta}(\mu)^{\frac{\be}{\be'}} \paren*{1+C_24^{\frac{\be\be'}{\be'-\be}} \|\mu\|_{L^\infty}\ep^\ds +  C_2\|\mu\|_{L^\infty}\eta^\ds\left(\frac{4^\be\|\mu\|_{L^\infty}}{(\inf\mu)C_1(\bec-\be)}\right)^{\frac{\be'}{\be'-\be}}}^{(N-1)\left(1-\frac{\be}{\be'}\right)}.
\end{equation}
Recalling \eqref{eq:recallbkp} and performing a little bookkeeping completes the proof of \cref{lem:expfacbnd}.

\subsubsection{Proof of \cref{lem:Ivan}}\label{sssec:mLHLSlem2}
Here, we prove \cref{lem:Ivan}.

Let $\nu\in\P(\T^\ds)$. Then unpacking the definition of the left-hand side,
\begin{align}\label{eq:FEmaxLHS}
\be\bar{F}_{\eta}(\nu,\mu) - H(\nu\vert\mu) &= \paren*{\frac{\be}{2}\int_{(\T^\ds)^2}\g(x-y) d\nu^{\otimes 2}(x,y) -\frac{\be}{\bec}\int_{\T^\ds}\log(\nu)d\nu}  \nn\\
&\ph +\frac{\be}{2}\int_{(\T^\ds)^2}\paren*{\g_{(\eta)}(x-y)-\g(x-y)}d\nu^{\otimes 2}(x,y) \nn\\
&\ph-\be\int_{(\T^\ds)^2}\g_{(\eta)}(x-y)d\nu(x)d\mu(y) + \frac{\be}{2}\int_{(\T^\ds)^2}\g_{(\eta)}(x-y)d\mu^{\otimes 2}(x,y) \nn\\
&\ph + \frac{\be}{\bec}\int_{\T^\ds}\log(\mu)d\nu- \paren*{1-\frac{\be}{\bec}}\int_{\T^\ds}\log\paren*{\frac{d\nu}{d\mu}} d\nu.
\end{align}
From \cref{lem:logHLS} and \eqref{eq:ggE}, we know that
\begin{align}\label{eq:CLHLSdef}
\sup_{\rho\in\P_{ac}} \paren*{\frac{1}{2}\int_{(\T^\ds)^2}\g(x-y) d\rho^{\otimes 2}(x,y) -\frac{1}{\bec}\int_{\T^\ds}\log(\rho)d\rho}\eqqcolon C_{LHLS} \in [0,\infty).
\end{align}
From \eqref{eq:getadef}, we know that for $\ep\in (0,\ds)$,
\begin{align}
\int_{(\T^\ds)^2}\left|\g_{(\eta)}(x-y)-\g(x-y)\right|d\nu^{\otimes 2}(x,y) &\leq -\int_{(\T^\ds)^2}\log\left(\frac{|x-y|}{\max(|x-y|,\eta)}\right)d\nu^{\otimes 2}(x,y) \nn\\
&\leq \int_{|x|\leq \eta} \left(\frac{|x|}{\eta}\right)^{-\ep} + \frac1\ep\int_{\T^\ds}\log(\nu)d\nu \nn\\
&\leq C_\ep\eta^{\ds}+ \frac1\ep\int_{\T^\ds}\log(\frac{d\nu}{d\mu})d\nu + \frac1\ep\|\log\mu\|_{L^\infty}, \label{eq:getagdiffFennu}
\end{align}
where we have used Fenchel's inequality. Since we also have by \eqref{eq:getadef} that
\begin{align}\label{eq:getamuLinf}
\|\g_{(\eta)}\ast\mu\|_{L^\infty} \leq \eta^{\ds}\|\mu\|_{L^\infty}\int_{|x|\leq 1}|\log|x||+ \|\g\ast\mu\|_{L^\infty},
\end{align}
it follows that
\begin{multline}
\be\bar{F}_{\eta}(\nu,\mu) - H(\nu\vert\mu) \leq \be C_{LHLS} + \frac{C_\ep \be \eta^{\ds}}{2} + \frac{3\be}{2}\Big(C\eta^{\ds}\|\mu\|_{L^\infty}+ \|\g\ast\mu\|_{L^\infty}\Big)\\
 + \left(\frac{\be}{\bec}+\frac{\be}{2\ep}\right)\|\log\mu\|_{L^\infty} - \left(1-\frac{\be}{\bec} - \frac{ \be}{2\ep}\right)\int_{\T^\ds}\log(\frac{d\nu}{d\mu})d\nu.
\end{multline}

{We now show that the left-hand side achieves a maximum at some $\nu$.} Let $\nu_n\in \P(\T^\ds)$ be a maximizing, increasing sequence for the left-hand side of \eqref{eq:FEmaxLHS}. Without loss of generality, we may assume that $\be\bar{F}_{\eta}(\nu_0,\mu) - H(\nu_0\vert\mu)\geq 0$. Then if $\be<\frac{\bec}{2}$, we may choose $\ep\in (0,\ds)$ sufficiently small, so that $1-\frac{\be}{\bec}-\frac{\be}{2\ep}>0$ and
\begin{multline}\label{eq:REnun}
\paren*{1-\frac{\be}{\bec} -\frac{\be}{2\ep}}\int_{\R^\ds}\log\paren*{\frac{d\nu_n}{d\mu}} d\nu_n \leq C_{LHLS} + \frac{C_\ep \be \eta^{\ds}}{2} \\
+ \frac{3\be}{2}\Big(C\eta^{\ds}\|\mu\|_{L^\infty}+ \|\g\ast\mu\|_{L^\infty}\Big)  + \left(\frac{\be}{\bec}+\frac{\be}{2\ep}\right)\|\log\mu\|_{L^\infty},
\end{multline}
implying $\nu_n$ is uniformly bounded in $L\log L(\T^\ds)$. The Dunford-Pettis theorem implies that there is a $\nu\in\P_{ac}(\T^\ds)$ such that (up to passing to a subsequence) $\nu_n \rightharpoonup \nu$ as $n\rightarrow\infty$. Since $\g_{(\eta)}$ is continuous {and  the relative entropy is lower semicontinuous with respect to weak convergence of measures, it is straightforward to show that
%\begin{equation}
%\lim_{n\rightarrow\infty} \bar{F}_{\eta}(\nu_n,\mu) = \bar{F}_{\eta}(\nu,\mu).
%\end{equation}
%By t,
%\begin{equation}
%-H(\nu\vert \mu)\geq  -\lim_{n\rightarrow\infty}H(\nu_n \vert \mu).
%\end{equation}
%Hence,
\begin{align}
\sup_{\rho\in\P(\T^\ds)} \paren*{\be\bar{F}_{\eta}(\rho,\mu) - H(\rho\vert\mu)} &= \lim_{n\rightarrow\infty}\Bigg(\be\bar{F}_{\eta}(\nu_n,\mu) -H(\nu_n\vert\mu)\Bigg) \leq \be\bar{F}_{\eta}(\nu,\mu) -  H(\nu\vert\mu).
\end{align}}
This implies that $\nu$ is a maximizer. 

By classical arguments (e.g., see the proof of \cite[Lemma 2.1]{AS2022}),
\begin{equation}\label{eq:minmeasid}
\log\paren*{\frac{d\nu}{d\mu}} -\be\g_{(\eta)}\ast (\nu-\mu) =  -\log\E_{\mu}\Bigg[e^{\be\g_{(\eta)}\ast(\nu-\mu)}\Bigg] \qquad \text{a.e.} \ x\in \supp(\mu).
\end{equation}
Since $\|\log\mu\|_{L^\infty}<\infty$ by assumption, and therefore $\mathrm{essinf}_{x\in\T^\ds}\mu(x)>0$ on $\T^\ds$, the identity \eqref{eq:minmeasid} holds a.e. on $\T^\ds$. Moreover, since $\g_{(\eta)}\ast (\nu-\mu)$ is continuous, \eqref{eq:minmeasid} shows that $\log\frac{d\nu}{d\mu}$ coincides a.e. with a continuous function. In particular, this implies that there are $c_1,c_2>0$ such that $c_1\leq \frac{d\nu}{d\mu}\leq c_2$. By exponentiating both sides of \eqref{eq:minmeasid}, we see that $\nu$ is a solution to
\begin{align}
\nu = \mu \frac{e^{\be\g_{(\eta)}\ast(\nu-\mu)}}{\int_{\T^\ds}e^{\be\g_{(\eta)}\ast(\nu-\mu)}d\mu}.
\end{align}

We now want to prove that $\nu=\mu$, in particular that the maximizer is unique. The proof of this step is similar to that of \cref{prop:ssunq}, except we no longer need to work so hard to establish $\beta$-independent bounds for $\|\g_{(\eta)}\ast\nu\|_{L^\infty}, \int_{\T^\ds}d\nu(\g_{(\eta)}\ast\nu)$, as these are provided by the minimal value of the free energy. Reusing notation, set $\hs \coloneqq \g_{(\eta)}\ast(\nu-\mu)$.
%\begin{align}
%\Dm^{\ds}\hs = \frac{\rho_\be e^{\be\hs}}{\int_{\T^\ds}\rho_\be e^{\be\hs}} - 1, \qquad \text{where} \ \rho_\be \coloneqq e^{-\be\g\ast%\mu}\mu.
%\end{align}
By Jensen's inequality and Fubini-Tonelli,
\begin{align}\label{eq:Jensenhslb}
\int_{\T^\ds}e^{\be\hs}d\mu\geq e^{\int_{\T^\ds}\be\hs d\mu} \geq e^{-\be\|\g_{(\eta)}\ast\mu\|_{L^\infty}\|\nu-\mu\|_{L^1}} \geq e^{-2\be\|\g_{(\eta)}\ast\mu\|_{L^\infty}}.
\end{align} 
Write
\begin{align}
(\nu-\mu)(x) = \mu(x)\Bigg[\frac{e^{\be\hs(x)}}{\int_{\T^\ds}e^{\be\hs}d\mu}-1\Bigg] = \mu(x)\Bigg[\frac{\int_{\T^\ds}[e^{\be\hs(x)} - e^{\be\hs(y)}]d\mu(y)}{\int_{\T^\ds}e^{\be\hs}d\mu}\Bigg].
\end{align}
By the mean value theorem,
\begin{align}
\left|e^{\be\hs(x)} - e^{\be\hs(y)}\right| \leq \be e^{\be\|\hs\|_{L^\infty}} |\hs(x)-\hs(y)| \leq \be e^{\be\|\hs\|_{L^\infty}}(|\hs(x)| + |\hs(y)|).
\end{align}
By Young's and Holder's inequalities,
\begin{align}
%\int_{\T^\ds}|\hs(x)|dx \leq \|\g_{(eta)}\|_{L^1} \|\nu-\mu\|_{L^1}, \\
\int_{\T^\ds}|\hs(y)|d\mu(y) \leq \|\g_{(\eta)}\|_{L^1}\|\mu\|_{L^\infty} \|\nu-\mu\|_{L^1}.
\end{align}
Therefore,
\begin{align}
\|\nu-\mu\|_{L^1} &\leq \frac{1}{\int_{\T^\ds}e^{\be\hs}d\mu} \int_{(\T^\ds)^2}\left|e^{\be\hs(x)} - e^{\be\hs(y)}\right| d\mu(x)d\mu(y) \nn\\
&\leq  \frac{\be e^{\be\|\hs\|_{L^\infty}}}{\int_{\T^\ds}e^{\be\hs}d\mu} \int_{(\T^\ds)^2}(|\hs(x)|+|\hs(y)|)d\mu(x)d\mu(y) \nn\\
&\leq  \frac{2\be e^{\be\|\hs\|_{L^\infty}}}{\int_{\T^\ds}e^{\be\hs}d\mu} \|\g_{(\eta)}\|_{L^1}\|\mu\|_{L^\infty} \|\nu-\mu\|_{L^1}, \label{eq:numuL1pre}
\end{align}
where in obtaining the last line we also use that $\int_{\T^\ds}\mu=1$. Similar to \eqref{eq:getagdiffFennu} and \eqref{eq:getamuLinf}, we have
\begin{align}
\|\hs\|_{L^\infty} &\leq \|(\g-\g_{(\eta)})\ast \mu\|_{L^\infty} + \|(\g-\g_{(\eta)})\ast \nu\|_{L^\infty} + \|\g\ast\mu\|_{L^\infty}+ \|\g\ast\nu\|_{L^\infty} \nn\\
&\leq  C_{\ds}\eta^{\ds}\|\mu\|_{L^\infty} + C_\ep\eta^{\ds}+ \frac2\ep\Big(\int_{\T^\ds}\log(\frac{d\nu}{d\mu})d\nu +\|\log\mu\|_{L^\infty}\Big) + \|\g\ast\mu\|_{L^\infty} + \|e^{\ep\g}\|_{L^1},
\end{align}
for any $\ep\in (0,\ds)$. Since the bound \eqref{eq:REnun} holds with $\nu_n$ replaced by $\nu$, we may apply it to the relative entropy term on the preceding right-hand side to obtain
\begin{multline}\label{eq:hsLinfinf}
\|\hs\|_{L^\infty} \leq \frac{2}{\ep}\paren*{1-\frac{\be}{\bec} -\frac{\be}{2\ep}}^{-1}\Bigg(C_{LHLS} + \frac{C_\ep \be \eta^{\ds}}{2} + \frac{3\be}{2}\Big(C\eta^{\ds}\|\mu\|_{L^\infty}+ \|\g\ast\mu\|_{L^\infty}\Big) \\
 + \left(\frac{\be}{\bec}+\frac{\be}{2\ep}\right)\|\log\mu\|_{L^\infty}\Bigg) + C_{\ds}\eta^{\ds}\|\mu\|_{L^\infty} + C_\ep\eta^{\ds}+ \frac2\ep\|\log\mu\|_{L^\infty} + \|\g\ast\mu\|_{L^\infty} + \|e^{\ep\g}\|_{L^1}.
\end{multline}
Recalling from above that
\begin{align}
\|\g_{(\eta)}\|_{L^1} \leq C_{\ds}\eta^\ds + \|\g\|_{L^1}
\end{align}
and applying the bounds \eqref{eq:Jensenhslb}, \eqref{eq:hsLinfinf} to \eqref{eq:numuL1pre}, we arrive at
\begin{multline}
\|\nu-\mu\|_{L^1} \leq \|\nu-\mu\|_{L^1}2\be\|\mu\|_{L^\infty}(C_{\ds}\eta^\ds + \|\g\|_{L^1})\exp\Bigg[\be\Bigg(\frac{2}{\ep}\paren*{1-\frac{\be}{\bec} -\frac{\be}{2\ep}}^{-1}\Bigg(C_{LHLS} + \frac{C_\ep \be \eta^{\ds}}{2}\\
 + \frac{3\be}{2}\Big(C\eta^{\ds}\|\mu\|_{L^\infty}+ \|\g\ast\mu\|_{L^\infty}\Big)  + \left(\frac{\be}{\bec}+\frac{\be}{2\ep}\right)\|\log\mu\|_{L^\infty}\Bigg) + C_{\ds}\eta^{\ds}\|\mu\|_{L^\infty} + C_\ep\eta^{\ds}+ \frac2\ep\|\log\mu\|_{L^\infty}\\
  + \|\g\ast\mu\|_{L^\infty} + \|e^{\ep\g}\|_{L^1}\Bigg)\Bigg]e^{-2\be\left(\|\g\ast\mu\|_{L^\infty} - C_{\ds}\eta^\ds\|\mu\|_{L^\infty}\right)}.
\end{multline}
Fixing some choice of $\ep$, let $\be_{0}<\frac{\bec}{2}$ be such that
\begin{align}
2\be_0\|\g\|_{L^1}e^{\frac{2\be_0}{\ep}\left(1-\frac{\be_0}{\bec}-\frac{\be_0}{2\ep}\right)^{-1}C_{LHLS} + \be_{0}\|e^{\ep\g}\|_{L^1}} = 1.
\end{align}
Then for $\be<\be_{0}$, we may find $\d_{\be},\eta_\be>0$ such that for $\|\log\mu\|_{L^\infty}\leq \d_\be$ and $\eta\leq \eta_\be$, we have that there exists a $c\in [0,1)$ such that
\begin{align}
\|\nu-\mu\|_{L^1} \leq  c\|\nu-\mu\|_{L^1},
\end{align}
which implies that $\nu=\mu$. This completes the proof of \cref{lem:Ivan}.

{\begin{remark}
If $\ds=2$, we may take advantage of the fact that $C_{LHLS}$ defined \eqref{eq:CLHLSdef} equals $0$. This is a consequence of the uniqueness of stationary solutions to the Kirkwood-Monroe equation in the critical case $\be=\bec$, which has been discussed above. 
\end{remark}

\begin{remark}\label{rem:KWequiv}
If $\ds=2$, then equation \eqref{eq:minmeasid} with $\eta=0$ is equivalent to the Kazhdan-Warner equation \cite{KW1974}. Indeed, for reference density $\mu$ and replacing $\g$ by $\frac1\cdd\g$, consider a solution $\nu$ of
\begin{align}
\nu = \frac{\mu}{M_\nu} e^{\beta \g\ast (\nu-\mu)}, \qquad M_\nu \coloneqq \int_{\T^2}e^{\beta\g\ast(\nu-\mu)}d\mu,
\end{align}
Setting $\mathsf{u} \coloneqq \beta\g\ast \nu$, we see that $\mathsf{u}$ is a solution of
\begin{align}
-\Delta\mathsf{u} = \beta\left(K\frac{e^{\mathsf{u}}}{M_\mathsf{u}}-1\right), \qquad M_{\mathsf{u}} \coloneqq \int_{\T^2}K e^{\mathsf{u}}dx,
\end{align}
where $K \coloneqq \mu e^{-\beta\g\ast\mu}$. Note that $\frac{K e^\mathsf{u}}{M_\mathsf{u}} = e^{\mathsf{u}+\log(K) - \log(M_\mathsf{u})}$. Setting $\mathsf{w} \coloneqq \mathsf{u}+\log(K) - \log(M_\mathsf{u})$, one finds
\begin{align}\label{eq:wKWeqn}
\Delta\mathsf{w} + \beta e^{\mathsf{w}} = \beta+\Delta\log(K).
\end{align}

In \cite{GM2019}, the symmetry of solutions of \eqref{eq:wKWeqn} has been studied under the condition $\beta+\Delta\log(K) \geq 0$ with $\be\leq 4$. This condition may be ensured when the reference measure $\mu$ is not too far from being uniform. More precisely, let $\varphi \in C^2$, $\ep>0$, and write $\mu = 1+\ep\varphi$. By chain rule,
\begin{align}
\Delta\log(K) = \Delta\left(\log(\mu) - \beta\g\ast\mu\right) &= \div\left(\frac{\nabla\mu}{\mu}-\beta\nabla\g\ast\mu \right) =\frac{\Delta\mu}{\mu} - \frac{|\nabla\mu|^2}{\mu^2} + \beta\mu,
\end{align}
which implies that
\begin{align}
\beta + \Delta\log(K) = \be(1+\mu) + \frac{\Delta\mu}{\mu} - \frac{|\nabla\mu|^2}{\mu^2}  &= 2\beta + \ep\be\varphi + \ep \frac{\Delta\varphi}{1+\ep\varphi} - \ep^2\frac{|\nabla \varphi|^2}{(1+\ep\varphi)^2}.
\end{align}
Fixing $\beta$ and $\varphi$, we can always take $\ep>0$ sufficiently small so that the right-hand side is $\geq 0$. Unfortunately, only uniqueness for the case $K\equiv 1$ is treated in \cite{GM2019}---nor seemingly elsewhere in the literature---under the condition $\beta+\Delta\log(K) \geq 0$ for nonconstant $K$, and it is unclear to us how to generalize the proof of uniqueness for constant $K$.\footnote{The condition $K\equiv 1$ implies that $\mu=\muu$. Indeed, since $\mu$ is a probability density, we have $\int_{\T^\ds}e^{\be\g\ast\mu}=1$. By Jensen's inequality, $\int_{\T^\ds}e^{\be\g\ast\mu}\geq e^{\be\int_{\T^\ds}\g\ast\mu} = 1$, since $\int_{\T^\ds}\g=0$. This means we have an equality case for Jensen's inequality. Since $e^z$ is not affine, this implies that $\g\ast\mu$ is constant a.e. Hence, $\Dm^{\ds}\g\ast\mu = \cdd(\mu-1) = 0$, proving the claim.}
\end{remark}}

\subsubsection{Conclusion of proof}\label{sssec:mLHLSconc}
We now conclude the proof of \cref{thm:mainMLHLS}.

 Recalling our starting point \eqref{eq:gr1F} and using \eqref{eq:ZNretabetapre},
\begin{multline}
\E_{f_N}\left[\Fr_N(\XN,\mu)\right] \leq \frac{1}{\be}H_N(f_N \vert \mu^{\otimes N}) + {CN\be\|\mu\|_{L^\infty}\eta^\ds}\\
+\log\mathbb{E}_{\mu^{\otimes N}}\left[\exp\Bigg(N\be\Fr_{N,\eta}({\ux_N},\mu) + \frac{\be}{2N}\sum_{\substack{1\leq i\neq j\leq N \\ |x_i-x_j|\leq \eta }}\log\frac{\eta}{|x_i-x_j|}\Bigg)\right].
\end{multline}
Next, applying \cref{eq:expfacbnd} to the second term on the preceding right-hand side, we find
\begin{multline}
\E_{f_N}\left[\Fr_N(\XN,\mu)\right] \leq \frac{1}{\be}H_N(f_N \vert \mu^{\otimes N}) + \frac{\be\log K_{N,\be',\eta}(\mu)}{\be' N} + {C_0\eta\be\paren*{\|\nabla\mu\|_{L^\infty} + \frac{1}{\ep} }+C\be\|\mu\|_{L^\infty}\eta^\ds} \\
+{\left(1-\frac{\be}{\be'}\right)}\log\paren*{1+C_24^{\frac{\be\be'}{\be'-\be}} \|\mu\|_{L^\infty}\ep^\ds +  C_2\|\mu\|_{L^\infty}\eta^\ds\left(\frac{4^\be\|\mu\|_{L^\infty}}{(\inf\mu)C_1(\bec-\be)}\right)^{\frac{\be'}{\be'-\be}}}
\end{multline}
for any $\be'\in (\be,\bec)$. Applying the estimate \eqref{eq:logZNrbeta'bnd} to the $\log K_{N,\be',\eta}(\mu)$ term in the first line above, we obtain
\begin{multline}\label{eq:MHLSpresub}
\E_{f_N}\left[\Fr_N(\XN,\mu)\right] \leq \frac{1}{\be}H_N(f_N \vert \mu^{\otimes N}) +\frac{\be}{\be'}\left(\frac{C}{M^{1/\ds}\eta}  + \frac{C\|\nabla\log\mu\|_{L^\infty}}{M^{1/\ds}} + \frac{C(M+\frac{1}{2})\log N}{N}\right) \\
+{\left(1-\frac{\be}{\be'}\right)}\log\paren*{1+C_24^{\frac{\be\be'}{\be'-\be}} \|\mu\|_{L^\infty}\ep^\ds +  C_2\|\mu\|_{L^\infty}\eta^\ds\left(\frac{4^\be\|\mu\|_{L^\infty}}{(\inf\mu)C_1(\bec-\be)}\right)^{\frac{\be'}{\be'-\be}}}\\
+{C_0\eta\be\paren*{\|\nabla\mu\|_{L^\infty} + \frac{1}{\ep} }+C\be\|\mu\|_{L^\infty}\eta^\ds}.
\end{multline}

Optimize the choices of $\eta,\ep, M$ as follows. Setting
\begin{equation}
4^{\ds+1}\frac{\eta}{\ep} = \ep^\ds \Longrightarrow \ep = 4\eta^{\frac{1}{\ds+1}} \Longrightarrow 4^{\ds+1}\frac{\eta}{\ep} = \ep^\ds  = 4^\ds\eta^{\frac{\ds}{\ds+1}}.
\end{equation}
Now setting
\begin{equation}
\frac{1}{\eta M^{1/\ds}} = \eta^{\frac{\ds}{\ds+1}} \Longrightarrow \eta = M^{-\frac{\ds+1}{\ds(2\ds+1)}} \Longrightarrow \frac{1}{\eta M^{1/\ds}} = \eta^{\frac{\ds}{\ds+1}} = M^{-\frac{1}{2\ds+1}}.
\end{equation}
Finally, setting
\begin{equation}
\frac{M}{N} = M^{-\frac{1}{2\ds+1}} \Longrightarrow M = N^{\frac{2\ds+1}{2(\ds+1)}} \Longrightarrow \frac{M}{N} =  M^{-\frac{1}{2\ds+1}} =  N^{-\frac{1}{2(\ds+1)}}.
\end{equation}
We see that the above choices for $\eta,\ep,M$, satisfy all the constraints imposed above. Substituting these choices into the right-hand side of \eqref{eq:MHLSpresub} then Taylor expanding $\log(1+z)$ and performing a little algebra, we see that the proof of \cref{thm:mainMLHLS} is complete.

\section{The modulated free energy and mean-field convergence}\label{sec:MFE}
In this section, we prove \cref{thm:mainUT}. Recall from the introduction the dissipation inequality for the modulated free energy:
\begin{multline}\label{gronwalla}
E_N(f_N^t, \mu^t) -  E_N(f_N^0,\mu^0) \leq -\frac{1}{\beta^2 N} \int_0^t \int_{(\T^\ds)^N}   \left| \nab \sqrt{ \frac{f_N^\tau}{\mathbb{Q}_{N,\beta} (\mu^\tau) } } \right|^2 d \mathbb{Q}_{N,\beta}(\mu^\tau)d\tau\\
 +\frac{1}{2}\int_0^t\E_{f_N^\tau}\Bigg[\int_{(\T^\ds)^2\setminus\triangle} (u^\tau(x)-u^\tau(y))\cdot \nabla\g(x-y) d\left(\frac1N\sum_{i=1}^N\d_{x_i}- \mu^\tau\right)^{\otimes 2}(x,y)\Bigg]d\tau,
\end{multline}
where $u^t \coloneqq \nabla\log\mu^t - \be\nabla \g \ast \mu^t$. The relative Fisher information term on the right-hand side is nonpositive, so we discard it.\footnote{It would be interesting to exploit this term to via a modulated LSI as in \cite{RS2023lsi}, but this appears difficult. Even for $\ds=1$, a LSI for the Gibbs measure $\mathbb{Q}_{N,\be}(\muu)$ does not seem known.} The second term is handled by means of a functional inequality.

\subsection{Functional inequality}
The following lemma is implicit in \cite{BJW2020}---and is really a consequence of the earlier work \cite{JW2018}. We sketch its proof for the reader's convenience.

\begin{lemma}\label{lem:logFI}
Let $\g \in C^1(\T^\ds\setminus\{0\})$ and satisfy the bound $\left|\nabla\g(x-y)\right| \leq \frac{C}{|x-y|}$ for some constant $C>0$. Then for any $f_N \in \P_{ac}((\T^\ds)^N)$, $\mu\in \P_{ac}(\T^\ds)$, and Lipschitz vector field $v$,
\begin{multline}\label{eq:logFI}
\left|\E_{f_N}\Bigg[\int_{(\T^\ds)^2\setminus\triangle} (v(x)-v(y))\cdot \nabla\g(x-y) d\left(\frac1N\sum_{i=1}^N\d_{x_i}- \mu\right)^{\otimes 2}(x,y)\Bigg]\right| \\
\leq {C\sqrt{2\mathsf{C}_0}\|\nabla v\|_{L^\infty}}\left(H_N(f_N \vert \mu^{\otimes N}) + \frac{\log 4}{N}\right).
\end{multline}
where $\mathsf{C}_0>0$ depends only on $\ds$.
\end{lemma}
\begin{proof}
Define the function $\Phi:(\T^\ds)^N\rightarrow\R$,
\begin{equation}
\Phi(\ux_N) \coloneqq \int_{(\T^\ds)^2\setminus\triangle} (v(x)-v(y))\cdot \nabla\g(x-y) d\left(\frac1N\sum_{i=1}^N\d_{x_i} - \mu\right)^{\otimes 2}(x,y).
\end{equation}
By the Donsker-Varadhan lemma, for any $\eta>0$,
\begin{equation}\label{eq:preetaBJW}
\E_{f_N}[\Phi] \leq  \frac{1}{\eta}\paren*{H_N(f_N\vert \mu^{\otimes N}) + \frac1N\log \E_{\mu^{\otimes N}}[e^{\eta N\Phi}] }  .
\end{equation}
Introduce the function $\phi: (\T^\ds)^2 \rightarrow \R$,
\begin{multline}
\phi(x,y) \coloneqq (v(x)-v(y))\cdot \nabla\g(x-y) - \int_{\T^\ds}(v(x)-v(z))\cdot\nabla\g(x-z)d\mu(z) \\
- \int_{\T^\ds}(v(z)-v(y))\cdot\nabla\g(z-y)d\mu(z) + \int_{(\T^\ds)^2}(v(w)-v(z))\cdot\nabla\g(w-z)d\mu^{\otimes 2}(w,z).
\end{multline}
We adopt the convention that $\phi(x,x)\coloneqq 0$. With this notation, we see that
\begin{equation}
\Phi(\ux_N) = \frac{1}{N^2}\sum_{1\leq i,j\leq N} \phi(x_i,x_j) = \E_{\mu_{\ux_N}^{\otimes 2}}[\phi].
\end{equation}
By direct computation, one checks that $\int_{\T^\ds}\phi(x,y)d\mu(y) = \int_{\T^\ds}\phi(x,y)d\mu(x)=0$. Furthermore, by the mean-value theorem and our assumption on $\nabla\g$,
\begin{align}
\forall x\neq y\in\T^\ds,\qquad \left|\phi(x,y)\right| &\leq \|\nabla v\|_{L^\infty} |x-y||\nabla\g(x-y)| \leq C\|\nabla v\|_{L^\infty}. \label{eq:PhiMVT}
\end{align}
According to \cite[Theorem 4]{JW2018} (see also \cite[Section 5]{LLN2020} for a simpler proof), there is a universal constant $\mathsf{C}_0>0$, such that if $\sqrt{\mathsf{C}_0}\eta\|\phi\|_{L^\infty}<1$, then
\begin{align}
\log\mathbb{E}_{\mu^{\otimes N}}\left[e^{\eta N\Phi}\right] \leq \log\left(\frac{2}{1-\mathsf{C}_0\eta^2\|\phi\|_{L^\infty}^2}\right).
\end{align}
Replacing $\Phi$ by $-\Phi$ and repeating the above reasoning, we then find that
\begin{align}
\left|\mathbb{E}_{f_N}\left[\Phi\right]\right| \leq \frac{1}{\eta}H_N(f_N\vert\mu^{\otimes N}) + \frac{1}{\eta N}\log\left(\frac{2}{1-\mathsf{C}_0\eta^2\|\phi\|_{L^\infty}^2}\right).
\end{align}
%\begin{align}
%E_N(f_N, \mu) \geq \left(\frac{1}{\beta}-\frac1\eta\right)H_N(f_N\vert\mu^{\otimes N}) -  \frac{1}{\eta N}\log\left(\frac{2}{1-C_0\|\phi_\eta\|_{L^\infty}^2}\right).
%\end{align}
Choosing $1/\eta = {C\sqrt{2\mathsf{C}_0}\|\nabla v\|_{L^\infty}}$, we see that \eqref{eq:logFI} holds.
\end{proof}

\subsection{Gr\"{o}nwall argument}
Applying the functional inequality \eqref{eq:logFI} to the second term on the right-hand side of \eqref{gronwalla} pointwise in $\tau$, we find
\begin{multline}\label{eq:GronFI}
\frac{1}{2}\int_0^t\E_{f_N^\tau}\Bigg[\int_{(\T^\ds)^2\setminus\triangle} (u^\tau(x)-u^\tau(y))\cdot \nabla\g(x-y) d\left(\frac1N\sum_{i=1}^N\d_{x_i}- \mu^\tau\right)^{\otimes 2}(x,y)\Bigg]d\tau \\
\leq \frac{{C\sqrt{2\mathsf{C}_0}}}{2}\int_0^t \|\nabla u^\tau\|_{L^\infty}\left(H_N(f_N^\tau \vert (\mu^\tau)^{\otimes N}) + \frac{\log 4}{N}\right)d\tau.
\end{multline}

Fix $\be<\be' < {\bei}$, where {$\bei$} is the threshold from \cref{thm:mainMLHLS}. The reason for the two inverse temperatures will become clear momentarily. Let $\delta_{\be'}$ also be as in \cref{thm:mainMLHLS}, and we suppose that we have an initial datum $\mu^0 \in W^{2,\infty}$. By \cref{rem:W2infglob}, we then have that
\begin{align}\label{eq:MFElogmuW2inf}
\forall t\geq 0, \qquad \|\log\mu^t\|_{W^{2,\infty}} \leq \W(\beta, \|\log\mu^0\|_{W^{2,\infty}})e^{-\mathsf{c}(\beta, \|\log\mu^0\|_{W^{2,\infty}}) t}
\end{align}
for some continuous, nonincreasing functions $\W, \mathsf{c} :[0,\infty)^2\rightarrow [0,\infty)$, which vanish if any of their arguments are zero. Supposing that $\mu^0$ is such that
\begin{align}
\W(\beta, \|\log\mu^0\|_{W^{2,\infty}})\leq \delta_{\be'},
\end{align}
 we then have that $\|\log\mu^t\|_{W^{2,\infty}}\leq \delta_{\be'}$ for every $t\geq 0$.

As a substitute for the modulated free energy,  introduce the auxiliary quantity
\begin{align}
\mathcal{E}_N^t \coloneqq  E_N(f_N^t,\mu^t) +\mathbf{C}(\beta',\sup_{\tau\geq 0}\|\log\mu^\tau\|_{W^{1,\infty}}) N^{-\ga}
\end{align}
where $\mathbf{C}(\cdot), \ga$ are just as in the error term in \eqref{eq:mainMLHLS1}. This additive error ensures that $\mathcal{E}_N^t\geq 0$, which allows to perform a Gr\"onwall argument on this quantity. Returning to our starting integral identity \eqref{gronwalla}, using triangle inequality and \eqref{eq:GronFI}, we have
\begin{align}\label{eq:HNpreins}
\mathcal{E}_N^t \leq \mathcal{E}_N^0 + \frac{{C\sqrt{2\mathsf{C}_0}}}{2}\int_0^t \|\nabla u^\tau\|_{L^\infty}\left(H_N(f_N^\tau \vert (\mu^\tau)^{\otimes N}) + \frac{\log 4}{N}\right)d\tau.
\end{align}
By \cref{thm:mainMLHLS},
\begin{align}
E_N(f_N^t, \mu^t) &\geq \left(\frac1\be - \frac{1}{\be'}\right)H_N(f_N^t \vert (\mu^t)^{\otimes N}) - \mathbf{C}(\beta',\|\log\mu^t\|_{W^{1,\infty}}) N^{-\ga} \nn\\
&\geq\left(\frac1\be - \frac{1}{\be'}\right)H_N(f_N^t \vert (\mu^t)^{\otimes N}) - \mathbf{C}(\beta',\sup_{\tau\geq 0}\|\log\mu^\tau\|_{W^{1,\infty}}) N^{-\ga},
\end{align}
which implies that
\begin{align}\label{eq:HNEcN}
H_N(f_N^t \vert (\mu^t)^{\otimes N}) \leq \frac{\be\be'}{\be'-\be}\Big(E_N(f_N^t, \mu^t) + \mathbf{C}(\beta',\sup_{\tau\geq 0}\|\log\mu^\tau\|_{W^{1,\infty}}) N^{-\ga}\Big) = \frac{\be\be'}{\be'-\be}\mathcal{E}_N^t.
\end{align}
Inserting this bound into the right-hand side of \eqref{eq:HNpreins}, we find that
\begin{equation}
\mathcal{E}_N^t \leq \mathcal{E}_N^0 + \frac{{C\sqrt{2\mathsf{C}_0}}}{2}\frac{\be\be'}{\be'-\be}\int_0^t \|\nabla u^\tau\|_{L^\infty}\mathcal{E}_{N}^\tau d\tau.
\end{equation}
By the Gr\"onwall-Belllman lemma,
\begin{equation}\label{eq:EcGBfin}
\mathcal{E}_N^t \leq \mathcal{E}_N^0 e^{\frac{{C\sqrt{2\mathsf{C}_0}}}{2}\frac{\be\be'}{\be'-\be}\int_0^t \|\nabla u^\tau\|_{L^\infty}d\tau}.
\end{equation}
Applying \eqref{eq:HNEcN} to $\mathcal{E}_N^t$ on the left-hand side, then unpacking the definition of $\mathcal{E}_{N}^0$ on the right-hand side, we obtain from \eqref{eq:EcGBfin} that
\begin{align}
&\frac{\be'-\be}{\be\be'}H_N(f_N^t \vert (\mu^t)^{\otimes N}) \nn\\
&\leq \Big(E_N(f_N^0,\mu^0) +\mathbf{C}(\beta',\sup_{\tau\geq 0}\|\log\mu^\tau\|_{W^{1,\infty}}) N^{-\ga}\Big)e^{\frac{{C\sqrt{2\mathsf{C}_0}}}{2}\frac{\be\be'}{\be'-\be}\int_0^t \|\nabla u^\tau\|_{L^\infty}d\tau} \nn\\
&\leq \Big(E_N(f_N^0,\mu^0) +\mathbf{C}(\be',\W(\beta, \|\log\mu^0\|_{W^{1,\infty}}))N^{-\ga}\Big)e^{\frac{{C\sqrt{2\mathsf{C}_0}}}{2}\frac{\be\be'}{\be'-\be}\int_0^t \|\nabla u^\tau\|_{L^\infty}d\tau}. \label{eq:HNGBfin}
\end{align}

Recalling from above the definition of $u$, the triangle inequality gives
\begin{align}
\|\nabla u^\tau\|_{L^\infty} \leq \frac1\be\|\nabla^{\otimes 2}\log\mu^\tau\|_{L^\infty} + \|\nabla^{\otimes 2}\g\ast\mu^\tau\|_{L^\infty}.
\end{align}
The first term on the right-hand side is handled by \eqref{eq:MFElogmuW2inf}. For the second term, note that since $\g\in L^1$, we may crudely estimate
\begin{align}
\|\nabla^{\otimes 2}\g\ast\mu^\tau\|_{L^\infty} \lesssim \|\nabla^{\otimes 2}\mu^\tau\|_{L^\infty} &\lesssim \Big(\|\nabla^{\otimes 2}\log \mu^\tau\|_{L^\infty} + \|\nabla\log\mu^\tau\|_{L^\infty}^2\Big) e^{\|\log\mu^\tau\|_{L^\infty}} \nn\\
&\leq \tl{\W}(\beta, \|\log\mu^0\|_{W^{2,\infty}})e^{-\mathsf{c}(\beta, \|\log\mu^0\|_{W^{2,\infty}}) \tau},
\end{align}
where the penultimate inequality follows from writing $\mu = e^{\log \mu}$ and using product rule. Here, $\tl{\W}:[0,\infty)^2\rightarrow [0,\infty)$ is some other continuous, increasing function of its arguments, vanishing if any of them are zero. Dropping the $\tilde{}$ superscript, we therefore find that
\begin{align}
\int_0^\infty \|\nabla u^\tau\|_{L^\infty}d\tau \leq {\W}(\beta, \|\log\mu^0\|_{W^{2,\infty}})\int_0^\infty e^{-\mathsf{c}(\beta, \|\log\mu^0\|_{W^{2,\infty}}) \tau}d\tau \leq \frac{{\W}(\beta, \|\log\mu^0\|_{W^{2,\infty}})}{\mathsf{c}(\beta, \|\log\mu^0\|_{W^{2,\infty}}) }.
\end{align}
Substituting this bound into the right-hand side of \eqref{eq:HNGBfin} completes the proof of \cref{thm:mainUT}. 

\begin{comment}
Since
\begin{align}
\|\nabla\g\ast u^\tau\|_{L^\infty} \lesssim \begin{cases} \|\nabla\mu^\tau\|_{C^{0+}, & {\ds=1} \\ \|\mu^\tau\|_{\dot{C}^{0+}}, & {\ds =2 } \\  \|\mu^\tau-1\|_{L^?}, & {\ds\geq 3} \end{cases}
\end{align}
it follows from triangle inequality that in all cases of $\ds$,
\begin{align}
\|\nabla u^\tau\|_{L^\infty} \lesssim 
\end{align}
\end{comment}

\subsection{Counterexample to uniform-in-time propagation of chaos}
We now demonstrate the impossibility of uniform-in-time propagation of chaos if $\be>\bels$ using our nonlinear instability result \cref{prop:NLinstab}.

Let $\mu_\ep^0 \coloneqq 1+2\ep\cos(2\pi k\cdot x)$, for $|k|=1$, and let $\mu_\ep^t$ be the unique solution to equation \eqref{eq:lima} with initial datum $\mu_\ep^0$. By \cref{prop:NLinstab}, we know that provided $\ep>0$ is sufficiently small, there is a $t_\ep = O(\log(1/\ep))$ such that $\|\mu^{t_\ep}-\muu\|_{L^1}\geq \frac12$. Let $f_N^t$ be an entropy solution to \eqref{eq:Lioua} with initial datum $(\mu_\ep^0)^{\otimes N}$. By triangle inequality,
\begin{align}\label{eq:fN1muepmuu}
\|f_{N;1}^{t_\ep}-\mu_\ep^{t_\ep}\|_{L^1} + \|f_{N;1}^{t_\ep}-\muu\|_{L^1}  \geq \|\mu_{\ep}^{t_\ep} - \muu\|_{L^1} \geq \frac12.
\end{align}
On the other hand, $H_{N}(f_N^0 \vert (\mu_\ep^0)^{\otimes N}) = 0$, and similar to the computation in the proof of \cref{prop:ssnounq}, we may estimate
\begin{align}
H_N(f_N^0 \vert \muu^{\otimes N}) = \int_{\T^\ds}\log(\mu_\ep^0)d\mu_\ep^0 \leq \frac{\ep^2\paren*{1+8\ep} }{\be}.
\end{align}
Thus, the left-hand side of \eqref{eq:fN1muepmuu} cannot be bounded solely in terms of $N$, $H_N(f_N^0,\muu^{\otimes N})$. 

%Suppose that \eqref{eq:EntUT} holds. Then by subadditivity of relative entropy and Pinsker's inequality,
%\begin{align}\label{eq:Pindis}
%\|f_{N;1}^t-\mu^t\|_{L^1}^2 \leq C_1\paren*{o_N(1) + H_N(f_N^0 \vert (\mu^0)^{\otimes N})}.
%\end{align}
% Combining by the reverse triangle inequality with \eqref{eq:Pindis}, we find
%\begin{align}\label{eq:fN1muWTC}
%\|f_{N;1}^{t_\ep}-\muu\|_{L^1} \geq \frac12 - \sqrt{C\paren*{N^{-\al} + H_N(f_N^0\vert (\mu_\ep^0)^{\otimes N})}} = \frac12 - \sqrt{CN^{-\al}} > \frac14,
%\end{align}
%provided $N$ is taken sufficiently large. Using the relation \eqref{eq:Pindis} again with $\mu^t=\muu$, we also have
%\begin{align}\label{eq:fN1mupresub}
%\|f_{N;1}^{t_\ep} - \muu\|_{L^1} \leq \sqrt{C\paren*{N^{-\al} + H_N(f_N^0\vert \muu^{\otimes N})}}.
%\end{align}

%Inserting this estimate into the right-hand side of \eqref{eq:fN1mupresub} and taking $N$ sufficiently large and $\ep>0$ sufficiently small, we obtain a contradiction to \eqref{eq:fN1muWTC}.

\bibliographystyle{alpha}
\bibliography{../../MASTER}
\end{document}